\documentclass[11pt,reqno]{amsart}
\usepackage[utf8]{inputenc}
\usepackage{fullpage}
\usepackage{amsmath,amsthm,amssymb,amsfonts,amstext, mathtools,thmtools,thm-restate,pinlabel}
\usepackage[osf]{newpxtext} 

\usepackage{graphicx}
\usepackage{xcolor}
\usepackage{url}
\usepackage{verbatim}
\usepackage{enumerate}
\usepackage{comment}
\usepackage{hyperref}
\usepackage{scrextend}
\usepackage{subcaption}
\usepackage{caption}
\usepackage{subcaption}
\usepackage{tikz}
\usepackage{pgfplots}
\pgfplotsset{compat=1.15}
\usepackage{mathrsfs}
\usetikzlibrary{arrows}
\pgfplotsset{every axis/.append style={
                    label style={font=\tiny},
                    tick label style={font=\tiny}  
                    }}

\newtheorem{thm}{Theorem}
\numberwithin{thm}{section}

\newtheorem{cor}[thm]{Corollary}
\newtheorem{lem}[thm]{Lemma}

\newtheorem{prop}[thm]{Proposition}
\theoremstyle{definition}

\newtheorem{rem}[thm]{Remark}

\newcommand{\R}{\mathbb{R}}

\newcommand{\Z}{\mathbb{Z}}

\newcommand{\cvec}[2]{\begin{pmatrix} #1 \\ #2 \end{pmatrix}}

\newcommand{\slr}{{\rm SL}_2(\R)}
\newcommand{\pslr}{{\rm PSL}_2(\R)}

\newcommand{\glr}{{\rm GL}_2(\R)}

\newcommand{\slz}{{\rm SL}_2(\Z)}

\DeclarePairedDelimiter\floor{\lfloor}{\rfloor}

\begin{document}


\title{Slope Gap Distribution of Saddle Connections on the $2n$-gon}

\author[J.~Berman]{Jonah Berman}
\address{Department of Mathematics \\
Yale University\\
10 Hillhouse Avenue\\
New Haven, CT 06520}
\email{\href{mailto:jonah.berman@yale.edu}{jonah.berman@yale.edu}}

\author[T.~McAdam]{Taylor McAdam}
\address{Department of Mathematics \\
Yale University\\
10 Hillhouse Avenue\\
New Haven, CT 06520}
\email{\href{mailto:taylor.mcadam@yale.edu}{taylor.mcadam@yale.edu}}

\author[A.~Miller-Murthy]{Ananth Miller-Murthy}
\address{Department of Mathematics \\
Yale University\\
10 Hillhouse Avenue\\
New Haven, CT 06520}
\email{\href{mailto:ananth.miller-murthy@yale.edu}{ananth.miller-murthy@yale.edu}}

\author[C.~Uyanik]{Caglar Uyanik}
\address{Department of Mathematics \\
University of Wisconsin, Madison  \\
480 Lincoln Drive\\
Madison, WI 53706}
\email{\href{mailto:caglar@math.wisc.edu}{caglar@math.wisc.edu}}

\author[H.~Wan]{Hamilton Wan}
\address{Department of Mathematics \\
Yale University\\
10 Hillhouse Avenue\\
New Haven, CT 06520}
\email{\href{mailto:hamilton.wan@yale.edu}{hamilton.wan@yale.edu}}

\begin{abstract} 
We explicitly compute the limiting slope gap distribution for saddle connections on any $2n$-gon for $n\geq 3$. Our calculations show that the slope gap distribution for a translation surface is not always unimodal, answering a question of Athreya. 
We also give linear lower and upper bounds for number of non-differentiability points as $n$ grows. The latter result exhibits the first example of a non-trivial bound on an infinite family of translation surfaces and answers a question by Kumanduri-Wang. 
\end{abstract}

\maketitle

\section{Introduction}

Consider, for $n\ge3$, a regular $2n$-gon $O_{2n}$ in the plane $\mathbb{R}^2\cong\mathbb{C}$ where the opposite sides of $O_{2n}$ are glued by Euclidean translations. The resulting surface, also denoted by $O_{2n}$, is topologically a closed surface. If $n=2k$, then $O_{2n}$ is a surface of genus $k$ and has a single cone point of angle $2\pi(n-1)$. If $n=2k+1$, then $O_{2n}$ is a surface of genus $k$, with two cone points where the cone angles are both $\pi(n-1)$ \cite{SU}. For example, gluing the opposite sides of the regular decagon yields a genus $2$ surface with two cone points (marked with black and red), where at each of the cone points the total angle is $4\pi$. See section \ref{translationsurfaces} for details. 

\bigskip
\begin{figure}[h!]
\centering
\begin{tikzpicture}[scale=1.5, rotate=0]

\foreach \pos in {0,5} {\draw[red] ({cos(36*\pos)}, {sin(36*\pos)})--({cos(36+36*\pos)},{sin (36+36*\pos)});}
\foreach \pos in {1,6} {\draw[green] ({cos(36*\pos)}, {sin(36*\pos)})--({cos(36+36*\pos)},{sin (36+36*\pos)});}
\foreach \pos in {2,7} {\draw[orange] ({cos(36*\pos)}, {sin(36*\pos)})--({cos(36+36*\pos)},{sin (36+36*\pos)});}
\foreach \pos in {3,8} {\draw[blue] ({cos(36*\pos)}, {sin(36*\pos)})--({cos(36+36*\pos)},{sin (36+36*\pos)});}
\foreach \pos in {4,9} {\draw[purple] ({cos(36*\pos)}, {sin(36*\pos)})--({cos(36+36*\pos)},{sin (36+36*\pos)});}

\node at ({cos(36*1)},{sin(36*1)}) [circle,color=black,fill,inner sep=0.6pt]{}; 
\node at ({cos(36*7)},{sin(36*7)}) [circle,color=black,fill,inner sep=0.6pt]{}; 
\node at ({cos(36*3)},{sin(36*3)}) [circle,color=black,fill,inner sep=0.6pt]{}; 
\node at ({cos(36*9)},{sin(36*9)}) [circle,color=black,fill,inner sep=0.6pt]{}; 
\node at ({cos(36*5)},{sin(36*5)}) [circle,color=black,fill,inner sep=0.6pt]{}; 
\node at ({cos(36*2)},{sin(36*2)}) [circle,color=red,fill,inner sep=0.6pt]{}; 
\node at ({cos(36*4)},{sin(36*4)}) [circle,color=red,fill,inner sep=0.6pt]{}; 
\node at ({cos(36*6)},{sin(36*6)}) [circle,color=red,fill,inner sep=0.6pt]{}; 
\node at ({cos(36*8)},{sin(36*8)}) [circle,color=red,fill,inner sep=0.6pt]{}; 
\node at ({cos(36*10)},{sin(36*10)}) [circle,color=red,fill,inner sep=0.6pt]{}; 
\end{tikzpicture}
\caption{The regular decagon, $O_{10}$. Black vertices are identified together and red vertices are identified together.}
\end{figure}
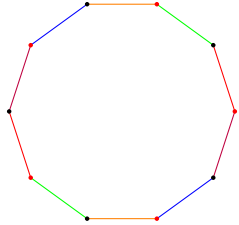

Outside of the finitely many cone points, the surface $O_{2n}$ is locally Euclidean and hence one can talk about straight lines. A \emph{saddle connection} $\gamma$ on $O_{2n}$ is a straight line trajectory that connects a cone point to another cone point without passing through a cone point in the interior. To each saddle connection $\gamma$ one can assign a holonomy vector $v_{\gamma}$ which records how far $\gamma$ travels in $\mathbb{C}$. 

In this article we are interested in fine statistical properties of the set of holonomy vectors on $O_{2n}$.
Veech proved that the number of saddle connections (rather, the set of holonomy vectors associated to saddle connections) grow quadratically with length, and saddle connection directions equidistribute on $S^1$ with respect to Lebesgue measure \cite{Vee98,Vorobets}, which suggests that saddle connection directions appear ``randomly.'' As a finer notion of randomness, one can study the \emph{gaps} between saddle connections \cite{Ath13}. Since the set of saddle connections are symmetric with respect to the coordinate axes, we consider the set of holonomy vectors that lie in the first quadrant. 

Let $\mathbb{S}_{2n}$ denote the set of slopes of holonomy vectors $v_{\gamma}$ of $O_{2n}$ with positive real component and non-negative imaginary component. Write $\mathbb{S}_{2n}$ as an increasing union of $\mathbb{S}_{2n}(k)$ for $k\to\infty$, where 
\[\mathbb{S}_{2n}(k)=\left\{v_{\gamma}\mid 0<\mathrm{Re}(v_{\gamma})\le k, \ \mathrm{Im}(v_\gamma)\geq 0\right\},
\]
and write the slopes in $\mathbb{S}_{2n}(k)$ in increasing order, denoted by
\[
\mathbb{S}_{2n}(k)=\left\{0=s_0^k<s_1^k<s^k_2<\cdots\right\}.
\]
We are interested in the distribution of gaps $s^k_i - s^k_{i-1}$ between consecutive slopes in $\mathbb{S}_{2n}(k)$ as $k\to\infty$.  It can be shown that the gaps in $\mathbb{S}_{2n}(k)$ eventually repeat after an index $N(k)$ depending on $k$ (for details, see Section \ref{slopegapdistributions}).  Moreover, Veech shows in \cite{Vee98} that for any Veech  (lattice) surface, $N(k)$ grows quadratically with $k$, see also \cite{Vorobets}.  Thus we define the set of renormalized slope gaps on $O_{2n}$ as
\[
\mathcal{G}_{2n}(k)=\left\{ k^2 (s_i^k-s_{i-1}^k) \ \vline \ 1\le i\le N(k), \ s_i^{k}\in\mathbb{S}_{2n}(k)\right\}.
\]
We have the following theorem for the limiting distribution of these sets.

\begin{thm}\label{distribution_description}
For any regular $2n$-gon $O_{2n}$ where $n\ge3$, there exists a limiting probability distribution function $f: \mathbb{R} \to [0,\infty)$ such that \[\lim_{k \to \infty} \frac{|\mathcal{G}_{2n}(k) \cap (a,b)|}{N(k)} = \int_{a}^b f(x) \ dx.\] The distribution is a piecewise analytic function with finitely many domains of analyticity that are describable in terms of integrals of elementary functions. Moreover, the distribution has finitely many points of non-differentiability, has no support at $0$, and has a quadratic tail. 
\end{thm}

Theorem \ref{distribution_description} is a special case of the results contained in \cite{Ath13,KumanduriWang}, and it tells us that the slopes of the $2n$-gon are not randomly distributed because if they were, we would expect the slope gap distribution to have support at $0$ and an exponential tail. The novelty of this paper is that our results allow one to explicitly calculate the distribution function $f$ in an algorithmic way for any $n$. For examples of the distributions for a variety of $n$, see Figure \ref{trgt3}. One consequence of our calculations is the observation that the slope gap distribution of a translation surface is not always unimodal (that is, it can have multiple local maxima).  To our knowledge, this is the first example of such a distribution, answering a question of Athreya.

Following the strategy in \cite{ACL, AC13} and \cite{UW}, we will translate the problem of finding the slope gap distribution into a problem involving dynamics. In particular, we will exploit the slope gap preserving properties of the horocycle flow 
and study return times (under the horocycle flow) of translation surfaces to an appropriate Poincar\'e section, i.e. a lower dimensional subset of the space through which almost every horocycle orbit passes in a nonempty, discrete set of times.


The information needed to explicitly calculate the distribution is a generalized description of the first return time function of the horocycle flow pertaining to the slope gaps of the regular unit $2n$-gon, as described in the following theorem.

\begin{thm}\label{return_time_description} There is a staircase surface $\mathcal{S}_{2n}$ in the $\glr$-orbit of $O_{2n}$ for which a Poincar\'e section to the horocycle flow on $\slr\cdot\mathcal{S}_{2n}$ can be parametrized by two disjoint triangles. In these coordinates the return time function is a piecewise function which can be described by $n+1$ elementary functions of the form \[R(x,y) = \frac{b}{x(ax + by)}\]
where $a$ and $b$ are constants with explicit formulas computed in Propositions \ref{small prop} and \ref{big prop}.
\end{thm}

The slope gap distribution of the $2n$-gon is then easily determined from the slope gap distribution of the staircase surface by a simple rescaling. See Section \ref{flatgeometryof2ngon} for a description of this staircase surface. Furthermore, building on Theorems \ref{distribution_description} and \ref{return_time_description}  we prove: 

\begin{restatable}{thm}{thmbounds}
\label{linear_upper_bound}
The number of non-differentiable points in the slope gap distribution for a regular 2n-gon has linear lower and upper bounds. In particular, \[ \frac{n}{5} -11 \leq \#\text{(Non-Differentiable Points)} \leq 2n +\left \lfloor \frac{n}{2} \right\rfloor + 1.\] 
\end{restatable}

To the best of our knowledge this theorem gives the first infinite family of translation surfaces for which there is a non-trivial upper (and lower) bound on the number of non-differentiability points. Theorem \ref{linear_upper_bound} answers a question by \cite{KumanduriWang}. 
Furthermore, in Section \ref{sect:bounds} we provide experimental evidence towards the precise asymptotics. 

\subsection{Some History on Slope Gap Distributions}

There is a rich body of work on the statistics of gaps between saddle connections, combining aspects of the geometry of translation surfaces, homogeneous dynamics, and number theory. 

Inspired by work of Eskin-Masur and Marklof-Str\"ombergsson \cite{EskinMasur,MarkStrom}, Athreya-Chaika prove that a translation surface is not a lattice surface if and only if the smallest gap in directions between saddle connections of bounded length decays faster than quadratically in length \cite{AthCha}. Their results show that the limiting gap distribution exists and is the same for almost every surface, and that this distribution has a quadratic tail and has support at zero, in contrast to lattice surfaces which never have support at zero. 

In  \cite{AC13}, Athreya-Cheung use dynamical methods to rederive the Hall distribution for the limiting behavior of gaps between elements in the Farey sequence of level $n$ (realized as slopes of saddle connections on the marked torus) and use their construction to obtain new fine statistical results in this setting. 
In \cite{ACL}, Athreya-Chaika-Leli\`evre compute the slope gap distribution for saddle connections on the golden L, the first example of an explicit computation of a slope gap distribution for a surface which is not a branched cover of a torus. 
In \cite{Ath13}, Athreya codifies the methods applied in the previous works into a series of meta-theorems that apply to gap sequences  satisfying a set of quite general conditions, which in particular apply to studying the slopes of saddle connections on both lattice and non-lattice translation surfaces.

In \cite{UW}, Uyanik-Work expand the list of surfaces for which the slope gap distribution has been explicitly calculated by finding the distribution for the regular octagon, which directly inspired the present work.  They also provide an algorithm for finding the slope gap distribution for any lattice surface, and use this to show that the distribution of any such surface is piecewise real analytic. Their algorithm is improved upon by Kumanduri-Sanchez-Wang in \cite{KumanduriWang} to further show that the distribution of a lattice surface always has a finite number of points of non-differentiability. They also prove that the slope gap distribution of any lattice surface has quadratic decay in the tail.

In \cite{Heersink}, Heersink uses a method developed by Fisher and Schmidt to lift the Poincar\'e section for $\slr/\slz$ found by Athreya-Cheung for any finite index subgroup $H$ of $\slz$, and uses that to compute gap distributions for various subsets of Farey fractions. In \cite{taha, taha2}, Taha finds a Poincar\'e section for the geodesic and horocycle flows on quotients $\slr$ by Hecke triangle groups to study the statistics of generalized Farey sequences.

In \cite{sanchez},  Sanchez provides the first explicit computation of a slope gap distribution for a non-lattice surface by finding the distribution for the class of surfaces known as double-slit tori. In a step toward computing the slope gap distribution for a generic surface (as described by \cite{AthCha}), Work finds a Poincar\'e section for the horocycle flow on $\mathcal{H}_1(2)$, the stratum of unit area 
translation surfaces with a single cone point of angle $6\pi$, and provides bounds on the return time function to this section \cite{Work}.

\subsection{Acknowledgements}

This work started as a project during 2020 SUMRY and the authors are grateful for the support from Yale University. The second named author is grateful for support from the NSF Postdoctoral Fellowship DMS-1903099. The authors thank Jayadev Athreya, Anthony Sanchez, Sunrose Shrestha, and Jane Wang for helpful conversations. The authors are grateful to the anonymous referee for a careful reading and useful feedback.

\section{Preliminaries}

\subsection{Translation Surfaces and the \texorpdfstring{$GL_2(\mathbb{R})$}{GL2R} action}\label{translationsurfaces}

A translation surface is a pair $(X,\omega)$ where $X$ is a Riemann surface, and $\omega$ is a holomorphic $1$-form on $X$. Outside of a finite subset $\Sigma$ of $X$, the $1$-form $\omega$ induces an atlas of charts $$\varphi_i:U_i\subset X\setminus \Sigma\to \mathbb{R}^2$$ where transition maps are Euclidean translations. The finite set of points in $\Sigma$ are called the \emph{singularities} and correspond to zeroes of the holomorphic one form \cite{Masur}. At a zero of order $k-1$ (which means that in local coordinates the chart looks like $z^k$ and hence $\omega=d(z^k)=z^{k-1}dz$) the total angle is $2\pi k$. 

Equivalently, a translation surface is given by a finite collection of polygons $\{P_1,\ldots, P_n\}$ in the plane together with a preferred vertical direction and such that for each edge there exists a parallel edge of the same length and opposite orientation (with respect to the interior of the polygons), and these pairs are glued together by a translation. Since translations are holomorphic functions and they preserve the standard $1$-form $dz$ on the plane, one gets a holomorphic  $1$-form on the glued together surface \cite{Masur}.

Two translation surfaces are \emph{equivalent} if there is an orientation preserving isometry that preserves the preferred vertical direction. Equivalently, there is a \emph{cut-translate-paste} transformation from one to another.

A translation surface $(X,\omega)$ comes equipped with topological data: the genus $g$, the number of cone points (number of zeroes of $\omega$), and the excess angle at each cone point (order of zeroes). The Riemann--Roch theorem asserts that the sum of the order of zeroes is equal to $2g-2$. Hence, we can record this topological data by a vector $\vec{\alpha}=(\alpha_1,\ldots, \alpha_k)$ where $\alpha_i$ is the order of $i^{th}$ zero. The set of equivalence classes of translation surfaces with fixed topological data $\vec{\alpha}$ is called a \emph{stratum}, and denoted by $\mathcal{H}(\vec{\alpha})$. 

A \emph{saddle connection} $\gamma$ on $(X,\omega)$ is a straight line trajectory that connects a cone point to another (not necessarily distinct) cone point without any other cone point in the interior. 

\bigskip
\begin{figure}[h!]
\centering
\begin{tikzpicture}[scale=1.5, rotate=0]

\foreach \pos in {0,5} {\draw[red] ({cos(36*\pos)}, {sin(36*\pos)})--({cos(36+36*\pos)},{sin (36+36*\pos)});}
\foreach \pos in {1,6} {\draw[green] ({cos(36*\pos)}, {sin(36*\pos)})--({cos(36+36*\pos)},{sin (36+36*\pos)});}
\foreach \pos in {2,7} {\draw[orange] ({cos(36*\pos)}, {sin(36*\pos)})--({cos(36+36*\pos)},{sin (36+36*\pos)});}
\foreach \pos in {3,8} {\draw[blue] ({cos(36*\pos)}, {sin(36*\pos)})--({cos(36+36*\pos)},{sin (36+36*\pos)});}
\foreach \pos in {4,9} {\draw[purple] ({cos(36*\pos)}, {sin(36*\pos)})--({cos(36+36*\pos)},{sin (36+36*\pos)});}

\draw[black] ({cos(36*6)}, {sin(36*6)})-- (0.9045, 0.2938);

\draw[black] ({cos(36)}, {sin(36)})-- (-0.9045, -0.2938);

\node at ({cos(36*1)},{sin(36*1)}) [circle,color=black,fill,inner sep=0.6pt]{}; 
\node at ({cos(36*7)},{sin(36*7)}) [circle,color=black,fill,inner sep=0.6pt]{}; 
\node at ({cos(36*3)},{sin(36*3)}) [circle,color=black,fill,inner sep=0.6pt]{}; 
\node at ({cos(36*9)},{sin(36*9)}) [circle,color=black,fill,inner sep=0.6pt]{}; 
\node at ({cos(36*5)},{sin(36*5)}) [circle,color=black,fill,inner sep=0.6pt]{}; 
\node at ({cos(36*2)},{sin(36*2)}) [circle,color=red,fill,inner sep=0.6pt]{}; 
\node at ({cos(36*4)},{sin(36*4)}) [circle,color=red,fill,inner sep=0.6pt]{}; 
\node at ({cos(36*6)},{sin(36*6)}) [circle,color=red,fill,inner sep=0.6pt]{}; 
\node at ({cos(36*8)},{sin(36*8)}) [circle,color=red,fill,inner sep=0.6pt]{}; 
\node at ({cos(36*10)},{sin(36*10)}) [circle,color=red,fill,inner sep=0.6pt]{}; 
\end{tikzpicture}

\caption{A saddle connection $\gamma$ on $X_{10}$}
\end{figure}
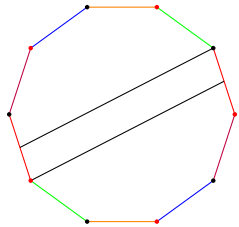

Integrating the $1$-form along an (oriented) saddle connection $\gamma$ determines a \emph{holonomy vector} 
\[
v_{\gamma}=\int_{\gamma}\omega\in\mathbb{C}
\]
which records how far and in what direction $\gamma$ travels in $\mathbb{C}$. 
This paper is concerned with holonomy vectors corresponding to saddle connections on $(X,\omega)$, but for brevity we will call them also saddle connections following the convention in \cite{ACL, UW} and denote the set of saddle connections on $(X,\omega)$ by $\Lambda_{sc}(X,\omega)$. For any translation surface $(X,\omega)$ the set $\Lambda_{sc}(X,\omega)$ is a discrete subset of $\mathbb{C}$, see \cite{HS, Zorich}.

There is a natural action of $\glr$ on the space of translation surfaces coming from the linear action of $\glr$ on $\mathbb{R}^2\cong\mathbb{C}$: for a matrix $A\in\glr$ and a translation surface $(X,\omega)$ given by a collection $\{P_1,\ldots, P_n\}$ of polygons, $A\cdot (X,\omega)$ is the translation surface given by $\{A\cdot P_1,\ldots, A\cdot P_n\}$. For $A\in\slr$, the action preserves the Euclidean area, hence induces an action on $\mathcal{H}_1(\vec{\alpha})$ of unit area translation surfaces of given topological type. 

The stabilizer of $(X,\omega)$ under the $\slr$ action is called the \emph{Veech group} of $(X,\omega)$ and denoted by ${\rm SL}(X,\omega)$. If ${\rm SL}(X,\omega)$ is a lattice, meaning that it has finite covolume in $\slr$ with respect to the Haar measure, then $(X,\omega)$ is called a \emph{lattice surface} or a \emph{Veech surface} \cite{HS}.  

\subsection{Flat geometry of \texorpdfstring{$2n$}{2n}-gons and Staircases} \label{flatgeometryof2ngon}

Let $O_{2n}$ be the regular $2n$-gon. In this section we prove that the matrix 
\[ 
M=\begin{pmatrix} 1 & -\cot(\pi/(2n)) \\ 0 & \sec\left(\frac{(n-2)\pi}{2n}\right)\end{pmatrix}=\begin{pmatrix} 1 & 0 \\ 0 & \sec\left(\frac{(n-2)\pi}{2n}\right)\end{pmatrix}\begin{pmatrix} 1 & -\cot(\pi/(2n)) \\ 0 & 1\end{pmatrix}
\]
takes $O_{2n}$ to a staircase shape $\mathcal{S}_{2n}$ that is easier to study. In other words, instead of studying the slope gap distribution on $O_{2n}$, we will study the slope gap distribution on $\mathcal{S}_{2n}$ and relate it to that of $O_{2n}$. In what follows we will blur the distinction between the polygon $\mathcal{S}_{2n}$ and the glued up surface, and denote both of them by $\mathcal{S}_{2n}$. Further, for the ease of notation we will denote $\mathcal{S}_{2n}$ by $\mathcal{S}$.

For $0 \leq i \leq \lceil n/2 \rceil - 1$ and $0\leq j\leq\lfloor n/2\rfloor - 1$, we define the finite sequences $h_i$ and $v_j$ as follows:\begin{eqnarray*}
    h_i & = & 1+\sum_{k=1}^i 2\cos(k\pi/n)=\csc\left( \frac{\pi}{2n}\right)\sin\left( \frac{\pi(1 + 2i)}{2n}\right) \\
    v_j & = & \cos\left(\frac{(n-2-2j)\pi}{2n}\right)\sec\left(\frac{(n-2)\pi}{2n}\right) =\csc\left( \frac{\pi}{n}\right)\sin\left( \frac{\pi(1 + j)}{n}\right)
\end{eqnarray*}

For $1 \leq i \leq \lceil n/2 \rceil - 1$, let $H_i$ be the rectangle with height $v_{i-1}$ and length $h_i$. Moreover, for $0\leq j\leq\lfloor n/2\rfloor - 1$, let $V_j$ be the rectangle with height $v_j$ and length $h_j$. Create the surface $\mathcal{S}$ by gluing the left edge of $H_1$ to the right edge of the unit square (which is $V_0$), the top edge of $H_i$ with the bottom edge of $V_i$, and the right edge of $V_i$ with the left edge of $H_{i+1}$ for all relevant $i$. Finally, identify opposite edges of the resulting staircase shape. This surface $\mathcal{S}$ is a staircase shape whose edge lengths are in the sets of all $h_i$ and $v_j$.

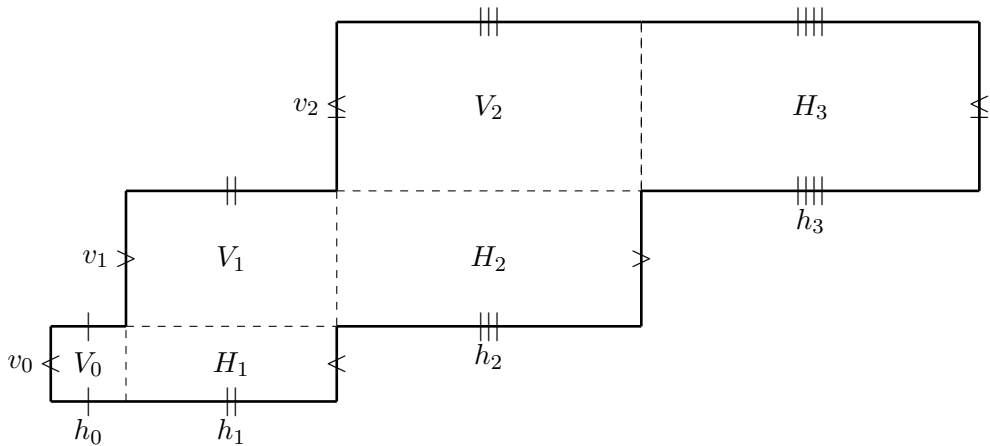
\begin{figure}[h!]
\centering
\begin{tikzpicture}[scale=1]
\coordinate (a) at (0,0);
\coordinate (b) at (0,1);
\coordinate (c) at (1,0);
\coordinate (d) at (1,1);
\coordinate (e) at (3.801937735804838,0);
\coordinate (f) at (3.801937735804838,1);
\coordinate (g) at (1,2.801937735804838);
\coordinate (h) at (3.801937735804838,2.801937735804838);
\coordinate (i) at (7.850855075327144,2.801937735804838);
\coordinate (j) at (7.850855075327144,1);
\coordinate (k) at (7.850855075327144,5.048917339522305);
\coordinate (l) at (3.801937735804838,5.048917339522305);
\coordinate (m) at (12.344814282762078,5.048917339522305);
\coordinate (n) at (12.344814282762078,2.801937735804838);

\draw [line width=1pt] (a) -- (b) node[midway] {$<$} node at (-0.4, 0.5) {$v_0$};
\draw [line width=1pt] (b)-- (d) node[midway] {$|$};
\draw [line width=1pt] (a)-- (c) node[midway] {$|$} node at (0.5, -0.4) {$h_0$};
\draw [line width=1pt] (c)-- (e) node[midway] {$||$} node at (2.4, -0.4) {$h_1$};
\draw [line width=1pt] (e)-- (f) node[midway] {$<$};
\draw [line width=1pt] (f)-- (j) node[midway] {$|||$} node at (5.825, 0.6) {$h_2$};
\draw [line width=1pt] (j)-- (i) node[midway] {$>$};
\draw [line width=1pt] (i)-- (n) node[midway] {$||||$} node at (10.1, 2.4) {$h_3$};
\draw [line width=1pt] (n)-- (m) node[midway] {$\leq$};
\draw [line width=1pt] (m)-- (k) node[midway] {$||||$};
\draw [line width=1pt] (k)-- (l) node[midway] {$|||$};
\draw [line width=1pt] (l)-- (h) node[midway] {$\leq$} node at (3.4, 3.92) {$v_2$};
\draw [line width=1pt] (h)-- (g) node[midway] {$||$};
\draw [line width=1pt] (g)-- (d) node[midway] {$>$} node at (0.6, 1.9) {$v_1$};

\draw [dashed] (c) -- (d);
\draw[dashed] (d) -- (f);
\draw [dashed] (h) -- (f);
\draw[dashed] (h) -- (i);
\draw [dashed] (k)-- (i);
\draw [dashed] (i)-- (k);

\node [draw=none] at (0.5, 0.5) {$V_0$};
\node [draw=none] at (2.4, 1.9) {$V_1$};
\node [draw=none] at (5.825, 3.92) {$V_2$};
\node [draw=none] at (2.4, 0.5) {$H_1$};
\node [draw=none] at (5.825, 1.9) {$H_2$};
\node [draw=none] at (10.1, 3.92) {$H_3$};
\end{tikzpicture}
    \caption{An example of $\mathcal{S}$ for $n = 7$.}
    \label{fig:first_example_of_staircase}
\end{figure}



\begin{prop}\label{transformation_to_staircase}
Left multiplication by the matrix $M=\begin{pmatrix} 1 & -\cot(\pi/(2n)) \\ 0 & \sec\left(\frac{(n-2)\pi}{2n}\right)\end{pmatrix}$ takes $O_{2n}$ to $\mathcal{S}$.  
\end{prop}

\begin{proof}

The surface $O_{2n}$ has a cylinder decomposition in $\theta=\pi/2n$ direction. Hence by a cut and paste transformation while keeping track of the side identifications, we can obtain a \emph{slanted} staircase representation of $O_{2n}$ as in Figure \ref{fig:slantedstaircase}. 

\begin{figure}[h!]
\centering
\begin{tikzpicture}[scale=1.9318, rotate=-15]

\draw[black,thick] ({cos(30*1)}, {sin(30*1)})-- ({cos(30*2)}, {sin(30*2)});

\draw[black,thick] ({cos(30*2)}, {sin(30*2)})-- ({cos(30*3)}, {sin(30*3)});

\draw[black,thick] ({cos(30*3)}, {sin(30*3)})-- ({cos(30*4)}, {sin(30*4)});

\draw[black,thick] ({cos(30*4)}, {sin(30*4)})-- ({cos(30*5)}, {sin(30*5)});

\draw[black,thick] ({cos(30*5)}, {sin(30*5)})-- ({cos(30*6)}, {sin(30*6)});

\draw[gray,thick] ({cos(30*7)}, {sin(30*7)})-- ({cos(30*8)}, {sin(30*8)});
\draw[gray,thick] ({cos(30*8)}, {sin(30*8)})-- ({cos(30*9)}, {sin(30*9)});
\draw[black,very thick] ({cos(30*9)}, {sin(30*9)})-- ({cos(30*10)}, {sin(30*10)});
\draw[gray,thick] ({cos(30*10)}, {sin(30*10)})-- ({cos(30*11)}, {sin(30*11)});
\draw[gray,thick] ({cos(30*11)}, {sin(30*11)})-- ({cos(30*12)}, {sin(30*12)});

\draw[black,dotted,thick] ({cos(30*1)}, {sin(30*1)})-- ({cos(30*12)}, {sin(30*12)});

\draw[black,dotted,thick] ({cos(30*6)}, {sin(30*6)})-- ({cos(30*7)}, {sin(30*7)});

\draw[black, very thick] ({cos(30*11)}, {sin(30*11)})-- ({cos(30*8)}, {sin(30*8)});

\draw[black, very thick] ({cos(30*12)}, {sin(30*12)})-- ({cos(30*7)}, {sin(30*7)});

\draw[gray] ({cos(30*1)}, {sin(30*1)})-- ({cos(30*6)}, {sin(30*6)});

\draw[gray] ({cos(30*2)}, {sin(30*2)})-- ({cos(30*5)}, {sin(30*5)});


\draw[black, very thick] ({cos(30*9)}, {sin(30*9)})-- ({cos(30*11)}, {sin(30*11)});
\draw[black, very thick] ({cos(30*8)}, {sin(30*8)})-- ({cos(30*12)}, {sin(30*12)});
\draw[black, very thick] ({cos(30*7)}, {sin(30*7)})-- ({cos(30*1)}, {sin(30*1)});
\draw[purple, dotted, very thick] ({cos(30*6)}, {sin(30*6)})-- ({cos(30*2)}, {sin(30*2)});
\draw[purple, dotted, very thick] ({cos(30*5)}, {sin(30*5)})-- ({cos(30*3)}, {sin(30*3)});
\draw[black, very thick] ({cos(30*10)}, {sin(30*10)})-- (1.36603, -0.366025);
\draw[black, very thick] ({cos(30*11)}, {sin(30*11)})-- ({cos(30*11)+cos(30*10)-cos(30*9)}, {sin(30*11)+sin(30*10)-sin(30*9)});

\draw[black, very thick] ({cos(30*11)}, {sin(30*11)})-- ({cos(30*11)+cos(30*2)-cos(30*6)}, {sin(30*11)+sin(30*2)-sin(30*6)});

\draw[black, very thick] ({cos(30*12)}, {sin(30*12)})-- ({cos(30*12)+cos(30*2)-cos(30*5)}, {sin(30*12)+sin(30*2)-sin(30*5)});

\draw[black, very thick] ({cos(30*8)}, {sin(30*8)})-- ({cos(30*8)+cos(30*5)-cos(30*3)}, {sin(30*8)+sin(30*5)-sin(30*3)});

\draw[black, very thick] ({cos(30*9)}, {sin(30*9)})-- ({cos(30*9)+cos(30*5)-cos(30*2)}, {sin(30*9)+sin(30*5)-sin(30*2)});

\draw[black, very thick] ({cos(30*12)}, {sin(30*12)})-- ({cos(30*12)+cos(30*1)-cos(30*7)}, {sin(30*12)+sin(30*1)-sin(30*7)});

\draw[black, very thick] ({cos(30*1)}, {sin(30*1)})-- ({cos(30*1)+cos(30*1)-cos(30*6)}, {sin(30*1)+sin(30*1)-sin(30*6)});

\draw[black, very thick] ({cos(30*7)}, {sin(30*7)})-- ({cos(30*7)+cos(30*6)-cos(30*2)}, {sin(30*7)+sin(30*6)-sin(30*2)});

\draw[black, very thick] ({cos(30*8)}, {sin(30*8)})-- ({cos(30*8)+cos(30*6)-cos(30*1)}, {sin(30*8)+sin(30*6)-sin(30*1)});

\fill[fill=green, opacity=0.1] ({cos(30*1)}, {sin(30*1)})--({cos(30*7)+cos(30*6)-cos(30*2)}, {sin(30*7)+sin(30*6)-sin(30*2)})--({cos(30*8)}, {sin(30*8)})--({cos(30*8)+cos(30*5)-cos(30*3)}, {sin(30*8)+sin(30*5)-sin(30*3)})--({cos(30*10)}, {sin(30*10)})--(1.36603, -0.366025)--({cos(30*11)}, {sin(30*11)})--({cos(30*11)+cos(30*2)-cos(30*6)}, {sin(30*11)+sin(30*2)-sin(30*6)})--({cos(30*12)}, {sin(30*12)})--({cos(30*12)+cos(30*1)-cos(30*7)}, {sin(30*12)+sin(30*1)-sin(30*7)})--cycle;

\end{tikzpicture}

\caption{Cylinder decomposition of $O_{2n}$ in the $\theta=\pi/2n$ direction}
\label{fig:slantedstaircase}
\end{figure}
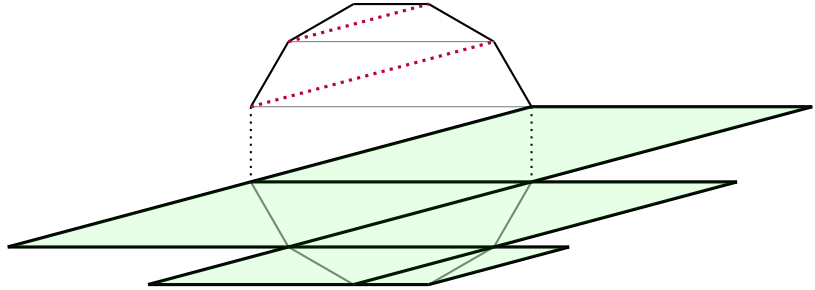

The matrix 
$$\begin{pmatrix} 1 & -\cot(\pi/(2n)) \\ 0 & 1\end{pmatrix}$$ is a shear that preserves the horizontal lines and distances and sends the lines with slope $\pi/2n$ to vertical lines. Hence we get a \emph{staircase to the left}, see Figure \ref{staircases}. 

On the other hand, the matrix
$$\begin{pmatrix} 1 & 0 \\ 0 & \sec\left(\frac{(n-2)\pi}{2n}\right)\end{pmatrix}$$
only scales the vertical direction while keeping the $x$ component constant, and it is only necessary to get the shortest vertical side to be of length $1$. We can then apply a cut-and-paste transformation to orient our staircase from left to right. The claim about the edge lengths of the staircase follows from a straightforward computation.

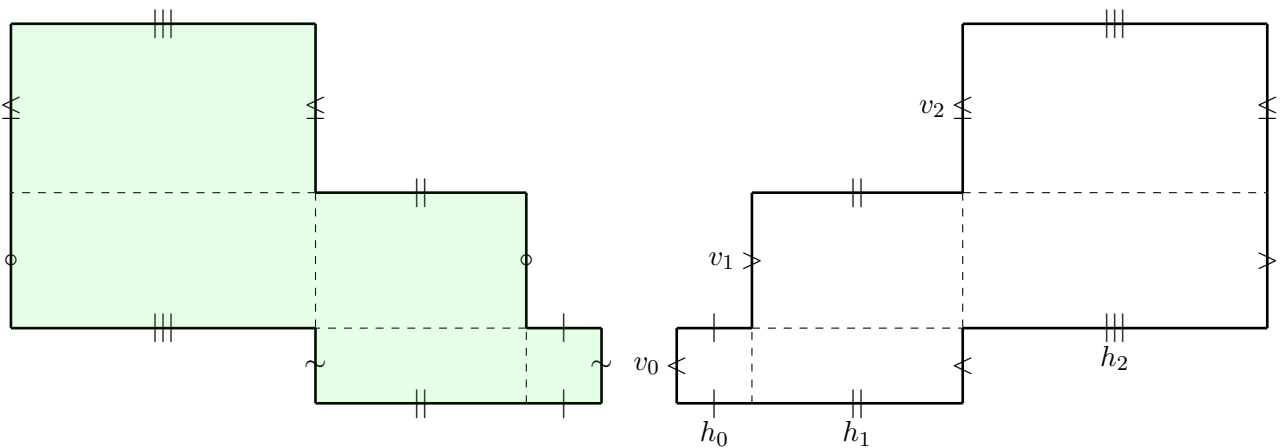
\begin{figure}[h!]

\begin{tikzpicture}[scale=1]
\begin{scope}
\coordinate (a) at (0,0);
\coordinate (b) at (0,1);
\coordinate (c) at (-1,0);
\coordinate (d) at (-1,1);
\coordinate (e) at (-3.801937735804838,0);
\coordinate (f) at (-3.801937735804838,1);
\coordinate (g) at (-1,2.801937735804838);
\coordinate (h) at (-3.801937735804838,2.801937735804838);
\coordinate (i) at (-7.850855075327144,2.801937735804838);
\coordinate (j) at (-7.850855075327144,1);
\coordinate (k) at (-7.850855075327144,5.048917339522305);
\coordinate (l) at (-3.801937735804838,5.048917339522305);

\draw [line width=1pt] (a) -- (b) node[midway] {$\sim$} ;
\draw [line width=1pt] (b)-- (d) node[midway] {$|$};
\draw [line width=1pt] (a)-- (c) node[midway] {$|$};
\draw [line width=1pt] (c)-- (e) node[midway] {$||$};
\draw [line width=1pt] (e)-- (f) node[midway] {$\sim$};
\draw [line width=1pt] (f)-- (j) node[midway] {$|||$};
\draw [line width=1pt] (j)-- (i) node[midway] {$\circ$};
\draw [line width=1pt] (k)-- (l) node[midway] {$|||$};
\draw [line width=1pt] (l)-- (h) node[midway] {$\leq$};
\draw [line width=1pt] (h)-- (g) node[midway] {$||$};
\draw [line width=1pt] (g)-- (d) node[midway] {$\circ$};

\draw [dashed] (c) -- (d);
\draw[dashed] (d) -- (f);
\draw [dashed] (h) -- (f);
\draw[dashed] (h) -- (i);
\draw [line width=1pt] (i)-- (k) node[midway] {$\leq$};

\fill[fill=green, opacity=0.1] (a)--(b)--(d)--(g)--(h)--(l)--(k)--(j)--(f)--(e)--cycle;
\end{scope}

\begin{scope}[shift={(1,0)}]

\coordinate (a) at (0,0);
\coordinate (b) at (0,1);
\coordinate (c) at (1,0);
\coordinate (d) at (1,1);
\coordinate (e) at (3.801937735804838,0);
\coordinate (f) at (3.801937735804838,1);
\coordinate (g) at (1,2.801937735804838);
\coordinate (h) at (3.801937735804838,2.801937735804838);
\coordinate (i) at (7.850855075327144,2.801937735804838);
\coordinate (j) at (7.850855075327144,1);
\coordinate (k) at (7.850855075327144,5.048917339522305);
\coordinate (l) at (3.801937735804838,5.048917339522305);
\coordinate (m) at (12.344814282762078,5.048917339522305);
\coordinate (n) at (12.344814282762078,2.801937735804838);

\draw [line width=1pt] (a) -- (b) node[midway] {$<$} node at (-0.4, 0.5) {$v_0$};
\draw [line width=1pt] (b)-- (d) node[midway] {$|$};
\draw [line width=1pt] (a)-- (c) node[midway] {$|$} node at (0.5, -0.4) {$h_0$};
\draw [line width=1pt] (c)-- (e) node[midway] {$||$} node at (2.4, -0.4) {$h_1$};
\draw [line width=1pt] (e)-- (f) node[midway] {$<$};
\draw [line width=1pt] (f)-- (j) node[midway] {$|||$} node at (5.825, 0.6) {$h_2$};
\draw [line width=1pt] (j)-- (i) node[midway] {$>$};
\draw [line width=1pt] (k)-- (l) node[midway] {$|||$};
\draw [line width=1pt] (l)-- (h) node[midway] {$\leq$} node at (3.4, 3.92) {$v_2$};
\draw [line width=1pt] (h)-- (g) node[midway] {$||$};
\draw [line width=1pt] (g)-- (d) node[midway] {$>$} node at (0.6, 1.9) {$v_1$};

\draw [dashed] (c) -- (d);
\draw[dashed] (d) -- (f);
\draw [dashed] (h) -- (f);
\draw[dashed] (h) -- (i);
\draw [line width=1pt] (i)-- (k) node[midway] {$\leq$};

\end{scope}

\end{tikzpicture}
\caption{A staircase representative of $O_{12}$}
\label{staircases}
\end{figure}

\end{proof}

We finish this section by recording some geometric properties of $\mathcal{S}$ that will be useful in the future. These result from the formulae for $h_i$ and $v_i$ and trigonometric identities.

\begin{rem}\label{vertical_squares}
For $1 \leq i \leq \floor{n/2}-1$, 
\[v_{i-1} + v_i = h_i.\]
Equivalently, for all relevant $i$, the rectangle formed by $V_i \cup H_i$ is a square.
\end{rem}

\begin{rem}\label{constant_aspect_ratio}
For $0 \leq i \leq \floor{(n-1)/2}-1$,
\[2 + 2\cos(\pi/n) = \frac{h_i + h_{i+1}}{v_i}.\]  Equivalently, the rectangles formed by $V_i \cup H_{i+1}$ all have the same aspect ratio.
\end{rem}

\begin{rem}\label{counting_backwards} 
For all $i$ and $n$, we have the identities
\[h_{n-i} = h_{i-1}\] 
and
\[v_{n-i} = v_{i-2}.\]
Although the interpretation of $h_i$ and $v_i$ as side lengths on our staircase shape does not make sense for $i> \lfloor n/2\rfloor$ or $i<0$, it will still be convenient in stating some of our results to use the formulas for $h_i$ and $v_i$ outside of this region.
\end{rem}

\section{From Slope Gaps to Dynamics}

\subsection{Dynamical Reframing}\label{dynamicalreframing}

Let $\mathcal{S}$ be the staircase representative of a regular $2n$-gon. Recall from the introduction that  
\[
\mathbb{S}(k) = \{0=s_0 < s_1 < s_2 < \cdots\}
\] 
is the ordered set of slopes of holonomy vectors on $\mathcal{S}$ with nonnegative imaginary component and positive real component $\leq k$.  Recall that the gaps between consecutive elements of $\mathbb{S}(k)$ repeat with a period of $N(k)$, and that $N(k)$ grows like $k^2$.

Let \[\mathcal{G}(k) = \{k^2(s_i - s_{i-1}) \ | \ 1 \leq i \leq N(k), \ s_i \in \mathbb{S}(k)\}\] denote the associated set of renormalized slope gaps on $\mathcal{S}$. Our goal is to compute the limit
\[\lim_{k \to \infty} \frac{|\mathcal{G}(k) \cap (a,b)|}{N(k)}\] 
for any interval $(a,b)\subseteq \mathbb{R}$. 
To compute this limit, we reframe the question about gaps between slopes in terms of return times for the horocycle flow to an appropriate Poincar\'e section. 

In this article, the horocycle flow is given by the left action of the following one parameter subgroup: 
\[\left\{h_s = \begin{pmatrix} 1 & 0 \\ -s & 1\end{pmatrix} \ \vline \ s \in \mathbb{R}\right\}.\] 
Observe that given any vector $v$,
\[\mathrm{slope}(h_sv) = \mathrm{slope}(v) - s,\] 
so this definition of the horocycle flow acts on slopes by translation. Most importantly, the horocycle flow preserves \textit{gaps} in slopes.

Let $\Gamma$ denote the Veech group of $\mathcal{S}$.  We will study the action of $h_s$ on the space $\slr/\Gamma$, which we can identify with $\slr\cdot \mathcal{S}$. If we denote the set of holonomy vectors on $\mathcal{S}$ by $\Lambda_{sc}(\mathcal{S})$, then for any $g\in\glr$ we have 
\[
g\cdot\Lambda_{sc}(\mathcal{S})=\Lambda_{sc}(g\cdot\mathcal{S}).
\]
Consider the set
\begin{align*}
    \Omega=\{g\Gamma\mid g\cdot\Lambda_{sc}(\mathcal{S}) \cap (0,1]\neq\emptyset\}
\end{align*}
in
$\slr/\Gamma$, where the interval $(0,1]$ is considered as a subset of the real axis inside of $\mathbb{C}$. In \cite{Ath13}, Athreya proved that $\Omega$ is a Poincar\'e section to the horocycle flow $h_s$ on $\slr/\Gamma$, meaning that for almost every $g\Gamma\in\slr/\Gamma$, the horocycle orbit $\{h_sg\Gamma\}_{s\in\R}$ intersects $\Omega$ in a nonempty, countable, discrete set of times. It is then possible to study the continuous horocycle flow on $\slr/\Gamma$ by studying its discrete time return map to the lower-dimensional set $\Omega$. Building on work of \cite{ACL, UW} we parametrize this Poincar\'e section and relate the first return time function for the Poincar\'e section to the slope gap distribution of $\mathcal{S}$.

\subsection{Description of the Poincar\'e Section}

We now give an explicit description of the Poincar\'e section for the surface $\mathcal{S}$ following the algorithm described in \cite{UW}. The parabolic and elliptic generators for $\Gamma_{O_{2n}}$, the Veech group of $O_{2n}$, are respectively given by \[S= \begin{pmatrix}1 & -2\cot(\pi/(2n)) \\ 0 & 1 \end{pmatrix}, \quad R=\begin{pmatrix} \cos(\pi/n) & \sin(\pi/n) \\ -\sin(\pi/n) & \cos(\pi/n) \end{pmatrix},\] so the corresponding generators $S'$ and $R'$ for $\Gamma$, the Veech group of $\mathcal{S}$, are given by conjugating $S$ and $R$ by $M$: \begin{eqnarray*}S'= MSM^{-1}&=& \begin{pmatrix}1 & -2(1+\cos(\pi/n)) \\ 0 & 1 \end{pmatrix}, \\[10pt]R' = MRM^{-1} &=& \begin{pmatrix} 1+2\cos(\pi/n) & 2(1+\cos(\pi/n)) \\ -1 & -1 \end{pmatrix}.\end{eqnarray*}

It has been shown by Veech, see \cite{HSinvariants}, that the Veech group for $O_{2n}$ (and hence $\mathcal{S}$) has two cusps. In fact, the image of $\Gamma_{O_{2n}}$ in $\pslr$ is isomorphic to the triangle group $\Delta(n,\infty, \infty)$.

Let $P_1$ and $P_2$ be the maximal parabolic subgroups representing the conjugacy classes of all maximal parabolic subgroups corresponding to these cusps. It follows from \cite{V89} that the set $\Lambda_{sc}(\mathcal{S})$ of saddle connections is a disjoint union of $\Gamma = \text{SL}(\mathcal{S})$ orbits of saddle connections,
\[\Lambda_{sc}(\mathcal{S})=\Gamma\cdot w_1\sqcup \Gamma\cdot w_2
\]
where each $w_i$ is the shortest holonomy vector in the  eigenspace corresponding to a generator for the infinite cyclic factor of $P_i$. We will see in the proof of Proposition \ref{transversalprop} that for our setting we may use $w_1 = \begin{pmatrix} 1 \\ 0 \end{pmatrix}$ and $w_2 = \begin{pmatrix} 0 \\ 1\end{pmatrix}$.

The decomposition of $\Lambda_{sc}(\mathcal{S})$ into disjoint orbits allows us to write \[\Omega=\Omega_{w_1} \cup \Omega_{w_2},\] 
where $\Omega_{w_i} = \{g\Gamma \ | \ g\Gamma\cdot w_i \cap (0, 1] \neq \emptyset\}$. 

Let \[M_{x,y} =\begin{pmatrix}x & y \\ 0 & x^{-1} \end{pmatrix},\] \[\Omega_1 = \{(x,y) \in \R^2 \ | \ 0 < x \leq 1, 1-2(1+\cos(\pi/n))x < y \leq 1\},\] and \[\Omega_2 = \{(x,y) \in \R^2 \ | \ 0 < x \leq 1, 1-x < y \leq 1\}.\] 

\begin{prop}\label{transversalprop}
There are coordinates from the set $\Omega$ to $\Omega_1 \cup \Omega_2$. More precisely, the bijection between $\slr/\Gamma$ and $\slr \cdot \mathcal{S}$ induces a bijection from  $\Omega_{w_i}$ to $\{M_{x,y}C_i\cdot \mathcal{S}\mid (x,y)\in\Omega_i\}$ where the latter set is in bijection with $\Omega_i$ for a suitable matrix $C_i$. 
\end{prop}

\begin{proof}
   This statement follows from Theorem 1.2 in \cite{UW}, but we illustrate the explicit computations for our setting.  The Veech group $\Gamma$ has two cusps, and let $S_1$ and $S_2$ be the generators of the infinite cyclic factors corresponding to maximal parabolic subgroups. One can check that \[S_1 = S' = \begin{pmatrix} 1 & -2(1+\cos(\pi/n)) \\ 0 & 1 \end{pmatrix}\] and 
   \begin{align}
       S_2 = ((R')^{n-1}S')^{-1} = \begin{pmatrix} 1 & 0 \\ 1 & 1 \end{pmatrix}.\label{S2}
   \end{align}
       Note that $S_1$ and $S_2$ both have eigenvalue $1$, and that their eigenvectors are in the directions of $w_1$ and $w_2$, respectively. Moreover, $w_1$ and $w_2$ are the shortest holonomy vectors in these directions. Now, pick $C_1 = \begin{pmatrix} 1 & 0 \\ 0 & 1 \end{pmatrix}$, so that \[C_1S_1C_1^{-1} = \begin{pmatrix} 1 & \alpha_1 \\ 0 & 1 \end{pmatrix} = \begin{pmatrix} 1 & -2(1+\cos(\pi/n)) \\ 0 & 1 \end{pmatrix}\] and $C_1 \cdot w_1 = \begin{pmatrix} 1 \\ 0 \end{pmatrix}$. Thus, we can set \[\Omega_1 = \{(x,y) \ |\ 0 < x \leq 1, 1 - 2(1+\cos(\pi/n))x < y \leq 1\}.\]

    Now, pick $C_2 = \begin{pmatrix} 0 & 1 \\ -1 & 0 \end{pmatrix}$, so that \[C_2S_2C_2^{-1} = \begin{pmatrix} 1 & \alpha_2 \\ 0 & 1 \end{pmatrix} = \begin{pmatrix} 1 & -1 \\ 0 & 1 \end{pmatrix}\] and $C_2 \cdot w_2 = \begin{pmatrix} 1 \\ 0 \end{pmatrix}$. Thus, we can set \[\Omega_{2} = \{(x,y) \ |\ 0 < x \leq 1, 1 - x < y \leq 1\}.\]

\end{proof}

\subsection{Return time and volume computations}

Recall that the Poincar\'e section $\Omega$ is defined as 
\[\Omega=\{g\Gamma\mid g\cdot\Lambda_{sc}(\mathcal{S}) \cap (0,1]\neq\emptyset\}.
\]

Given a point $X$ in $\Omega$ the return time function $R: \Omega\to \R^+$  gives the amount of time required for $X$ to return to $\Omega$ under the horocycle flow $h_s$. Since $h_s$ acts on slopes by translation, in the language of our parametrization, for any $(x,y) \in \Omega_i$, the return time $R(x,y)$ is the slope of the saddle connection (with horizontal component at most $1$) on $M_{x,y}C_i\cdot\mathcal{S}$ with smallest positive slope among all saddle connections with horizontal component at most $1$. 

In this section, we prove that there are finitely many saddle connections of interest on $C_i\cdot\mathcal{S}$, i.e. for any given $(x,y) \in \Omega_i$, there is a finite set of \textit{fixed} saddle connections on $C_i\mathcal{S}$ that are candidates for being the saddle connection on $M_{x,y}C_i\cdot\mathcal{S}$ of minimum positive slope.

Following the terminology introduced in \cite{KumanduriWang}, we say that a (holonomy) vector on $C_i\cdot\mathcal{S}$ is a \emph{winner} or a \emph{winning saddle connection} at the point $(x,y)\in\Omega_i$ if it is the vector whose image under $M_{x,y}$ has the smallest slope among all saddle connections with positive slope and horizontal length at most $1$ on $M_{x,y}C_i\cdot\mathcal{S}$.

In Section \ref{returntimeforomega2}, we show that the only winning saddle connection for $\Omega_2$ is $\cvec{0}{1}$. On the other hand, there are $n$ saddle connections of interest for $\Omega_1$, given by $\begin{pmatrix} 0 \\ 1 \end{pmatrix}$ and the diagonals of the rectangles $H_i$ and $V_i$ (see Figure \ref{winningvectors}).  Each of these saddle connections wins at some point on $\Omega_1$ and hence partition $\Omega_1$ into $n$ regions.
We provide an explicit description of this partition in Section \ref{returntimeforomega1}. These divisions and saddle connections therefore give us an explicit description of the return time function: if $(x,y) \in \Omega_i$ is arbitrary and $\cvec{a}{b}$ is the winning saddle connection on $M_{x,y}C_i\cdot\mathcal{S}$, then \[R(x,y) = \frac{b}{x(ax + by)}.\]

\begin{figure}[h!]
\centering
\begin{tikzpicture}[scale=0.9]
\coordinate (a) at (0,0);
\coordinate (b) at (0,1);
\coordinate (c) at (1,0);
\coordinate (d) at (1,1);
\coordinate (e) at (3.801937735804838,0);
\coordinate (f) at (3.801937735804838,1);
\coordinate (g) at (1,2.801937735804838);
\coordinate (h) at (3.801937735804838,2.801937735804838);
\coordinate (i) at (7.850855075327144,2.801937735804838);
\coordinate (j) at (7.850855075327144,1);
\coordinate (k) at (7.850855075327144,5.048917339522305);
\coordinate (l) at (3.801937735804838,5.048917339522305);
\coordinate (m) at (12.344814282762078,5.048917339522305);
\coordinate (n) at (12.344814282762078,2.801937735804838);

\draw [line width=1pt] (a) -- (b) node at (-0.4, 0.5) {$v_0$};
\draw [line width=2pt, color=magenta] (a) -- (b);
\draw [line width=1pt] (b)-- (d);
\draw [line width=1pt] (a)-- (c) node at (0.5, -0.4) {$h_0$};
\draw [line width=1pt] (c)-- (e) node at (2.4, -0.4) {$h_1$};
\draw [line width=1pt] (e)-- (f);
\draw [line width=1pt] (f)-- (j) node at (5.825, 0.6) {$h_2$};
\draw [line width=1pt] (j)-- (i);
\draw [line width=1pt] (i)-- (n) node at (10.1, 2.4) {$h_3$};
\draw [line width=1pt] (n)-- (m);
\draw [line width=1pt] (m)-- (k);
\draw [line width=1pt] (k)-- (l);
\draw [line width=1pt] (l)-- (h) node at (3.4, 3.92) {$v_2$};
\draw [line width=1pt] (h)-- (g);
\draw [line width=1pt] (g)-- (d) node at (0.6, 1.9) {$v_1$};
\draw [line width=2pt, color=magenta] (a) -- (d);
\draw [line width=2pt, color=magenta] (c) -- (f);
\draw [line width=2pt, color=magenta] (d) -- (h);
\draw [line width=2pt, color=magenta] (f) -- (i);
\draw [line width=2pt, color=magenta] (h) -- (k);
\draw [line width=2pt, color=magenta] (i) -- (m);

\draw [dashed] (c) -- (d);
\draw[dashed] (d) -- (f);
\draw [dashed] (h) -- (f);
\draw[dashed] (h) -- (i);
\draw [dashed] (k)-- (i);
\end{tikzpicture}
    \caption{The saddle connections of interest for $\Omega_1$ (in magenta) on $\mathcal{S}$ for $n = 7$.}
   \label{winningvectors}
\end{figure}
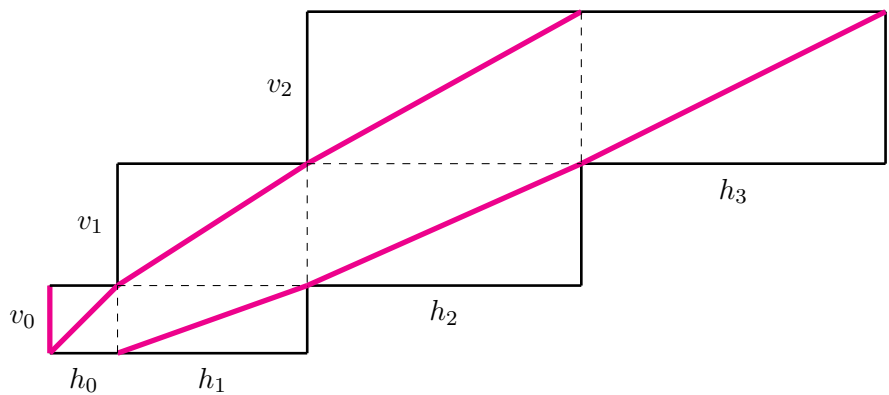

\subsubsection{Description of \texorpdfstring{$R(x,y)$}{R(x,y)} Over \texorpdfstring{$\Omega_2$}{Omega2}}\label{returntimeforomega2}

 We describe $R(x,y)$ for $(x,y) \in \Omega_2$ by showing that $\cvec{0}{1}$ is the only winning vector on $\Omega_2$. 

\begin{figure}[h!]
    \centering
    \begin{tikzpicture}[scale=1]
\begin{axis}[
axis lines=middle,
ymajorgrids=false,
xmajorgrids=false,
xmin=0,
xmax=1.5,
ymin=0,
ymax=1.5,
xlabel={$x$},
ylabel={$y$},]
\draw [line width=1pt] (0,1)-- (1,1);
\draw [line width=1pt] (1,1)-- (1,0);
\draw [line width=1pt, dashed] (1,0)-- (0,1);
\node [draw=none] at (0.7, 0.7) (a) {$\Omega_2$};
\end{axis}
\end{tikzpicture}
    \caption{An illustration of $\Omega_2$ with coordinates in $\R^2$, for any $n$.}
    \label{fig:omega2}
\end{figure}
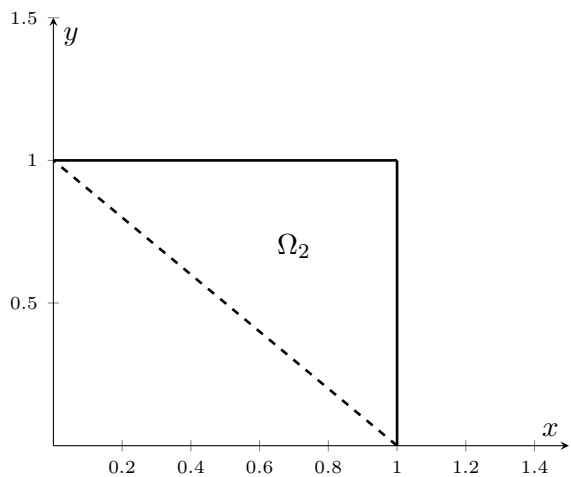

\begin{prop}\label{small prop}
The vector $\cvec{0}{1}$ is the only winning saddle connection on $\Omega_2$, so the return time function is given by $R(x,y) = \frac{1}{xy}$ for $(x,y)\in\Omega_2$.
\end{prop}

\begin{proof}
Since we will be studying $\Omega_{2}$, we will let $\mathcal{S}^R = C_2\mathcal{S}$ and study the saddle connections on $\mathcal{S}^R$. We claim that the image of the saddle connection $\cvec{0}{1}$ on $\mathcal{S}^R$ is the only saddle connection whose image is a vector with positive slope and horizontal component at most $1$. 

Let $\cvec{a}{b}$ be some saddle connection on $\mathcal{S}^R$ with an image under $M_{x,y}$ (for $(x,y) \in \Omega_2$) that has positive slope and horizontal component at most $1$. Note that $M_{x,y}\cvec{a}{b} = \cvec{ax + by}{b/x}$ where $0 < ax + by \leq 1$. Observe that $b\ge0$ by as $\Omega_2$ is in the first quadrant. 

For the saddle connection $\cvec{a}{b}$ we compute the slope of its image as \[\frac{b}{x(ax + by)} = \frac{1}{x(x(a/b) + y))}.\] Then, note that if $a/b < 0$ (recall that the slope is positive), then the slope would be greater than the case where $a/b = 0$, i.e. when $\cvec{a}{b} = \cvec{0}{1}$, so we need only consider the cases when $a \geq 0$. Moreover, since $\cvec{a}{b}$ cannot be horizontal (or else its image would be horizontal), $b$ must be positive. 

The shortest positive vertical distance on $\mathcal{S}^R$ between any cone points is $1$, so $b \geq 1$. Suppose for the sake of contradiction that $\cvec{a}{b}\neq\cvec{0}{1}$, and hence $a > 0$. Since the shortest positive horizontal distance between cone points on the surface $\mathcal{S}^R$ is $1$, we have $a \geq 1$. Combining this with the fact that $b \geq 1$, we have \[ax + by \geq x + y.\] Since $(a,b)\in\Omega_2$ we have \[ax + by \geq x + y > x + 1 - x = 1,\] which is a contradiction. This forces $a = 0$, so our desired original saddle connection must be $\cvec{a}{b} = \cvec{0}{1}$ over all of $\Omega_2$.
\end{proof}


\subsubsection{Description of \texorpdfstring{$R(x,y)$}{R(x,y)} Over \texorpdfstring{$\Omega_1$}{Omega1}}\label{returntimeforomega1}

To rephrase an earlier remark, given $(x, y) \in \Omega_1$, we say the saddle connection $v = \cvec{a}{b}$ is a winner at $(x, y)$ if $M_{x, y}v$ has horizontal component $0 < ax + by \leq 1$ and positive slope, and if for all $v' = \cvec{a'}{b'}$ saddle connections such that $M_{x, y}v'$ has horizontal component $0 < a'x + b'y \leq 1$ and positive slope, we have $0 < \mathrm{slope}(M_{x, y}v) \leq \mathrm{slope}(M_{x, y}v')$.
%
In $\Omega_1$, there are only finitely many winners and each winner has an associated convex polygonal region $P_i \subset \Omega_1$ as described in Proposition \ref{big prop}.

\begin{prop}\label{big prop}
    In the region \[P_1 = \{(x, y) \in \Omega_1 \ |\ x + y > 1\}\] the winner is $\cvec{0}{1}$, which forms the vertical edge of the rectangle $V_0$, so the return time is given by \[R(x,y) = \frac{1}{xy}.\] 
    
    For $2 \leq i \leq n-1$, in the region 
    \[P_i = \{(x, y) \in \Omega_1\ | \ xh_{i - 2} + yv_{i - 2} \leq 1, xh_{i - 1} + yv_{i - 1} > 1, xh_1 + y > 1\}\] 
    the winner is $\cvec{h_{i - 2}}{v_{i - 2}} = \cvec{h_{n+1-i}}{v_{n-i}}$. For $2\leq i \leq \lfloor n/2\rfloor +1$, this is the positive-sloped diagonal of the rectangle $V_{i - 2}$. For $\lfloor n/2\rfloor +2\leq i \leq n-1$, this is the positive-sloped diagonal of $H_{n+1-i}$. Hence the return time in either case is given by \[R(x,y) = \frac{v_{i-2}}{x(h_{i-2}x + v_{i-2}y)}.\] 
    
   
    
    
    In the region \[P_n = \{(x, y) \in \Omega_1\ |\ h_1x + y \leq 1\}\] the winner is $\cvec{h_1}{v_0}$, which is the positive-sloped diagonal of the rectangle $H_1$, so the return time is given by \[R(x,y) = \frac{v_0}{x(h_{1}x + v_{0}y)}.\]

\begin{figure}[h!]
    \centering
        \begin{tikzpicture}[scale=0.8]
    \coordinate (a) at (0,0);
    \coordinate (b) at (0,1);
    \coordinate (c) at (1,0);
    \coordinate (d) at (1,1);
    \coordinate (e) at (3.801937735804838,0);
    \coordinate (f) at (3.801937735804838,1);
    \coordinate (g) at (1,2.801937735804838);
    \coordinate (h) at (3.801937735804838,2.801937735804838);
    \coordinate (i) at (7.850855075327144,2.801937735804838);
    \coordinate (j) at (7.850855075327144,1);
    \coordinate (k) at (7.850855075327144,5.048917339522305);
    \coordinate (l) at (3.801937735804838,5.048917339522305);
    \coordinate (m) at (12.344814282762078,5.048917339522305);
    \coordinate (n) at (12.344814282762078,2.801937735804838);
    
    \draw [line width=2pt, color=violet] (a) -- (b);
    \draw [line width=1pt] (b)-- (d);
    \draw [line width=1pt] (a)-- (c);
    \draw [line width=1pt] (c)-- (e);
    \draw [line width=1pt] (e)-- (f);
    \draw [line width=1pt] (f)-- (j);
    \draw [line width=1pt] (j)-- (i);
    \draw [line width=1pt] (i)-- (n);
    \draw [line width=1pt] (n)-- (m);
    \draw [line width=1pt] (m)-- (k);
    \draw [line width=1pt] (k)-- (l);
    \draw [line width=1pt] (l)-- (h);
    \draw [line width=1pt] (h)-- (g);
    \draw [line width=1pt] (g)-- (d);
    \draw [line width=1pt, color=blue] (a) -- (d);
    \draw [line width=1pt, color=red] (c) -- (f);
    \draw [line width=1pt, color=green] (d) -- (h);
    \draw [line width=1pt, color=orange] (f) -- (i);
    \draw [line width=1pt, color=yellow] (h) -- (k);
    \draw [line width=1pt, color=gray] (i) -- (m);
    
    \draw [dashed] (c) -- (d);
    \draw[dashed] (d) -- (f);
    \draw [dashed] (h) -- (f);
    \draw[dashed] (h) -- (i);
    \draw [dashed] (k)-- (i);
    \end{tikzpicture}
    \caption{Color-coded relevant saddle connections for $n = 7$}
    \label{fig:colorcode}
\end{figure}
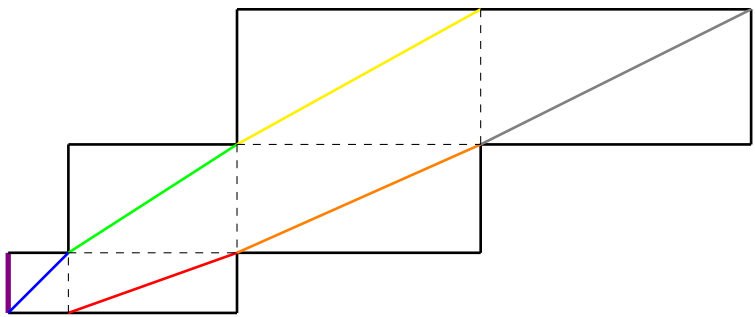

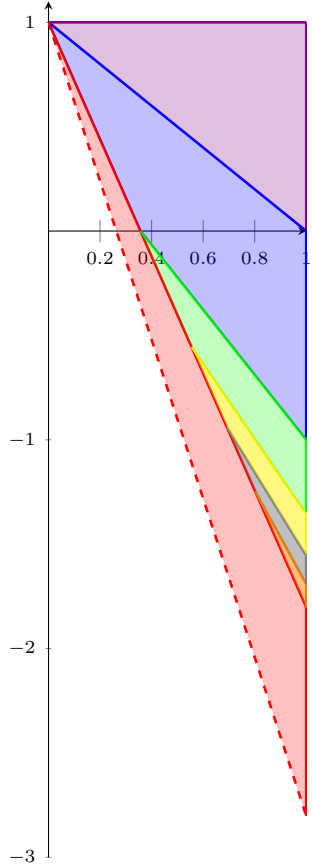
\begin{figure}[h!]
    \centering
    \begin{tikzpicture}
    \begin{axis}[
    axis lines=middle,
    ymajorgrids=false,
    xmajorgrids=false,
    xmin=0,
    xmax=1,
    ymin=-3,
    ymax=1.1,
    xscale=0.5,
    yscale=2]
    
    \coordinate (corner) at (1,1);
    \coordinate (A) at (0,1);
    \coordinate (B) at (1,0);
    \coordinate (C) at (0.357,0);
    \coordinate (D) at (0.555,-0.555);
    \coordinate (E) at (0.692, -0.939);
    \coordinate (F) at (0.802, -1.247);
    \coordinate (G) at (1, -1.802);
    \coordinate (H) at (1, -2.802);
    \coordinate (I) at (1, -1.692);
    \coordinate (J) at (1, -1.555);
    \coordinate (K) at (1, -1.3569);
    \coordinate (L) at (1, -1);

    \draw [line width=1pt, color=violet] (A) -- (corner);
    \draw [line width=1pt, color=violet] (corner) -- (B);
    \draw [line width=1pt, color=blue] (A)--(B);
    \draw [line width=1pt, color=red] (A)--(G);
    \draw [line width=1pt, color=red, dashed] (A)--(H);
    \draw [line width=1pt, color=red] (G)--(H);
    \draw [line width=1pt, color=orange] (F)--(I);
    \draw [line width=1pt, color=orange] (G)--(I);
    \draw [line width=1pt, color=gray] (I)--(J);
    \draw [line width=1pt, color=yellow] (J)--(K);
    \draw [line width=1pt, color=green] (K)--(L);
    \draw[line width=1pt, color=blue] (L)--(B);
    \draw[line width=1pt, color=green] (C)--(L);
     \draw[line width=1pt, color=yellow] (D)--(K);
     \draw[line width=1pt, color=gray] (E)--(J);
     
     \path[fill=violet, opacity=0.25] (A) -- (corner) -- (B);
     \path[fill=blue, opacity=0.25] (A) -- (B) -- (L) -- (C);
     \path[fill=green, opacity=0.25] (C) -- (L) -- (K) -- (D);
     \path[fill=yellow, opacity=0.5] (D) -- (K) -- (J) -- (E);
     \path[fill=gray, opacity=0.5] (E) -- (J) -- (I) -- (F);
     \path[fill=orange, opacity=0.5] (F) -- (I) -- (G);
     \path[fill=red, opacity=0.25] (A)--(G)--(H);
    \end{axis}
    \end{tikzpicture}
        \caption{The corresponding division of $\Omega_1$ for $n = 7$ using coordinates in $\mathbb{R}^2$. Each saddle connection is color coded to match the region it wins in. The purple region is $P_1$, the blue region is $P_2$, the green region is $P_3$, the yellow region is $P_4$, the gray region is $P_5$, the orange region is $P_6$, and the red region is $P_7$.}
    \label{fig:sample_division_of_transversal}
\end{figure}

\end{prop}
The proof of Proposition \ref{big prop} is quite long, and requires a case-by-case analysis, hence it is postponed to Appendix \ref{BigPropAppendix}.  On the other hand, using hyperbolic geometry we can compute the volume of $\slr/\Gamma_{O_{2n}}$ and compare it to the volume computation via integrating the return time function over the horocycle flow.  More precisely, since $R(x,y)$ is a roof function over $\Omega$ in the suspension space $\slr/\Gamma$, computing the volume under $R(x,y)$ should yield the volume of $\slr/\Gamma$.
The table in Figure \ref{volumetable} provides numerical volume computations for several values of $n$ to experimentally verify our division of the Poincar\'e section.

\begin{figure}[h!]
\centering
\begin{tabular}{ |c|c|c|c| } 
 \hline
 $n$ & Computed Volume & Actual Volume & Error \\
 \hline
 \hline
 2 & 4.9348 & $\pi^2/2$ & $< 0.001\%$\\
 \hline
 3 & 6.5797 & $2\pi^2/3$ & $< 0.001\%$\\ 
 \hline
 4 & 7.4022 & $3\pi^2/4$ & $< 0.001\%$\\ 
 \hline
 5 & 7.8957 & $4\pi^2/5$ & $< 0.001\%$ \\ 
 \hline
 6 & 8.2247 & $5\pi^2/6$ & $< 0.001\%$ \\
 \hline
 7 & 8.4597 & $6\pi^2/7$ & $< 0.001\%$ \\ 
 \hline
 8 & 8.6359 & $7\pi^2/8$ & $< 0.001\%$ \\ 
 \hline
 9 & 8.7730 & $8\pi^2/9$ & $< 0.001\%$ \\ 
 \hline
 10 & 8.8826 & $9\pi^2/10$ & $< 0.001\%$ \\
 \hline
 50 & 9.7622 & $49\pi^2/50$ & $< 0.001\%$ \\
 \hline
 100 & 9.7709 & $99\pi^2/100$ & $< 0.001\%$ \\
 \hline
\end{tabular}
\caption{Numerical computation of the volume of $\slr/\Gamma_{O_{2n}}$ using our Poincar\'e section and return time function compared to the exact volume. Within the limits of our computational software, there was essentially no error in the computed volume.}
\label{volumetable}
\end{figure}

\section{Slope Gap Distributions}
\label{slopegapdistributions}

Let $k\in \mathbb{N}$. Denote by $V(k)$ the vertical strip $(0,k]\times[0,\infty)$ and let $\Lambda^{\mathcal{S}}_{sc}(k)$ be the set of holonomy vectors of $\mathcal{S}$ that lie inside $V(k)$.  Let 
\[
\mathbb{S}(k) = \{0=s_0<s_1<s_2<\dots\}
\]
denote the ordered set of slopes of vectors in $\Lambda^{\mathcal{S}}_{sc}(k)$.
Recall that the horocycle flow 
\[
\left\{h_s = \begin{pmatrix} 1 & 0 \\ -s & 1 \end{pmatrix} \ \vline \ s \in \mathbb{R}\right\}
\]
acts on slopes by translation, so the slopes in $\mathbb{S}(k)$ are precisely the times when one of the vectors in $\Lambda^{\mathcal{S}}_{sc}(k)$ hits the horizontal axis.

From \eqref{S2} we see that the Veech group $\Gamma$ of $\mathcal{S}$ contains the parabolic element $S_2^{-1}=h_{1}$, which implies that $\mathcal{S}$ is periodic under the horocycle flow with period $1$. Hence $h_1\Lambda^{\mathcal{S}}_{sc}(k)= \Lambda^{\mathcal{S}}_{sc}(k)$, i.e. after subtracting $1$ from all of the slopes of holonomy vectors in the vertical strip $V(k)$ and removing the slopes that are negative, we get the same ordered set of slopes $\mathbb{S}(k)$ that we started with. In other words, there exists some $N(k)\in\mathbb{N}$ such that $s_{N(k)+i}-1=s_i$ for all $i\geq 0$.  This in turn implies that the slope gaps $s_i-s_{i-1}$ repeat after $i=N(k)$, since 
\[s_{N(k)+i} -s_{N(k)+i-1} = (s_i+1)-(s_{i-1}+1) = s_i-s_{i-1}.\]
Thus the gaps between the elements of $\mathbb{S}(k)$ are given by the (not renormalized) gap set
\[
\overline{\mathcal{G}}(k) = \{s_{i} - s_{i-1} \ \vline \ s_i \in \mathbb{S}(k), 1\leq i\leq N(k)\}.
\]
Since $N(k)$ grows like $k^2$ (this follows from work of Veech \cite{Vee95} for any Veech surface, of which $\mathcal{S}$ is an example), we may thus define our renormalized gap set to be 
\[
\mathcal{G}(k) = \{k^2(s_{i} - s_{i-1}) \ \vline \ s_i \in \mathbb{S}(k), 1\leq i\leq N(k)\}.
\]
Recall that our goal is to find the distribution of the renormalized gaps in $\mathcal{G}(k)$ as $k\to\infty$, i.e. we want to find a probability density function $f:\R\to[0,\infty)$ such that 
\[
\lim_{k\to\infty} \frac{|\mathcal{G}(k)\cap(a,b)|}{N(k)} = \int_a^b f(x) dx
\]
for any interval $(a,b)\subseteq \mathbb{R}$.  We will do this by representing the quantity on the left as a limit of Birkhoff sums over longer and longer periodic orbits for the first return map to $\Omega$ and leveraging known ergodic theory results for such limits.

Let 
\[
g_k = \begin{pmatrix}1/k & 0\\0 & k\end{pmatrix}.
\]
Observe that $g_k V(k) = V(1)$ and  $\mathrm{slope}(g_kv) = k^2\mathrm{slope}(v)$ for any $v\in\mathbb{R}^2$ and $k\in \mathbb{N}$.  Moreover, one can verify that
\begin{align*}
 g_k\cdot \Lambda^{\mathcal{S}}_{sc}(k) &= \Lambda^{g_k\mathcal{S}}_{sc}(1),\\
 k^2\mathbb{S}(k) &= \mathbb{S}^{g_k\mathcal{S}},
\end{align*}
and thus
\begin{align*}
\mathcal{G}(k) &= \overline{\mathcal{G}}^{g_k\mathcal{S}}
\end{align*}
where $\mathbb{S}^{g_k\mathcal{S}}$ denotes the ordered set of slopes of vectors in $\Lambda^{g_k\mathcal{S}}_{sc}(1)$ and $\overline{\mathcal{G}}^{g_k\mathcal{S}}$ is the associated non-renormalized gap set. In other words, the renormalized gaps between slopes of saddle connections on $\mathcal{S}$ with horizontal length less than or equal to $k$ is precisely equal to the set of (non-renormalized) gaps between slopes of saddle connections on $g_k\mathcal{S}$ with horizontal length less than or equal to $1$.
 
Let $T:\Omega\to\Omega$ denote the first return map for the horocycle flow to the Poincar\'e section $\Omega$, i.e. for $x=g\Gamma\in\Omega$ (considered as a subset of $\slr/\Gamma$) we have
\[
T(x) = h_{R(x)}x
\]
where $R:\Omega\to\mathbb{R}^+$ is the usual return time function. 

Observe that since $\Gamma\in\Omega$ is periodic under the horocycle flow with period $1$, then $g_k\Gamma\in\Omega$ is periodic with period $k^2$. 
Since the horocycle flow acts on slopes by translation, it preserves gaps between slopes. 
Then the slope gaps in 
\[
\mathcal{G}(k) = \overline{\mathcal{G}}^{g_k\mathcal{S}} = \{s'_{i} - s'_{i-1} \ \vline \ s'_i \in \mathbb{S}^{g_k\mathcal{S}}, 1\leq i \leq N(k)\}
\]
are precisely the times elapsed between consecutive intersections of the orbit of $g_k\Gamma$ with $\Omega$ under the horocycle flow, i.e.
\[
s'_{i} - s'_{i-1} = R(T^{i}(g_k\Gamma))
\] 
for all $i \geq 0$. If we let \[\chi_t(x) = \begin{cases} 1 & R(x) < t \\ 0 & R(x) \geq t
\end{cases},\] 
be the indicator function on $R^{-1}([0, t))$, then we can write a formula for the distribution of slope gaps as a limit of Birkhoff sums of the form
\begin{align}
    \lim_{k \to \infty} \frac{|\mathcal{G}(k) \cap [0,t)|}{N(k)} &= \lim_{k \to \infty} \frac{1}{N(k)} \sum_{i=0}^{N(k) - 1} \chi_t(T^i(g_k\Gamma)).\label{Birkhoff}
\end{align}
Thus, we have reframed our geometric problem about slope gaps on $\mathcal{S}$ into a dynamical problem concerning return times for a sequence of periodic points in $\Omega$.

We may now apply the following theorem and its corollary, which are adapted from Theorems 4.1-4.2 in \cite{ACL}, or Theorem 2.5 in \cite{Ath13}, which applies in a more general context. An  analogous theorem that applies to non-lattice surfaces also appears in \cite{ACL, Ath13} and relies on the equidistribution of long horocycle orbits proved in \cite{DS84}, however we do not need that result for our setting.

\begin{thm}[\cite{ACL}, Theorem 4.1]\label{ACL41}
Let $(X,\omega)$ be a lattice surface with Veech group $\Gamma$, and let $\Omega\subset\slr/\Gamma$ and $T$ be the Poincar\'e section and associated return map defined previously. If $x_k\in\Omega$ is a sequence of points that are periodic under the return map $T$ with periods $N(k)\to\infty$ as $k\to\infty$, then for any bounded, measurable function $f:\Omega\to\mathbb{R}$, we have
\[
\lim_{k\to\infty} \frac{1}{N(k)} \sum_{i=0}^{N(k)-1} f(T^i(x_k)) = \int_{\Omega} f dm
\]
where $m$ is the unique ergodic probability measure for $T$ supported on $\Omega$.
\end{thm}

\begin{proof}
We can realize $\slr/\Gamma$ as a suspension space over $\Omega$ with roof function given by the return time function:
\[
\slr/\Gamma \cong \{ (x,s) \ \vline \ x\in\Omega,\ s\in [0,R(x,y)]\}/\sim
\]
where $(x,R(x))\sim (T(x),0)$. The probability Haar measure $\mu$ on $\slr/\Gamma$ then decomposes as $d\mu = C ds dm$ for some constant $C$. A point $x_k\in\Omega$ is periodic for the return map $T$ if and only if $x_k$ is periodic under the continuous horocycle flow on $\slr/\Gamma$.
Let $\tau(k)$ denote the period of $x_k$ under the horocycle flow and let $\rho_k$ denote the invariant probability measure on $\slr/\Gamma$ supported on the periodic horocycle orbit of $x_k$. Let $\sigma_k$ denote the invariant probability measure on $\Omega$ supported on the periodic orbit of $x_k$ under $T$.
Observe that for bounded, measurable $f:\slr/\Gamma\to\mathbb{R}$, 
\begin{align*}
\int f d\rho_k &=  \frac{1}{\tau(k)}\int_0^{\tau(k)} f(h_sx_k)ds,\\
&= \frac{N(k)}{\tau(k)}\frac{1}{N(k)}\sum_{i=0}^{N(k)-1}\int_0^{R(T^i(x_k))} f(h_sT^i(x_k))ds,\\
&= \frac{N(k)}{\tau(k)} \int_\Omega \int_0^{R(x)} f(h_sx)ds d\sigma_k(x),
\end{align*}
in other words, $d\rho_k = \frac{N(k)}{\tau(k)} ds d\sigma_k$.  Moreover, by a theorem of Sarnak \cite{S81}, we know that long periodic horocycle orbits equidistribute with respect to the Haar measure, i.e. $d\rho_k\to d\mu= dsdm$. It follows from results in \cite{Vee98} that $\frac{N(k)}{\tau(k)}\to C$, and this implies that the corresponding long periodic orbits under $T$ must equidistribute in $\Omega$, i.e. $d\sigma_k\to dm$, which gives the result.
\end{proof}

One can check that the measure $dm$ on $\Omega$ is the appropriately scaled Lebesgue measure $dxdy$ on our parametrization of the Poincar\'e section given in Proposition \ref{transversalprop}.

\begin{cor}[\cite{ACL}, Theorem 4.2]
Let the setting be as in Theorem \ref{ACL41}, and
let $\mathcal{G}(k)$ be the renormalized gap set for  $(X,\omega)$. 
Then 
\begin{align}
\lim_{k \to \infty} \frac{|\mathcal{G}(k) \cap [0,t)|}{N(k)} = m\left(\{(x,y)\in\Omega \ \vline \ R(x,y)\in[0,t)\}\right).\label{measurecorollary}
\end{align}
\end{cor}

\begin{proof}
This is a direct application of Theorem \ref{ACL41} for the function $\chi_t$ and the sequence of points $x_k = g_k\Gamma$ with periods $N(k) = k^2$, along with the identity in \eqref{Birkhoff}.
\end{proof}

With our parametrization of $\Omega$ and formula for the return time function at every point in $\Omega$ calculated in the last section, it is now a standard problem in multivariable calculus to compute the right hand side of \eqref{measurecorollary}.  This gives us Theorem \ref{distribution_description} from the introduction.

Since our limiting distribution has no support at $0$ (indicating that there are no small gaps between slopes of saddle connections), this implies that the slope gap distribution for $O_{2n}$ is not random (if this were the case, we would expect an exponential distribution \textit{with} support at $0$). In Figure \ref{trgt3} we provide a few examples of the slope gap distributions on $O_{2n}$ for several values of $n$. 

An interesting conclusion that can be drawn from our calculation of the distribution for the $2n$-gon is that the slope gap distribution of a translation surface is not always unimodal.  For example, in the distribution for $O_{14}$ one finds that there is a local maximum at $t\approx 0.715353$ of about $0.691264$, followed by a local minimum at $t\approx 0.781831$ of about $0.681558$, followed by another local maximum at $t\approx 0.870497$ of about $0.700232$.  This answers a question of Jayadev Athreya.

\begin{figure}[h!]
\centering
\begin{subfigure}{0.49\textwidth}
\centering
  \includegraphics[width=.9\linewidth]{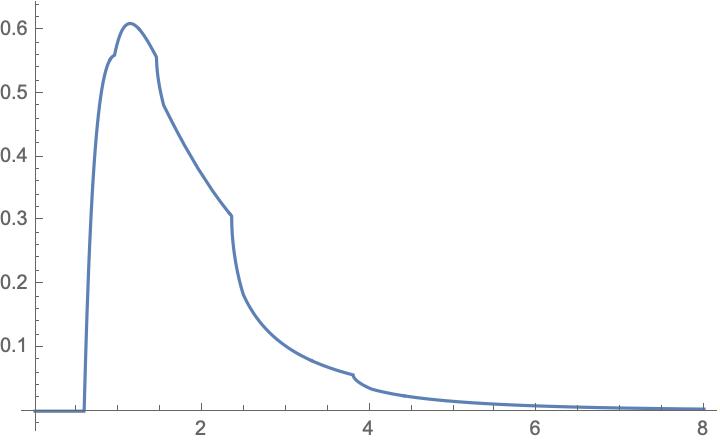}
  \caption{Limiting gaps for $O_{10}$}
  \label{O10}
\end{subfigure}
\hfill
\begin{subfigure}{0.49\textwidth}
\centering
  \includegraphics[width=.9\linewidth]{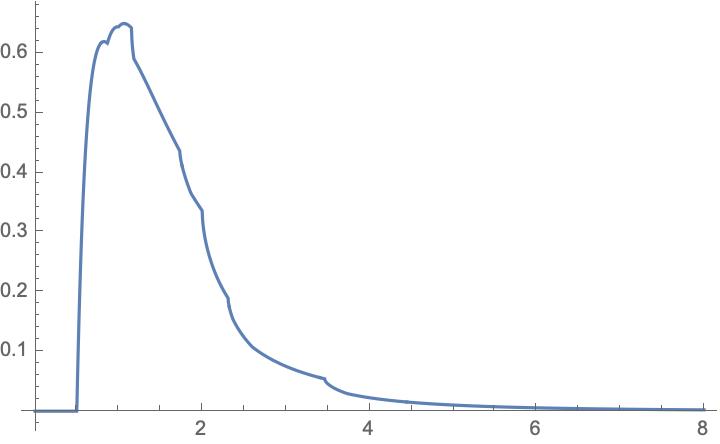}
  \caption{Limiting gaps for $O_{12}$}
  \label{O12}
\end{subfigure}
\label{trgt}
\\
\begin{subfigure}{0.49\textwidth}
\centering
  \includegraphics[width=.9\linewidth]{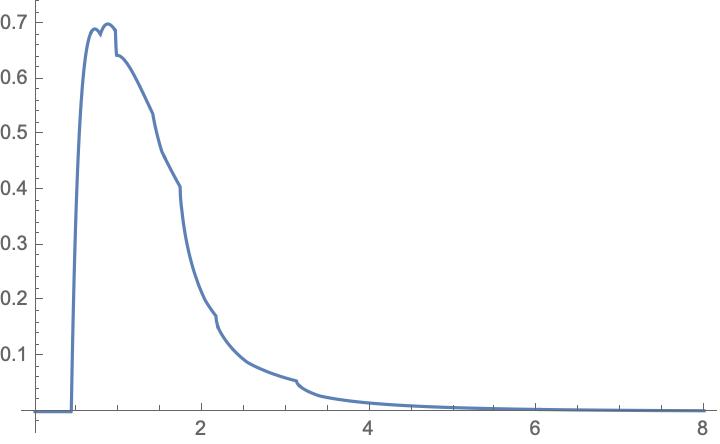}
  \caption{Limiting gaps for $O_{14}$}
  \label{O14}
\end{subfigure}
\hfill
\begin{subfigure}{0.49\textwidth}
\centering
  \includegraphics[width=.9\linewidth]{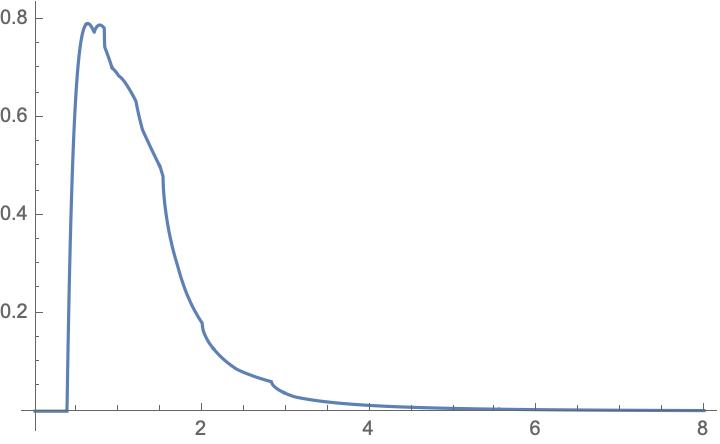}
  \caption{Limiting gaps for $O_{16}$}
  \label{O16}
\end{subfigure}
\label{trgt2}
\\
\begin{subfigure}{0.49\textwidth}
\centering
  \includegraphics[width=.9\linewidth]{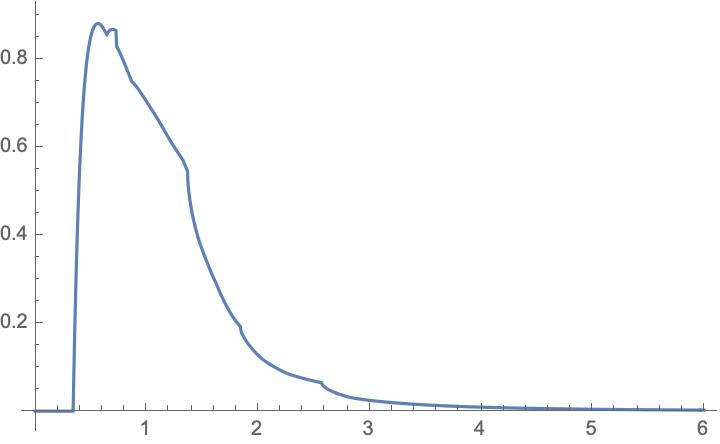}
  \caption{Limiting gaps for $O_{18}$}
  \label{O18}
\end{subfigure}
\hfill
\begin{subfigure}{0.49\textwidth}
\centering
  \includegraphics[width=.9\linewidth]{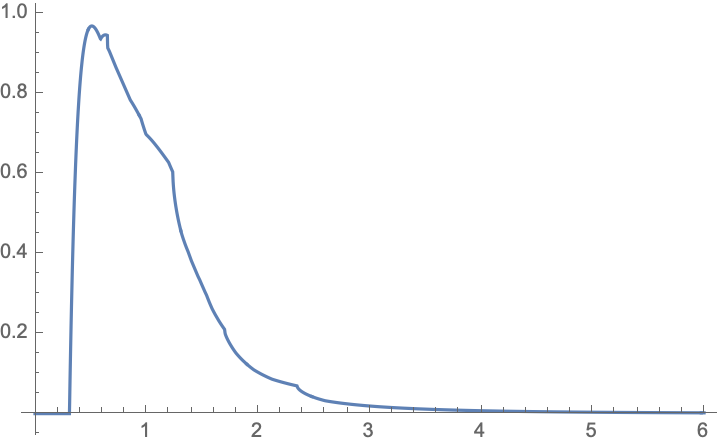}
  \caption{Limiting gaps for $O_{20}$}
  \label{O20}
\end{subfigure}
\caption{Gap distributions for the $2n$-gon for several values of $n$. Note the changed scaling on the horizontal axis in the last two graphs.}
\label{trgt3}
\end{figure}

\section{Bounds on the Number of Non-Differentiability Points}\label{sect:bounds}

Now that we have an explicit description of the slope gap distribution, we can prove Theorem \ref{linear_upper_bound}. 
Recall the theorem statement as follows.

\thmbounds*


\begin{proof}
    Since the slope gap distribution is computed by finding the area bounded by $R(x,y) = 0$ and $R(x,y) = t$ over $\Omega = \Omega_1 \cup \Omega_2$, we can see that the distribution has a non-differentiable point at $t$ only if the level set of the return time function crosses a boundary of some $P_i$ at time $t$. 
    Recall our division of $\Omega_1$ into the regions $P_1, P_2, \hdots, P_n.$ We can disregard the boundary crossings on $\Omega_2$ since they coincide with the boundary crossings for $P_1$ (they are the same region in $\mathbb{R}^2$ and have the same return time function as well). Thus, we can find an upper bound on the number of non-differentiable points by simply finding a bound on the number of boundary crossings on $\Omega_1$. The bulk of our proof will be in finding this upper bound. 
    We will then use the information we gather for computing the upper bound to get a linear lower bound in the second half of this section. 
    There are five cases we must consider.
    
    \medskip
    
    \textbf{Case 1.} We first count the number of boundary crossings in the region $P_1$. Recall the definition of \[P_1 = \{(x, y) \in \Omega_1 \ |\ x + y > 1\}\] and the return time function $R(x,y) = \frac{1}{xy}$ over $P_1$, which are displayed in Figure \ref{fig:P_1} for any $n$.
    \begin{figure}[h!]
        \centering
        \begin{tikzpicture}[scale=0.65]
    \begin{axis}[
    axis lines=middle,
    ymajorgrids=false,
    xmajorgrids=false,
    xmin=0,
    xmax=1.5,
    ymin=0,
    ymax=1.5,
    xlabel={$x$},
    ylabel={$y$},]
    \addplot[line width=1.5pt, black, domain=0.1:1.5] { 1/(3*x)};
    \draw [line width=1pt] (0,1)-- (1,1);
    \draw [line width=1pt] (1,1)-- (1,0);
    \draw [line width=1pt, dashed] (1,0)-- (0,1);
    \path[fill=violet, opacity=0.25] (0,1) -- (1,1) -- (1,0);
    
    \node [draw=none] at (0.7, 0.7) (a) {$P_1$};
    \end{axis}
    \end{tikzpicture}
        \caption{An illustration of $P_1$ with coordinates in $\R^2$, for any $n$, along with the hyperbola $R(x,y) = t$ (in black; here $t = 3$).}
        \label{fig:P_1}
    \end{figure}
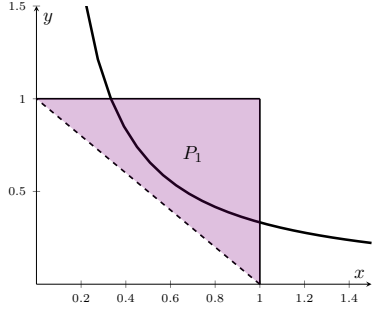
    
    As $t \to \infty$, the function $\frac{1}{xy} = t$ moves towards the origin, but never crosses the points $(0,1)$ and $(1,0)$ (the upper left and lower right corners of $P_1$). Thus, $\frac{1}{xy}$ first enters $P_1$ through the point $(1,1)$, which it crosses when \[t = 1,\] giving us a possible point of non-differentiability there. Moreover, when \[t = 4,\] $\frac{1}{xy} = t$ crosses the line $y = 1 - x$, giving us another possible point of non-differentiability. Since this function never passes through $(0,1)$ or $(1,0)$, there are no other possible boundary crossings for this region. Thus, we have \textbf{two} boundary crossings in this case.
    
    \medskip
    
    \textbf{Case 2.} We now count the number of boundary crossings in the region $P_2$. Recalling that $h_0 = v_0 = 1$ for all $n$, we have that \[P_2 = \{(x, y) \in \Omega_1\ | \ x + y \leq 1, xh_1 + yv_1 > 1, xh_1 + y > 1\}\] with $R(x,y) = \frac{1}{x(x + y)}.$ The constraints on $P_2$ give us four sides: the line $x + y = 1$, the line $xh_{1} + yv_{1} = 1$, the line $xh_1 + y = 1$, and the line $x = 1$. Denote the intersection of $x + y = 1$ and $xh_1 + y = 1$ as $A$, the intersection of $x + y = 1$ and $x = 1$ as $B$, the intersection of $xh_{1} + yv_{1} = 1$ and $x = 1$ as $C$, and the intersection of $xh_{1} + yv_{1} = 1$ and $xh_1 + y = 1$ as $D$.  Observe that for any $n$, we have $A=(0,1)$, $B=(1,0)$, and $C=(1,-1)$. Refer to the diagram in Figure \ref{fig:nondiffable_case_2} for an example of the shape of $P_2$ and the given labeling.
    
    \begin{figure}[h!]
    \centering
    \begin{tikzpicture}
    \begin{axis}[
    axis lines=middle,
    ymajorgrids=false,
    xmajorgrids=false,
    xmin=0,
    xmax=1.2,
    ymin=-1.2,
    ymax=1.2,
    xscale=0.5,
    yscale=1,
    xtick={0,0.5,1}]

    \coordinate[label=above right:$A$] (A) at (0,1);
    \coordinate[label=above right:$B$] (B) at (1,0);
    \coordinate[label=below right:$C$] (C) at (1, -1);
    \coordinate[label={[label distance=3pt]below left:$D$}] (D) at (0.357,0);

    \draw [line width=1pt, color=blue] (A)--(B);
    \draw [line width=1pt, color=blue] (B)--(C);
    \draw[line width=1pt, color=blue, dashed] (C)--(D);
    \draw[line width=1pt, color=blue, dashed] (A)--(D);
    
    \node[draw=none] at (0.5, 0.3) (label) {$P_2$};
    
    \addplot[black, line width=1.5pt, domain=0.1:1.5] {1/(3.6*x) - x};
     
     \path[fill=blue, opacity=0.25] (A) -- (B) -- (C) -- (D);
    \end{axis}
    \end{tikzpicture}
        \caption{In this example, $n = 7$. The black line is the hyperbola $R(x,y) = 3.6$.}
    \label{fig:nondiffable_case_2}
\end{figure}
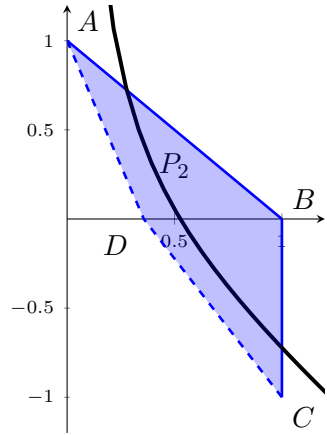

The hyperbola enters the region from the right and moves to the left as $t \to \infty$, so the hyperbola always enters the region at the point $B$. Then \[R(x,y) = \frac{1}{x(x+y)} = t\] intersects $B=(1,0)$ at $t = 1$, which does not give us a new possible non-differentiability point (since we have already accounted for $t = 1$ in Case 1). 

Since the hyperbola $R(x,y)=t$ is asymptotic to the lines $x=0$ and $y=-x$, we can see that it will never cross the points $A$ and $C$, but it will cross the lines $AD$ and $CD$ inside the region $P_2$.
Solving for the first intersection of $R(x,y) = t$ and the line $AD$ given by $xh_1 + y = 1$ gives us a quadratic equation with discriminant $t^2 - 8\cos(\pi/n)t$. This means that \[t = 8\cos(\pi/n)\] gives us the first intersection between the hyperbola and $AD$. Similarly, solving for the intersection between $R(x,y) = t$ and the line $CD$ given by $xh_1+yv_1=1$ also gives us a quadratic equation with discriminant $t^2 - 8\cos(\pi/n)t$, so the first intersections of the hyperbola with the lines $AD$ and $CD$ happen simultaneously. 

Finally, we can compute that the crossing of $D$ happens at $t = (1 + 2\cos(\pi/n))^2$. 
Thus, this case gives us at most \textbf{two} new boundary crossings: the 
simultaneous crossing of $AD$ and $DC$ and the crossing of $D$.
    
    \medskip
    
    \textbf{Case 3.} We now count the number of boundary crossings in the regions $P_i$ for $2 < i < n-1.$ Recall that in this case, \[P_i = \{(x, y) \in \Omega_1\ | \ xh_{i - 2} + yv_{i - 2} \leq 1, xh_{i - 1} + yv_{i - 1} > 1, xh_1 + y > 1\}\] with $R(x,y) = \frac{v_{i-2}}{x(h_{i-2}x + v_{i-2}y)}.$ The constraints on $P_i$ in this case give us four sides: the line $xh_{i - 2} + yv_{i - 2} = 1$, the line $xh_{i - 1} + yv_{i - 1} = 1$, the line $xh_1 + y = 1$, and the line $x = 1$. Denote the intersection of $xh_{i - 2} + yv_{i - 2} = 1$ and $xh_1 + y = 1$ as $A$, the intersection of $xh_{i - 2} + yv_{i - 2} = 1$ and $x = 1$ as $B$, the intersection of $xh_{i - 1} + yv_{i - 1} = 1$ and $x = 1$ as $C$, and the intersection of $xh_{i - 1} + yv_{i - 1} = 1$ and $xh_1 + y = 1$ as $D$. Refer to Figure \ref{fig:nondiffable_case_3} for an example of the shape of these $P_i$ and the given labeling.
    
    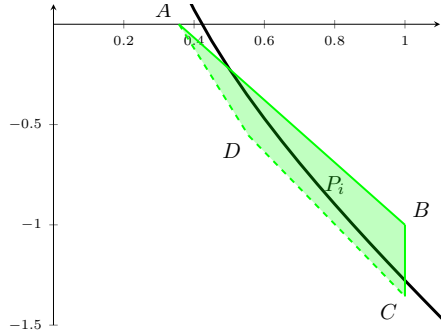
\begin{figure}[h!]
    \centering
    \begin{tikzpicture}[scale=0.75]
    \begin{axis}[
       axis lines=middle,
    ymajorgrids=false,
    xmajorgrids=false,
    xmin=0,
    xmax=1.1,
    ymin=-1.5,
    ymax=0.1,
    xscale=1,
    yscale=1]
    
    \coordinate[label=above left:{$A$}] (C) at (0.357,0);
    \coordinate[label=below left:{$D$}] (D) at (0.555,-0.555);
    \coordinate[label=below left:{$C$}] (K) at (1, -1.3569);
    \coordinate[label=above right:{$B$}] (L) at (1, -1);
    \addplot[domain=0.3:1.5, black, line width = 1.5pt] {1/(3.6*x) - 1.555*x};
    \node[draw=none] at (0.8, -0.8) (label) {$P_i$};
    \draw [line width=1pt, color=green, dashed] (C)--(D);
    \draw [line width=1pt, color=green] (K)--(L);
    \draw[line width=1pt, color=green] (C)--(L);
     \draw[line width=1pt, color=green, dashed] (D)--(K);
     
     \path[fill=green, opacity=0.25] (C) -- (L) -- (K) -- (D);
    \end{axis}
    \end{tikzpicture}
        \caption{In this example, $n = 7$ and $i = 3$. The black curve is the hyperbola $R(x,y) = 3.6$.}
    \label{fig:nondiffable_case_3}
\end{figure}
    
    The hyperbola enters the region at point $B$, which has coordinates $(1, \frac{1 - h_{i-2}}{v_{i-2}})$. We can solve to see that this happens at 
    \begin{align*}
        t = k(i,n) := v_{i-2}
        = \csc\left(\frac{\pi}{n}\right)\sin\left(\frac{\pi(i-1)}{n}\right).
    \end{align*}
    However, note that because $v_{n-i} = v_{i-2}$, the values of the $v_i$ repeat after $\floor{n/2}$ indices. Since we start at $t=v_1$ when $i=3$, this means that only the first $\floor{n/2}-1$ of the $n-4$ regions where $2<i<n-1$ can have distinct values for the time when $R(x,y)=t$ crosses $B$. Then for the remaining $n-\floor{n/2}-3$ regions we do not need to count an additional point of non-differentiability for the crossing of $B$ because we have already taken those times into account.
    
    Solving for the intersection of $R(x,y)=t$ and $AD$ gives a quadratic with discriminant $t^2 - 4t(h_1 - h_{i-2}/v_{i-2}),$ i.e., the first intersection happens at $t = 4(h_1 - h_{i-2}/v_{i-2})$. Moreover, solving for the intersection of $R(x,y)=t$ and $CD$ gives a quadratic with discriminant $t^2/v_{i-1}^2 - 4t(h_{i-1}/v_{i-1} - h_{i-2}/v_{i-2}),$ i.e., the first intersection happens at $t = 4v_{i-1}^2(h_{i-1}/v_{i-1} - h_{i-2}/v_{i-2})$. We may then verify that both of these quantities are equal to \[t= l(i,n) := 4\sin(i\pi/n)\csc(\pi(i - 1)/n),\] meaning that the hyperbola crosses $AD$ and $CD$ at the same time. 
    
    Observe, however, that the crossing of the lines $AD$ and $CD$ only counts as a boundary crossing for the region $P_i$ if it happens $\textit{inside}$ the region i.e. if it occurs after the hyperbola has already crossed $B$. This is equivalent to the condition that $k(i,n)<l(i,n)$. For small $n$, this is indeed the case for all $2<i<n-1$. However for large $n$, we find that $l(i,n)<k(i,n)$ for all but a small number of $2<i<n-1$.  In fact, one can show that $l(i,n)<k(i,n)$ for $7\leq i\leq n-2$ for all $n\geq 13$. We will take this into consideration when making our final count of boundary crossings arising from these regions.
    
    We now note that since $A$ is the intersection of $xh_{i - 2} + yv_{i - 2} = 1$ and $xh_1 + y = 1$, it has coordinates \[\left(\frac{v_{i-2}-1}{h_{i-2}-h_1v_{i-2}}, 1 - \frac{h_1(v_{i-2}-1)}{h_{i-2}-h_1v_{i-2}}\right).\] Similarly, we find that $C$ has coordinates \[\left(1, \frac{1}{v_{i-1}} - \frac{h_{i-1}}{v_{i-1}}\right).\] We may then compute that $R(x,y)=t$ crosses $A$ and $C$ simultaneously at \[t = m(i,n) := \frac{\csc^2(\pi/n)\sin((i-1)\pi/n)}{\cot(\pi/n) - \cot(i\pi/(2n))}.\]
    
    

    We can similarly solve to see that the hyperbola crosses point $D$ at \[t = r(i,n) := \frac{\csc(\pi/n)\sin(\pi(i-1)/n)\sin^2(\pi(i+1)/n)}{(\sin(i\pi/n)-\sin(\pi/n))^2}.\]
    
    Thus for $n\geq 13$ and $3\leq i \leq 6$, we have at most \textbf{four} total new crossings: a crossing at $B$, a simultaneous crossing of $AD$ and $CD$, a simultaneous crossing of $A$ and $C$, and a crossing of $D$. For $6\leq i \leq \floor{n/2}+1$, we have at most \textbf{three} new crossings, since the crossing of the lines $AD$ and $CD$ happens before the hyperbola reaches the region $P_i$. Finally, for $\floor{n/2}+2\leq i<n-1$, we have at most \textbf{two} new crossings, since we do not count the crossing at $B$.
    
    Although we have assumed here that $n\geq 13$, we may check directly that the resulting upper bound we obtain applies even when $4\leq n <13$. The true number of non-differentiable points for $4\leq n <12$ (calculated using numerical methods in Mathematica) is given in the table in Figure \ref{fig:tableofnondiffablepoints}, along with our upper bound. Unfortunately, calculating the true number of non-differentiable points for $n\geq 12$ exceeds our computational capacity.  However, for $n =12$ we may verify that the number of distinct time stamps corresponding to boundary crossings is $29$, which is less than our upper bound.  Since the number of non-differentiable points is bounded above by the number of boundary crossings, this is enough to show that our upper bound holds in all cases $n\geq 4$. Formulas for the time stamps corresponding to boundary crossings for each region $P_i$ for a given $n$ can be found in the table in Figure \ref{stamps}.

    
    \medskip

    \textbf{Case 4.} We now count the number of boundary crossings in $P_{n-1}$. Our region is \[P_{n-1} = \{(x, y) \in \Omega_1\  |\  xh_{2} + yv_{1} \leq 1, xh_1 + y > 1\}\] with return time given by \[R(x,y) = \frac{v_{1}}{x(h_{2}x + v_{1}y)}.\] Our region is bounded by three lines: $xh_2 + yv_1 = 1$, $xh_1 + y = 1$, and $x = 1$. Label the intersection of $xh_2+yv_1 = 1$ and $xh_1 + y=1$ as $A$, the intersection of $xh_2+yv_1=1$ and $x=1$ as $B$, and the intersection of $xh_1 + y = 1$ and $x = 1$ as $C$. Refer to Figure \ref{fig:nondiffable_case_6} for an example of the shape of this region and the given labeling.

    \begin{figure}[h!]
    \centering
    \begin{tikzpicture}
    \begin{axis}[
    axis lines=middle,
    ymajorgrids=false,
    xmajorgrids=false,
    xmin=0.6,
    xmax=1.1,
    ymin=-1.9,
    ymax=-1.2,
    xscale=1,
    yscale=1,
    xtick={0.7, 0.85, 1}]
    \node[draw=none] at (0.95,-1.4) (label) {$P_{n-1}$};
    \coordinate[label=left:{$A$}] (F) at (0.802, -1.247);
    \coordinate[label=right:{$C$}] (G) at (1, -1.802);
    \coordinate[label=right:{$B$}] (I) at (1, -1.692);

    \draw [dashed, line width=1pt, color=orange] (F)--(G);
    \draw [line width=1pt, color=orange] (F)--(I);
    \draw [line width=1pt, color=orange] (G)--(I);
    \addplot[line width=1.5pt, color = black, domain=0.83:1] {1/(2*x)-2.247*x};
     \path[fill=orange, opacity=0.5] (F) -- (I) -- (G);
    \end{axis}
    \end{tikzpicture}
        \caption{This is the example of $n = 7$. The black curve is the equation $R(x,y) = 2$.}
    \label{fig:nondiffable_case_6}
\end{figure}
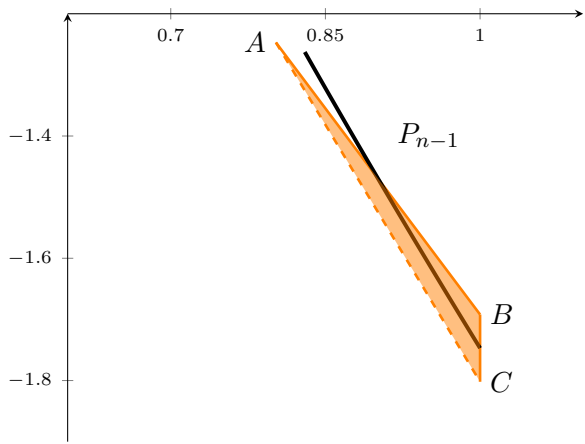
The hyperbola first enters the region at $B$. We find that this happens at $t = v_1$, which has already been accounted for in Case 3. Moreover, the same calculation we used in Case 3 to show that the hyperbola passes through $A$ and $C$ at the same time still applies in this case. 
We may then calculate that the hyperbola passes through the line $AC$ at $t=2\sec(\pi/n)$.

Thus, this case gives us at most \textbf{two} new boundary crossings: a crossing of the line $AC$ and the simultaneous crossing of $A$ and $C$.
    
    \medskip
    \textbf{Case 5.} We finally count the number of boundary crossings in $P_n$. Recall that \[P_n = \{(x, y) \in \Omega_1\ |\ xh_1 + y \leq 1\}\] with $R(x,y) = \frac{v_0}{x(h_1x + v_0y)} = \frac{1}{x(h_1x + y)}.$ Thus, our region has three boundaries: the line $xh_1 + y = 1$, the line $x = 1$, and the line $2(1+\cos(\pi/n))x + y = 1$. Figure \ref{fig:P_n} gives an example of what this region looks like.
    \begin{figure}[h!]
    \centering
    \begin{tikzpicture}[scale=0.4]
    \begin{axis}[
    axis lines=middle,
    ymajorgrids=false,
    xmajorgrids=false,
    xmin=0,
    xmax=1,
    ymin=-3,
    ymax=1.1,
    xscale=1,
    yscale=2]
    
    \addplot[line width = 1.5pt, black, domain=0.2:1] {1/(x*3) - 2.80194*x};
    
    \coordinate (A) at (0,1);
    \coordinate (G) at (1, -1.802);
    \coordinate (H) at (1, -2.802);

    \draw [line width=1pt, color=red] (A)--(G);
    \draw [line width=1pt, color=red, dashed] (A)--(H);
    \draw [line width=1pt, color=red] (G)--(H);
    
     \path[fill=red, opacity=0.25] (A)--(G)--(H);
    \end{axis}
    \end{tikzpicture}
        \caption{The region $P_n$ with $n = 7$ with the graph of the equation $R(x,y) = 3$.}
    \label{fig:P_n}
\end{figure}
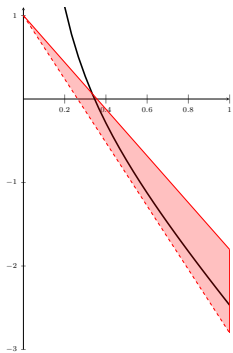
    
We find that the three corners of the region are $(1,1-h_1)$, $(0,1)$ and $(1, 1 - 2(1 + \cos(\pi/n))$. We can easily compute to see that the hyperbola enters the region by passing through $(1, 1-h_1)$ first at $t = 1$, which does not contribute a new time. 
Moreover, we can verify that the hyperbola does not pass through $(0,1)$ and $(1, 1 - 2(1 + \cos(\pi/n))$ because substituting these points into $R(x,y)$ yields an undefined value.
Finally, we may then solve to see that the hyperbola passes through the line $2(1 + \cos(\pi/n))x + y = 1$ at $t=4$, which also does not contribute a new point of non-differentiability. 

Thus, in this case we have \textbf{no} new boundary crossings.

\medskip

\textbf{Finishing Touches.} We now add up all the boundary crossings we have counted. In Case 1, we found at most $2$ boundary crossings. In Case 2, we found at most $2$ new boundary crossings as well. In Case 3, we found at most $4$ new crossings for each $2<i\leq 6$, of which there are 4, at most $3$ new crossings for
$6< i \leq \floor{n/2}+1$, of which there are $n-\floor{n/2}-5$, and at most $2$ new crossings for $\floor{n/2}+2\leq i <n-1$, of which there are $n-\floor{n/2}-3$.
In Case 4, we found at most $2$ new boundary crossings, and in Case 5 we found no new boundary crossings. Adding these values together, we have our desired upper bound of 
\begin{align*}
    \#\text{(Non-Differentiability Points)} &\leq 2 + 2 + 4\cdot 4+ 3\left(\left\lfloor\frac{n}{2}\right\rfloor-5\right)  +2\left(n - \left\lfloor \frac{n}{2} \right\rfloor - 3\right) + 2 \\
    &\leq 2n+\left\lfloor \frac{n}{2} \right\rfloor+1.
\end{align*}

\textbf{Lower Bound.} In order to compute a lower bound, we must understand the possible ways in which our upper bound has over counted the true number of non-differentiability points. First, we may have counted boundary crossings that do not actually happen. For example, in Case 3 we were somewhat conservative in our estimate for when the crossing of $AD$ and $CD$ happens outside of the region $P_i$, and for larger $n$ the range of $i$ for which $m(i,n)<k(i,n)$ can likely be expanded. Second, we could have over counted times where the boundary crossings of two different regions happen to coincide. We have already accounted for this in a number of cases, but there could be additional cases of simultaneous crossings that we have not taken into account.

If this was the only concern, then it would suffice to show the existence of an infinite family of distinct time stamps of boundary crossings.  However, it is \textit{a priori} possible that two or more simultaneous boundary crossings for different regions could "cancel out" an apparent point of non-differentiability, leading to the final distribution actually being differentiable at that point. Thus to prove our lower bound, we must demonstrate the existence of an infinite family of distinct boundary crossings that do not cancel out with any other boundary crossings coming from different regions.

The time stamps where the boundary crossings occur in each case are summarized in Figure \ref{stamps}.

\begin{figure}[h!]
\begin{center}
\renewcommand{\arraystretch}{1.5}
\begin{tabular}{ |c|c|} 
 \hline
 \textbf{Region} & \textbf{Time Stamps}\\ 
  \hline
 $P_1$& $t=1$, $t=4$ \\ 
  \hline
$P_2$ & $t=1$, $t=8\cos(\pi/n)$, $t=(1+2\cos(\pi/n))^2$
\\ 
  \hline
  $P_i$,\quad $2<i<n-1$
  & $t=k(i,n):=\csc\left( \frac{\pi}{n}\right)\sin\left( \frac{\pi(i-1)}{n}\right)$, \\ 
& 
$t=l(i,n):=4\sin(i\pi/n)\csc(\pi(i - 1)/n)$, \\
  & 
$t=m(i,n):=\frac{\csc^2(\pi/n)\sin(\pi(i-1)/n)}{\cot(\pi/n) - \cot(i\pi/(2n))}$,\\ 
&  
  $t=r(i,n):= \frac{\csc(\pi/n)\sin(\pi(i-1)/n)\sin^2(\pi(i+1)/n)}{(\sin(i\pi/n)-\sin(\pi/n))^2}$
  \\ 
    \hline
    $P_{n-1}$ & $t=\csc(\pi/n)\sin(2\pi/n)$, $t=m(n-1,n)$, and $t=2\sec(\pi/n)$ \\ 
     \hline
     $P_{n}$& $t=1$, $t=4$ \\ 
 \hline
\end{tabular}
\end{center}
\caption{Time stamps for boundary crossings.}  
\label{stamps}
\end{figure}

 
Experimental evidence based on computations in Mathematica suggests that the $\floor{n/2}-1$ distinct times coming from $k(i,n)$ and the $n-4$ times coming from $m(i,n)$ and $r(i,n)$, respectively, are all mutually distinct from each other for large $n$. Since there are only a finite number of remaining time stamps which could potentially cancel out a point of non-differentiability, this suggests a conjectural lower bound for the number of non-differentiability points of order $2n+\floor{n/2}$, which is comparable to the upper bound. 

In what follows we will provide a reasoning for the more modest lower bound. First, we will show that there are a large number of time stamps coming from $m(i,n)$ that are distinct from all time stamps coming from $k(i,n)$. Suppose that there are coinciding time stamps from these functions.  That means that there are integers $3\leq i \leq \floor{n/2} + 1$ and $3\leq j \leq n-2$ such that 
\begin{align*}
k(i,n) &= m(j,n)\\
\csc\left(\frac{\pi}{n}\right)\sin\left(\frac{\pi(i-1)}{n}\right) &= \frac{\csc^2(\pi/n)\sin(\pi(j-1)/n)}{\cot(\pi/n)-\cot(\pi j/(2n))}.
\end{align*}
Solving for $i$, we get
\begin{align*}
i =f(j,n):= 1+\frac{n}{\pi}\arcsin\left(\frac{\csc(\pi/n)\sin(\pi(j-1)/n)}{\cot(\pi/n)-\cot(\pi j/(2n))}\right)
\end{align*}
where we only need to consider the usual branch of arcsine since $3\leq i \leq \floor{n/2}+1$ by assumption. An investigation of the function $f(j,n)$ in Mathematica suggests that for large $n$, the function is quite close to $j+2$ in the first half of its domain and it is close to $n-j+2$ in the second half of its domain. We claim that for any $n\geq 38$ we have
\[
0<f(j,n)-j-2<1
\]
for $5\leq j \leq n/5$ and for $n\geq 5$ we have 
\[
0<f(j,n)-n+j-2<1
\]
for $4n/5 \leq j \leq n-2$.  This implies that there is no integer solution to the equation $i = f(j,n)$ for $j$ in these domains, which implies that the time stamps coming from $m(j,n)$ for these $j$ cannot coincide with any of the time stamps coming from $k(i,n)$.

First consider the inequality $0<f(j,n)-j-2$. Observing that sine is an increasing function in the stated region, we can rearrange this to
\begin{align*}
    \sin\left(\frac{\pi(j+1)}{n}\right) &< \frac{\csc(\pi/n)\sin(\pi(j-1)/n)}{\cot(\pi/n)-\cot(\pi j/(2n))}\\
    &= \frac{\sin(\pi(j-1)/n)\sin(\pi j/(2n))}{\sin(\pi j/(2n))\cos(\pi/n)-\cos(\pi j/(2n))\sin(\pi/n)}\\
    &= \frac{\sin(\pi(j-1)/n)\sin(\pi j/(2n))}{\sin(\pi (j-2)/(2n))}
\end{align*}
where we have used a trigonometric identity to simplify the denominator. Since sine is positive in this region, this is equivalent to 
\begin{align*}
    \sin\left(\frac{\pi(j+1)}{n}\right)\sin\left(\frac{\pi (j-2)}{2n}\right) &< \sin\left(\frac{\pi(j-1)}{n}\right)\sin\left(\frac{\pi j}{2n}\right)\\
    \cos\left(\frac{\pi(j+4)}{2n}\right) - \cos\left(\frac{3\pi j}{2n}\right) &< \cos\left(\frac{\pi(j-2)}{2n}\right) - \cos\left(\frac{\pi(3j-2)}{2n}\right)\\
    \cos\left(\frac{\pi(3j-2)}{2n}\right) - \cos\left(\frac{3\pi j}{2n}\right) &< \cos\left(\frac{\pi(j-2)}{2n}\right) - \cos\left(\frac{\pi(j+4)}{2n}\right)\\
    \sin\left(\frac{\pi(3j-1)}{2n}\right)\sin\left(\frac{\pi}{2n}\right) &< \sin\left(\frac{\pi(j+1)}{2n}\right)\sin\left(\frac{3\pi}{2n}\right).
\end{align*}
Now, using the Taylor series for sine, we may say that
\begin{align*}
    \sin\left(\frac{\pi(3j-1)}{2n}\right)\sin\left(\frac{\pi}{2n}\right) < \frac{\pi(3j-1)}{2n}\cdot\frac{\pi}{2n}
\end{align*}
as well as
\begin{align*}
    \left(\frac{\pi(j+1)}{2n}-\frac{1}{6}\left(\frac{\pi(j+1)}{2n}\right)^3\right)\cdot\left(\frac{3\pi}{2n}-\frac{1}{6}\left(\frac{3\pi}{2n}\right)^3\right) < \sin\left(\frac{\pi(j+1)}{2n}\right)\sin\left(\frac{3\pi}{2n}\right)
\end{align*}
so long as $0< j < \frac{2\sqrt{6}n}{\pi}-1$.  It is then straightforward to check that 
\begin{align*}
    \frac{\pi(3j-1)}{2n}\cdot\frac{\pi}{2n} <
    \left(\frac{\pi(j+1)}{2n}-\frac{1}{6}\left(\frac{\pi(j+1)}{2n}\right)^3\right)\cdot\left(\frac{3\pi}{2n}-\frac{1}{6}\left(\frac{3\pi}{2n}\right)^3\right)
\end{align*}
for $n\geq4$ and $j\geq 3$.

A similar calculation shows that the inequality $f(j,n)-j-2<1$ is equivalent to 
\begin{align*}
    \sin\left(\frac{\pi(j+2)}{2n}\right)\sin\left(\frac{2\pi}{n}\right) &< \sin\left(\frac{3\pi j}{2n}\right) \sin\left(\frac{\pi}{n}\right).
\end{align*}
Again using the Taylor series approximations, we know that the above holds whenever
\begin{align*}
    \frac{\pi(j+2)}{2n}\cdot\frac{2\pi}{n} <
    \left(\frac{3\pi j}{2n}-\frac{1}{6} \left(\frac{3\pi j)}{2n}\right)^3\right)\cdot\left(\frac{\pi}{n}-\frac{1}{6}\left(\frac{\pi}{n}\right)^3\right)
\end{align*}
which is true for $5\leq j \leq n/5$ when $n\geq 38$. 

The inequality $0<f(j,n)-n+j-2$ can be reduced to the inequality
\begin{align*}
    \sin\left(\frac{\pi(j-2)}{2n}\right) < \sin\left(\frac{\pi j}{2n}\right) 
\end{align*}
which is true since sine is an increasing function for $4n/5 \leq j \leq n-2$. Similarly, the inequality $f(j,n)-n+j-2<1$ is equivalent to 
\begin{align*}
    \cos\left(\frac{\pi(3j-6)}{2n}\right) < \cos\left(\frac{\pi(3j-2)}{2n}\right) 
\end{align*}
which is again true since cosine is increasing for $j$ in the stated region, so long as $n\geq 5$. This demonstrates the above claim. For $4\leq n <38$, we can verify by direct computation that $k(i,n)\neq m(j,n)$ for any $3\leq i \leq\floor{n/2}+1$ and $5\leq j \leq n/5$ or $4n/5 \leq j \leq n-2$.

This means that only time stamps coming from $m(j,n)$ for $j=3,4$ or $n/5<j<4n/5$ could possibly coincide with any of the time stamps coming from $k(i,n)$. Since $k(i,n)$ is a sinusoid with maximum at $\floor{n/2}+1$, we note that $k(i,n)\leq k(j,n)$ for any $3\leq i \leq n/5$ and $n/5 \leq j\leq 4n/5$. Moreover, since $k(i,n)$ represents the time at which the hyperbola for region $P_i$ crosses the point $B$ and $m(i,n)$ represents the time at which it crosses the points $A$ and $C$, we may observe that $k(i,n)<m(i,n)$ for any $i$ and $n$. Combining these bounds, we see that
\[
k(i,n)\leq k(j,n) < m(j,n)
\]
for any $3\leq i \leq n/5$ and $n/5 \leq j\leq 4n/5$. This means that for $3\leq i \leq n/5$, the function $k(i,n)$ (which is strictly increasing in this region) could only possibly coincide with either $m(3,n)$ or $m(4,n)$. In other words, there are at least $n/5 - 4$ values of $k$ that do not coincide with \textit{any} value of $m$. We also note that since the hyperbola always crosses the point $B$ for each region $P_i$, all of these times represent true boundary crossings.

To finish the proof of our lower bound, we now observe that in order for two simultaneous boundary crossings to "cancel out" an apparent point of non-differentiability, they would have to have opposing effects on the derivative of the the probability distribution, i.e. one would have to represent a sudden \textit{increase} in the rate of area accumulation as the hyperbola sweeps through its corresponding transversal region, while the other would have to represent a sudden \textit{decrease} in the rate area accumulation for its region. Recalling their definitions in Case 3, we find that the functions $k(i,n)$ and $r(i,n)$ represent an \textit{increase} in the rate of area accumulation, while $l(i,n)$ and $m(i,n)$ represent a \textit{decrease} in the rate of area accumulation for the region $P_i$. Recall also that for $n\geq 13$, the function $l(i,n)$ represents at most $4$ true boundary crossings, as the rest occur outside of the transversal region. We furthermore find that there are a total of $3$ boundary crossings of \textit{decreasing} type coming from Cases 1, 2, 4, and 5. 

Thus, the $n/5 -4$ (\textit{increasing}-type) values of $k(i,n)$ which we have already argued are distinct from any of the values of $m(i,n)$ can only possibly cancel with finitely many \textit{decreasing}-type values coming from other regions. Assuming that all possible cancellations happen, we get our lower bound of 
\begin{align*}
    \#\text{(Non-Differentiability Points)} &\geq \left(\frac{n}{5} - 4\right) - 4 -3=  \frac{n}{5}-11.
\end{align*}

Combining these two bounds, we conclude:
\[\frac{n}{5}-11 \leq \#\text{(Non-Differentiability Points)} \leq  2n+ \left \lfloor \frac{n}{2} \right \rfloor +1.\]
\end{proof}

The plot in Figure \ref{fig:plotofnondiffablepoints} and the table in Figure \ref{fig:tableofnondiffablepoints} show our upper bound against the actual number of non-differentiable points on the distribution (computed using numerical methods in Mathematica). We can see that our bound is a reasonable estimate for the number of non-differentiable points and appears to grow at roughly the same rate. In any case, we have shown that there is a linear upper and lower bound on the number of non-differentiable points for the slope gap distribution of the $2n$-gon.

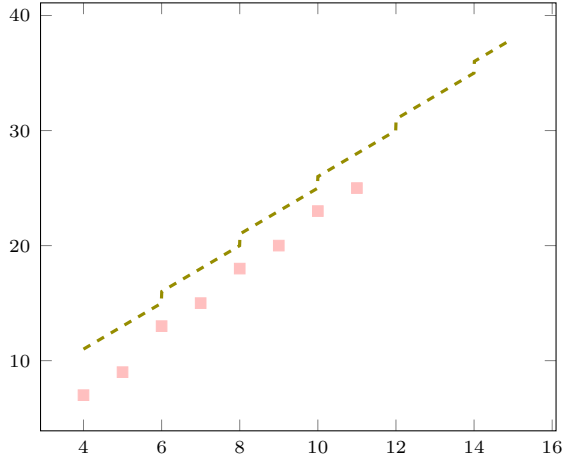
\begin{figure}
    \centering
    \begin{tikzpicture}
\begin{axis}[%
scatter/classes={%
         a={mark=*,cyan},%
		b={mark=square*,pink}}, xscale=1, yscale=1]
\addplot[scatter,only marks,%
    scatter src=explicit symbolic]%
table[meta=label] {
x y label
4 7 b
5 9 b
6 13 b
7 15 b
8 18 b
9 20 b
10 23 b
11 25 b
    };
\addplot [
    domain=4:15,
    samples=1000,
    dashed, very thick, olive
] {floor(x/2)+2*x +1};
\end{axis}
\end{tikzpicture}
    \caption{Plot of the upper bound from Theorem \ref{linear_upper_bound} (olive dashed line) and the actual number of non-differentiable points (pink squares) against $n$.}
    \label{fig:plotofnondiffablepoints}
\end{figure}


\begin{center}
\begin{figure}
\begin{tabular}{ |c|c|c| } 
 \hline
 $n$ & Upper Bound & \# of Non-Differentiable Points \\
 \hline
 \hline
 4 & 11 & 7\\
 \hline
 5 & 13 & 9\\ 
 \hline
 6 & 16 & 13\\ 
 \hline
 7 & 18 & 15 \\ 
 \hline
 8 & 21 & 18 \\
 \hline
 9 & 23 & 20 \\ 
 \hline
 10 & 26 & 23 \\ 
 \hline
 11 & 28 & 25 \\ 
 \hline
 12 & 31 & -- \\
 \hline
 13 & 33 & -- \\
 \hline
 14 & 36 & -- \\
 \hline
 15 & 38 & --\\
 \hline
\end{tabular}
    \caption{Table of the upper bound from Theorem \ref{linear_upper_bound} and number of non-differentiable points for small $n$.}
    \label{fig:tableofnondiffablepoints}
\end{figure}
\end{center}

\begin{appendix}

\section{Return Time Function Computations}\label{BigPropAppendix}

\begin{proof}[Proof of Proposition \ref{big prop}]

We begin with an outline of how the proof proceeds.
We must show that at every point in $\Omega_1$, the winning vector is the one we assert it to be in Proposition \ref{big prop}; the winning vector will dictate the return time at that point.
We begin by finding the set of all saddle connection holonomy vectors which win at some $(x, y) \in \Omega_1$.
By Lemma \ref{check_on_x=1}, this is exactly the set of vectors which win at points $(1, y) \in \Omega_1$, meaning that it suffices to find the winning vectors for points along the right edge ($x = 1$) of $\Omega_1$.
We claim that the vectors specified in Proposition \ref{big prop} win at these points (i.e. $\lambda_i$ wins at points $(1, y) \in P_i$).
To prove this, we establish conditions which any other vector must satisfy in order to win over our proposed vector in each region $P_i$, then we prove by exhaustive casework that for each $P_i$ there is no saddle connection holonomy vector satisfying these conditions.
This proves that our candidate vectors (those specified in Proposition \ref{big prop}) are the only vectors which win at points $(1, y) \in \Omega_1$, and hence at any point in $\Omega_1$.
Finally, we use the definition of a winning vector to describe the regions where each member of this finite set of vectors (the $\lambda_i$ vectors) wins, fully determining the return time function on $\Omega_1$.

\vspace{10pt}

We will first establish some useful lemmas, then we will address the 4 cases provided in the lemma separately.

Take an arbitrary saddle connection holonomy vector $\cvec{a}{b}$ and an arbitrary point $(x,y) \in \Omega_1$. First notice that the three restrictions $ax + by > 0$, $\text{slope}(M_{x, y}v) > 0$, and $x > 0$ together necessitate that $b > 0$, so to find the winning vector at $(x,y)$ (i.e., the vector with smallest slope on $M_{x,y} \mathcal{S}$), we need only consider saddle connections whose holonomy vectors have positive horizontal component.

Second, we make the observation that $M_{x,y}$ preserves the ordering of slopes. Intuitively, this follows because $M_{x,y}$ has positive determinant (so it is orientation preserving), and the image under $M_{x,y}$ of any vector with positive $x$-coordinate also has positive $x$-coordinate. Formally, for $v = \cvec{a}{b}$ and $w = \cvec{c}{d}$, if $\text{slope}(v) \leq \text{slope}(w)$, then we have:
\begin{align*}\text{slope}(M_{x, y}v) - \text{slope}(M_{x, y}w) &= \frac{bx^{-1}}{ax + by} - \frac{dx^{-1}}{cx + dy} \\[10pt] &=
\frac{bc + bdyx^{-1} - ad - bdyx^{-1}}{(ax + by)(cx + dy)} \\[10pt] &= \frac{bc - ad}{(ax + by)(cx + dy)} \\[10pt] &=
 \frac{ac(b/a - d/c)}{(ax + by)(cx + dy)} \leq 0\end{align*}
Thus, $\text{slope}(v) \leq \text{slope}(w)$ implies $\text{slope}(M_{x, y}v) \leq \text{slope}(M_{x, y}w)$, so to determine if $w$ wins over $v$, we can directly compare slopes. With these observations in mind, we proceed to our lemmas.

%
%
%
%

\begin{lem}\label{check_on_x=1}
    The set of saddle connections which win at some $(x, y) \in \Omega_1$ is exactly the set of saddle connections which win at some $(1, y) \in \Omega_1$.
\end{lem}

%
%
%
%

\begin{proof}
    Suppose the saddle connection with holonomy vector $ v_0 = \begin{pmatrix}a_0 \\ b_0 \end{pmatrix}$ is the saddle connection whose image under some $K = M_{x_0,y_0}$ with $(x_0,y_0) \in \Omega_1$ has the smallest positive slope of any image of a saddle connection with horizontal length at most $1$.
    We claim $v_0$ is also the vector satisfying this property for the matrix
    \[K' = M_{1, y_0 - (1 - x_0)a_0/b_0} = \begin{pmatrix}1 & y_0 - (1 - x_0)a_0/b_0 \\ 0 & 1 \end{pmatrix}.\] Observe that the matrix $K'$ corresponds to the point where the line with slope $a_0/b_0$ passing through the point $(x_0,y_0)$ intersects the line $x=1$. In other words, we will show that $v_0$ must win at this intersection point as well. We first provide a high-level outline of our argument for why this claim should be true.  Since $v_0$ wins at $(x_0,y_0)$, we know that $Kv_0$ has horizontal component with length at most $1$. Then, the point $(1, y_0 - (1 - x_0)a_0/b_0)$ has the property that $K'v_0$ has horizontal component with length at most $1$ as well, so $v_0$ is a valid candidate at this new point. Any vector $v$ winning over $v_0$ at $(1, y_0 - (1 - x_0)a_0/b_0)$ must have greater slope than $v_0$, so a straightforward computation shows that $K'v$ has a longer horizontal component than $Kv$. Since $v$ cannot win at $(x_0,y_0)$, this means that $Kv$ has horizontal component longer than $1$, contradicting the fact that the horizontal component of $K'v$ is at most $1$.
    
    We proceed to our complete argument. First we must show that $(1, y_0 - (1 - x_0)a_0/b_0) \in \Omega_1$. In order for $v_0$ to be a candidate winning vector at $(x_0,y_0)$, it must satisfy the condition that $M_{x_0,y_0}v_0$ has $x$-coordinate between $0$ and $1$, hence it must be true that $a_0x_0 + b_0y_0 \leq 1$.
    We also know from the boundaries of $\Omega_1$ that $y_0 > -2(1 + \cos(\pi/n))x_0 + 1$, and we know that because the smallest vertical edge in $S'$ is $1$, $b_0 \geq 1$.
    Using these inequalities, we find:
    \begin{align*}a_0x_0 -2(1 + \cos(\pi/n))b_0x_0  &\leq a_0x_0 -2(1 + \cos(\pi/n))b_0x_0 + b_0 - 1  \\[10pt] &< a_0x_0 + b_0y_0 - 1 \leq 0, \end{align*} and \[a_0/b_0 > -2(1 + \cos(\pi/n)).\]
    We also know $a_0/b_0$ is positive, so $-a_0/b_0 < 0$.
    Thus, $-a_0/b_0$ is between the slopes of the lower and upper boundaries of $\Omega_1$, so the line of slope $-a_0/b_0$ passing through $(x_0, y_0)$ will intersect the right edge of $\Omega_1$ (which is included in $\Omega_1$) at $x = 1$; this point is exactly $(1, y_0 - (1 - x_0)a_0/b_0)$.
    
    Now we show that $K'v_0$ is the saddle connection of horizontal length at most $1$ with the smallest positive slope.
    Since the image of $v_0$ under $K$ has horizontal length at most 1, we know
    \[x_0a_0 + y_0b_0 \leq 1\]
    Well, the same must be true for the horizontal length of the image of $v_0$ under $K'$. Henceforth we will denote the horizontal component of $Kv_0$ as $c$, and we will denote the horizontal component of $K'v_0$ as $c'$.
    \[c' = 1\cdot a_0 + (y_0 - (1 - x_0)a_0/b_0)\cdot b_0 = x_0a_0 + y_0b_0 \leq 1,\] so $K'v_0$ has horizontal length at most $1$ as well (making it a valid candidate).
    Now let us consider the set of vectors $v$ whose image under $K'$ may have a smaller positive slope than $K'v_0$.
    Both $K$ and $K'$ are shear matrices which preserve the order of slopes of vectors, so $v$ must have smaller slope than $v_0$, and $Kv$ must have smaller slope than $Kv_0$.
    In other words, for $v = \cvec{a}{b}$,
    it must be true that
    \[b/a < b_0/a_0 \implies a_0b/b_0 < a\]
    Let $d := ax_0 + by_0$ denote the horizontal component of $Kv$. Now consider the horizontal component of $K'v$ (which we will denote by $d'$):
    \begin{align*}d' &= 1\cdot a + (y_0 - (1 - x_0)a_0/b_0)\cdot b \\[10pt] &= (a - a_0b/b_0)(1 - x_0) + ax_0 + by_0\\[10pt] &= |a - a_0b/b_0|\cdot|1 - x_0| + ax_0 + by_0 \\[10pt] &> ax_0 + by_0 = d\end{align*}
    Well, since $v_0$ is the saddle connection whose image under $K$ has the smallest positive slope of any image of a saddle connection with horizontal length at most 1, and since $Kv$ has smaller slope than $Kv_0$, it must be true that the horizontal component of $Kv$ is greater than 1, so
    \[d' > d > 1,\] contradicting the fact that the horizontal length of $K'v_0$ must be at most $1$.
    Thus there is no vector $v$ such that $K'v$ has smaller slope than $K'v_0$ and $K'v$ has horizontal component at most $1$.
    In other words, if $v_0$ is the saddle connection whose image under some $K = M_{x_0,y_0}$ with $(x_0,y_0) \in \Omega_1$ has the smallest positive slope of any image of a saddle connection with horizontal length at most 1, then $v_0$ also satisfies this property for some $K' = M_{1, y}$ with $(1, y) \in \Omega_1$.
\end{proof}

For the sake of brevity, we introduce some new notation.
Let $\sigma_i$ for $1 \leq i \leq \lceil n/2 \rceil - 1$ represent the saddle connections joining the lower left and upper right vertices on the rectangles $H_i$ and $\nu_i$ for $0 \leq i \leq \lfloor n/2 \rfloor - 1$ represent the saddle connections joining the lower left and upper right vertices on on the rectangles $V_i$, as shown in the diagram in Figure \ref{fig:sigma_label} for $n = 7$.

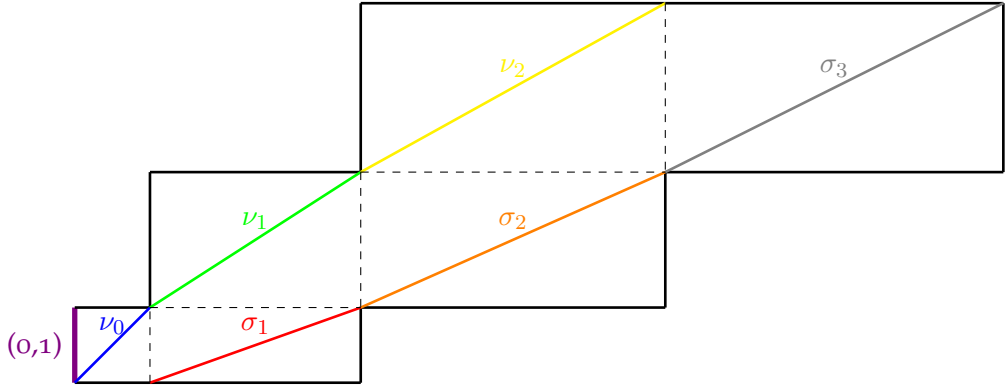
\begin{figure}[h!]
    \centering
        \begin{tikzpicture}
    \coordinate (a) at (0,0);
    \coordinate (b) at (0,1);
    \coordinate (c) at (1,0);
    \coordinate (d) at (1,1);
    \coordinate (e) at (3.801937735804838,0);
    \coordinate (f) at (3.801937735804838,1);
    \coordinate (g) at (1,2.801937735804838);
    \coordinate (h) at (3.801937735804838,2.801937735804838);
    \coordinate (i) at (7.850855075327144,2.801937735804838);
    \coordinate (j) at (7.850855075327144,1);
    \coordinate (k) at (7.850855075327144,5.048917339522305);
    \coordinate (l) at (3.801937735804838,5.048917339522305);
    \coordinate (m) at (12.344814282762078,5.048917339522305);
    \coordinate (n) at (12.344814282762078,2.801937735804838);
    
    \draw [line width=2pt, color=violet] (a) -- (b) node[midway,left] {(0,1)};
    \draw [line width=1pt] (b)-- (d) ;
    \draw [line width=1pt] (a)-- (c);
    \draw [line width=1pt] (c)-- (e);
    \draw [line width=1pt] (e)-- (f);
    \draw [line width=1pt] (f)-- (j);
    \draw [line width=1pt] (j)-- (i);
    \draw [line width=1pt] (i)-- (n);
    \draw [line width=1pt] (n)-- (m);
    \draw [line width=1pt] (m)-- (k);
    \draw [line width=1pt] (k)-- (l);
    \draw [line width=1pt] (l)-- (h);
    \draw [line width=1pt] (h)-- (g);
    \draw [line width=1pt] (g)-- (d);
    \draw [line width=1pt, color=blue] (a) -- (d) node[midway, above] {$\nu_0$};
    \draw [line width=1pt, color=red] (c) -- (f) node[midway,above] {$\sigma_1$};
    \draw [line width=1pt, color=green] (d) -- (h) node[midway, above] {$\nu_1$};
    \draw [line width=1pt, color=orange] (f) -- (i) node[midway, above] {$\sigma_2$};
    \draw [line width=1pt, color=yellow] (h) -- (k) node[midway, above] {$\nu_2$};
    \draw [line width=1pt, color=gray] (i) -- (m) node[midway, above] {$\sigma_3$};
    
    \draw [dashed] (c) -- (d);
    \draw[dashed] (d) -- (f);
    \draw [dashed] (h) -- (f);
    \draw[dashed] (h) -- (i);
    \draw [dashed] (k)-- (i);
    \end{tikzpicture}
    \caption{Labeling of $\mathcal{S}'$ using $\sigma$'s and $\nu$'s for $n = 7$}
    \label{fig:sigma_label}
\end{figure}

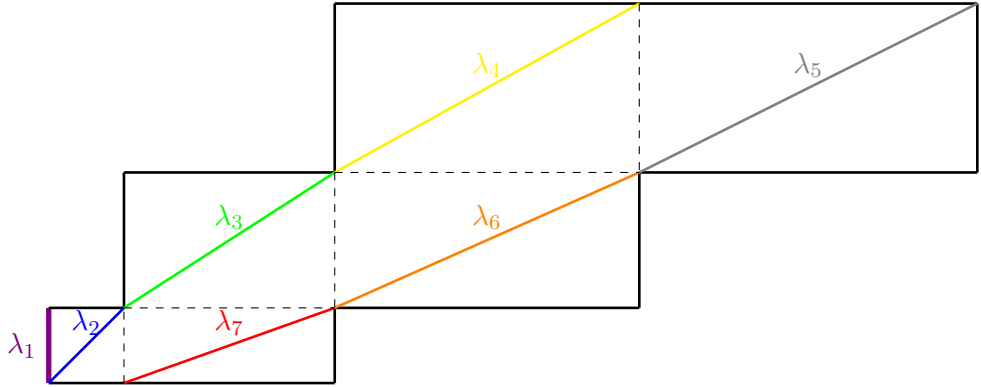
\begin{figure}[h!]
    \centering
        \begin{tikzpicture}
    \coordinate (a) at (0,0);
    \coordinate (b) at (0,1);
    \coordinate (c) at (1,0);
    \coordinate (d) at (1,1);
    \coordinate (e) at (3.801937735804838,0);
    \coordinate (f) at (3.801937735804838,1);
    \coordinate (g) at (1,2.801937735804838);
    \coordinate (h) at (3.801937735804838,2.801937735804838);
    \coordinate (i) at (7.850855075327144,2.801937735804838);
    \coordinate (j) at (7.850855075327144,1);
    \coordinate (k) at (7.850855075327144,5.048917339522305);
    \coordinate (l) at (3.801937735804838,5.048917339522305);
    \coordinate (m) at (12.344814282762078,5.048917339522305);
    \coordinate (n) at (12.344814282762078,2.801937735804838);
    
    \draw [line width=2pt, color=violet] (a) -- (b) node[midway,left] {$\lambda_1$};
    \draw [line width=1pt] (b)-- (d) ;
    \draw [line width=1pt] (a)-- (c);
    \draw [line width=1pt] (c)-- (e);
    \draw [line width=1pt] (e)-- (f);
    \draw [line width=1pt] (f)-- (j);
    \draw [line width=1pt] (j)-- (i);
    \draw [line width=1pt] (i)-- (n);
    \draw [line width=1pt] (n)-- (m);
    \draw [line width=1pt] (m)-- (k);
    \draw [line width=1pt] (k)-- (l);
    \draw [line width=1pt] (l)-- (h);
    \draw [line width=1pt] (h)-- (g);
    \draw [line width=1pt] (g)-- (d);
    \draw [line width=1pt, color=blue] (a) -- (d) node[midway, above] {$\lambda_2$};
    \draw [line width=1pt, color=red] (c) -- (f) node[midway,above] {$\lambda_7$};
    \draw [line width=1pt, color=green] (d) -- (h) node[midway, above] {$\lambda_3$};
    \draw [line width=1pt, color=orange] (f) -- (i) node[midway, above] {$\lambda_6$};
    \draw [line width=1pt, color=yellow] (h) -- (k) node[midway, above] {$\lambda_4$};
    \draw [line width=1pt, color=gray] (i) -- (m) node[midway, above] {$\lambda_5$};
    
    \draw [dashed] (c) -- (d);
    \draw[dashed] (d) -- (f);
    \draw [dashed] (h) -- (f);
    \draw[dashed] (h) -- (i);
    \draw [dashed] (k)-- (i);
    \end{tikzpicture}
    \caption{Labeling of $\mathcal{S}'$ with $\lambda$'s for $n = 7$}
    \label{fig:lambda_label}
\end{figure}

Note that as holonomy vectors, $\sigma_i = \cvec{h_i}{v_{i-1}}$ and $\nu_i = \cvec{h_i}{v_i}$.
Thus, restating Proposition $\ref{big prop}$ in terms of this new notation, our goal is to show that the vector $\cvec{0}{1}$ wins on $P_1$, that $\nu_{i-2}$ wins on $P_i$ for $1 < i \leq \lceil n/2 \rceil$, that $\sigma_{n-i}$ wins on $P_i$ for $\lfloor n/2 \rfloor +1 < i < n$, and that $\cvec{h_1}{v_0}$ wins on $P_n$.
Finally, let $\lambda_i$ denote the vector which we claim wins on $P_i$.

Lemma \ref{check_on_x=1} implies that in order to find the full set of winning vectors for $\Omega_1$, we only need to consider the set of vectors which win at points in the intersection of $\Omega_1$ and the line $x = 1$.
We will first establish that this set of vectors is indeed the vectors $\lambda_i$.
Then, knowing that one of the $\lambda_i$ wins at every point in $\Omega_1$ , we will show that each $\lambda_i$ wins on $P_i$ as compared to the other $\lambda_j$ vectors; since $\cup_i P_i = \Omega_1$, this provides a full description of the return time function on $\Omega_1$.

%
%
%
%

For each $i$, in order to prove that $\lambda_i$ wins on $P_i \cap \{x = 1\}$, we establish four conditions which must be satisfied by any vector $\cvec{a}{b}$ which wins over $\lambda_i$ in the region, and prove that all saddle connections on $\mathcal{S}'$ apart from $\lambda_i$ violate at least one condition.

\begin{itemize}

\item \textbf{Condition 1} is a restriction on the horizontal length of the vector, i.e. there exists some $q > 0$ such that $a < q$ (we use an explicit $q$ which differs between cases).

\item \textbf{Condition 2} is the restriction that $b/a < slope(\lambda_i)$ (otherwise $\cvec{a}{b}$ would not have a smaller slope than $\lambda_i$).

\item Now, for vector $\cvec{a}{b}$, define $f\cvec{a}{b} = \frac{1 - a}{b}$.
Notice that for $(1, y) \in \Omega_1$, in order for the image of $\cvec{a}{b}$ to have horizontal component at most $1$, $y$ must be at most $f\cvec{a}{b}$.
Thus we can invoke another condition, \textbf{Condition 3}, that for a candidate vector $v$, $f(v)$ must be greater than the lower bound for $y$ in the region $P_i \cap \{x = 1\}$, which is simply $f(\lambda_{i + 1})$.
So this condition can be simplified to $a - b \cdot f(\lambda_{i+1}) < 1$.
Notice that this constraint can be rephrased as follows, with $m = b/a$ being the slope of a given vector
\[a ( 1 - m \cdot f(\lambda_{i+1}) ) < 1\]
So, if $v$ and $v'$ are vectors such that $v'$ has a larger horizontal length $a$ and a smaller slope $m$ than $v$, then if $v$ does not satisfy the above constraint, neither will $v'$.

\item \textbf{Condition 4} is the restriction that $a/b < 2 + 2\cos(\pi/n)$ (which is incidentally the aspect ratio of the rectangle formed by merging $H_{i+1}$ and $V_{i}$ for all $i$).
This is because for $(1, y) \in \Omega_1$, $y > 1 - (2 + 2\cos(\pi/n))$, and for any saddle connection $\cvec{a}{b}$ of $S'$, $1/b \leq 1$, so we can rearrange the previously-mentioned constraint $a + by \leq 1$ as follows:
\[1 - a/b \geq 1/b - a/b \geq y > -1 - 2\cos(\pi/n)\]
The inequality $a/b < 2 + 2\cos(\pi/n)$ directly follows from the above inequality.
\end{itemize}

\begin{rem} Conditions 1, 2, and 3 do not refer to anything specific about $\mathcal{S}'$, so these conditions must hold for all Veech surfaces. Of course, the value $q$ in Condition $1$ will depend on the surface; we demonstrate how to explicitly compute $q$ in our casework. On the other hand, an analogue of Condition 4 should hold for all Veech surfaces (after scaling to have $1/b \leq 1$ for all holonomy vectors $\cvec{a}{b}$ on the surface): the constraint $a + by \leq 1$ yields an analogous inequality after substituting in the minimum $y$-coordinate in some region $\Omega_i$ for a general Veech surface.
\end{rem}

\begin{rem}
Since straightforward analogues of the aforementioned four conditions can be applied to narrow down winning vectors for any Veech surface and since the techniques we used in Section \ref{sect:bounds} can be easily adapted once the winning vectors are known, we believe that the methods in this paper can be generalized to compute bounds on the number of non-differentiability points in the slope gap distributions of general Veech surfaces.
\end{rem}

\begin{lem}\label{lem:forder}
    $f(\lambda_i) < f(\lambda_{i + 1})$
\end{lem}
\begin{proof}
    We first show $f(\nu_0) > f\cvec{0}{1}$:
    \[f(\nu_0) - f((0, 1)) = \frac{h_0 - 1}{v_0} - \frac{0 - 1}{1} = \frac{1 - 1}{1} + 1 = 1 >0\]
    Now we show show $f(\nu_{i + 1}) > f(\nu_i)$:
    \begin{align*}f(\nu_{i + 1}) - f(\nu_i) &= \frac{h_{i + 1} - 1}{v_{i + 1}} - \frac{h_i - 1}{v_i} \\[10pt] &=
    \frac{v_i^2 + v_iv_{i + 1} - v_i - v_iv_{i + 1} - v_{i - 1}v_{i + 1} + v_{i + 1}}{v_iv_{i + 1}} \\[10pt] &= \frac{1 + v_{i + 1} - v_i}{v_iv_{i + 1}} > 0\end{align*}
    Now we show $f(\sigma_{i + 1}) > f(\nu_i)$:
    \[f(\sigma_{i + 1}) - f(\nu_i) = \frac{h_{i + 1} - 1}{v_i} - \frac{h_i - 1}{v_i} = \frac{h_{i + 1} - h_i}{v_i} > 0\]
    Finally we show $f(\sigma_i) > f(\sigma_{i + 1})$:
    \begin{align*}f(\sigma_i) - f(\sigma_{i + 1}) &= \frac{h_i - 1}{v_{i - 1}} - \frac{h_{i + 1} - 1}{v_i} \\[10pt] &= \frac{v_i(v_{i - 1} + v_i - 1) - v_{i - 1}(v_i + v_{i + 1} - 1)}{v_{i - 1}v_i}
    \\[10pt] &= \frac{v_i^2 - v_{i - 1}v_{i + 1} + v_{i - 1} - v_i}{v_{i - 1}v_i}\\[10pt] &= \frac{1 + v_{i - 1} - v_i}{v_{i - 1}v_i} > 0\end{align*}
\end{proof}

\begin{lem}\label{lem:order}
\[\operatorname{Slope}(\lambda_2) > \operatorname{Slope}(\lambda_3) > \cdots > \operatorname{Slope}(\lambda_n)\]
\end{lem}

\begin{proof}
First we show that for $0 \leq i < \lfloor n/2 \rfloor - 1$, $\operatorname{Slope}(\nu_i) > \operatorname{Slope}(\nu_{i+1})$.
By direct computation, \[\operatorname{Slope}(\nu_i) = \frac{1}{\sin \left(\frac{\pi  i}{n}\right) \csc \left(\frac{\pi  (i+1)}{n}\right)+1},\] so we show that this is a decreasing function of $i$. Indeed, computing the derivative: \[\frac{d}{di}\operatorname{Slope}(\nu_i) = -\frac{\pi  \sin \left(\frac{\pi }{n}\right)}{n \left(\sin \left(\frac{\pi 
   i}{n}\right)+\sin \left(\frac{\pi  (i+1)}{n}\right)\right)^2} < 0.\] Since $\lambda_{i+1} = \nu_i$ for $0 \leq i \leq \lfloor n/2 \rfloor - 1$, we've shown the inequality $\operatorname{Slope}(\lambda_j) > \operatorname{Slope}(\lambda_{j+1})$ for $j = 2, \ldots, \lfloor n/2 \rfloor$.

Now we show that $\operatorname{Slope}(\lambda_{\lfloor n/2 \rfloor + 1}) > \operatorname{Slope}(\lambda_{\lfloor n/2 \rfloor + 2})$.
For even $n$, we must show that $\operatorname{Slope}(\nu_{n/2 - 1}) > \operatorname{Slope}(\sigma_{n/2 - 1})$.
The horizontal component of both of these vectors is $h_{n/2 - 1}$.
However, the vertical component of $\nu_{n/2 - 1}$ is $v_{n/2 - 1}$, which is greater than $v_{n/2 - 2}$, the vertical component of $\sigma_{n/2 - 1}$. This proves the claim for even $n$.
For odd $n$, we must show that $\operatorname{Slope}(\nu_{(n-3)/2}) > \operatorname{Slope}(\sigma_{(n-1)/2})$. The vertical components of both of these vectors is $v_{(n-3)/2}$, so it suffices to compare the horizontal components. The horizontal component of $\nu_{(n-3)/2}$ is $h_{(n-3)/2}$, which is less than $h_{(n-1)/2}$, the horizontal component of $\sigma_{(n-1)/2}$. This proves the inequality for odd $n$. 

Finally we show that for $1 \leq i < \lceil n/2 \rceil - 1$, $\operatorname{Slope}(\sigma_{i+1}) > \operatorname{Slope}(\sigma_{i})$.
By direct computation, \[\operatorname{Slope}(\sigma_i) = \frac{1}{\sin \left(\frac{\pi  (i+1)}{n}\right) \csc \left(\frac{\pi  i}{n}\right)+1}\] so we show that this is an increasing function of $i$. Indeed, computing the derivative: \[\frac{d}{di}\operatorname{Slope}(\sigma_i) = \frac{\pi  \sin \left(\frac{\pi }{n}\right)}{n \left(\sin \left(\frac{\pi 
   i}{n}\right)+\sin \left(\frac{\pi  (i+1)}{n}\right)\right)^2} > 0.\] Since $\lambda_{n+1-i} = \sigma_i$ for $1 \leq i < \lceil n/2 \rceil - 1$, we've shown the inequality $\operatorname{Slope}(\lambda_j) > \operatorname{Slope}(\lambda_{j+1})$ for $\lceil n/2 \rceil + 1 \leq j < n$. Together, these inequalities are equivalent to the lemma.
\end{proof}

\noindent \textbf{High-level overview of casework.} Recall that the goal of our casework is to show that the winning vectors are found among the $\lambda_i$. In \textbf{Case 1}, we prove that $\cvec{0}{1}$ wins when $0 < y \leq 1$. In \textbf{Case 2}, we show that the vectors $\nu_i$ win on $P_{i+1} \cap \{x=1\}$ for $0 \leq i \leq \lfloor n/2 \rfloor - 1$. In \textbf{Case 3}, we prove that the vector $\nu_{\lfloor n/2 \rfloor - 1}$ wins on the region $P_{\lfloor n/2 \rfloor + 1} \cap \{x=1\}$. In \textbf{Case 4}, we show that the vectors $\sigma_i$ win in $P_{n+1 - i} \cap \{x=1\}$. In \textbf{Case 5}, we show that $\sigma_1$ wins on $P_n \cap\{x=1\}$. Each of our cases are divided into several subcases which each eliminate certain vectors from winning over the proposed winning vector. After our casework, we argue that the vectors $\lambda_i$ win in the subregions prescribed in Proposition \ref{big prop}.

\medskip

\noindent \textbf{Case 1.} First we will prove $\cvec{0}{1}$ wins on $x = 1$, $0 < y \leq 1$. Condition 3 stipulates that $a < 1$; since each horizontal length in $S'$ is at least 1, the only vectors satisfying this constraint are vertical ones. These have the same slope, as do their images when multiplied by matrices $M_{x, y}$, so we simply choose $\cvec{0}{1}$ as the winning vector (any choice of vertical saddle connection on $S'$ would be valid).
Note that for $x = 1$, $0 < y \leq 1$, $M_{x, y}\cvec{0}{1} = \cvec{y}{1}$, so the image has positive horizontal length at most $1$ as desired.

\bigskip

Before proceeding to the other cases, we will label vertices as follows: the concave vertex on the upper left side where the sides of length $v_k$ and $h_{k - 1}$ meet will be labeled $L_k$, and the concave vertex on the lower right side where the sides of length $v_{k - 1}$ and $h_{k}$ meet will be labeled $R_k$, as shown in the diagram in Figure \ref{fig:labeled_vertices1}. Under this labeling, note that $\nu_i$ joins $L_{i}$ and $L_{i+1}$. 

\begin{figure}[h]
\begin{tikzpicture}
\coordinate[label=below left:$L_0$] (a) at (0,0);
\coordinate (b) at (0,1);
\coordinate[label=below:$R_0$] (c) at (1,0);
\coordinate[label=above left:$L_1$] (d) at (1,1);
\coordinate (e) at (3.801937735804838,0);
\coordinate[label=below right:$R_1$] (f) at (3.801937735804838,1);
\coordinate (g) at (1,2.801937735804838);
\coordinate[label=above left:$L_2$] (h) at (3.801937735804838,2.801937735804838);
\coordinate[label=below right:$R_2$] (i) at (7.850855075327144,2.801937735804838);
\coordinate (j) at (7.850855075327144,1);
\coordinate[label=below right:$L_3$] (k) at (7.850855075327144,5.048917339522305);
\coordinate (l) at (3.801937735804838,5.048917339522305);
\coordinate[label=below right:$R_3$] (m) at (12.344814282762078,5.048917339522305);
\coordinate (n) at (12.344814282762078,2.801937735804838);

\draw [line width=1pt] (a) -- (b) node at (-0.4, 0.5) {};
\draw [line width=1pt] (b)-- (d);
\draw [line width=1pt] (a)-- (c) node at (0.5, -0.4) {};
\draw [line width=1pt] (c)-- (e) node at (2.4, -0.4) {};
\draw [line width=1pt] (e)-- (f);
\draw [line width=1pt] (f)-- (j) node at (5.825, 0.6) {};
\draw [line width=1pt] (j)-- (i);
\draw [line width=1pt] (i)-- (n) node at (10.1, 2.4) {};
\draw [line width=1pt] (n)-- (m);
\draw [line width=1pt] (m)-- (k);
\draw [line width=1pt] (k)-- (l);
\draw [line width=1pt] (l)-- (h) node at (3.4, 3.92) {};
\draw [line width=1pt] (h)-- (g);
\draw [line width=1pt] (g)-- (d) node at (0.6, 1.9) {};

\draw [dashed] (c) -- (d);
\draw[dashed] (d) -- (f);
\draw [dashed] (h) -- (f);
\draw[dashed] (h) -- (i);
\draw [dashed] (k)-- (i);
\end{tikzpicture}
    \caption{An example of vertex labeling on $\mathcal{S'}$ for $n = 7$.}
    \label{fig:labeled_vertices1}
\end{figure}
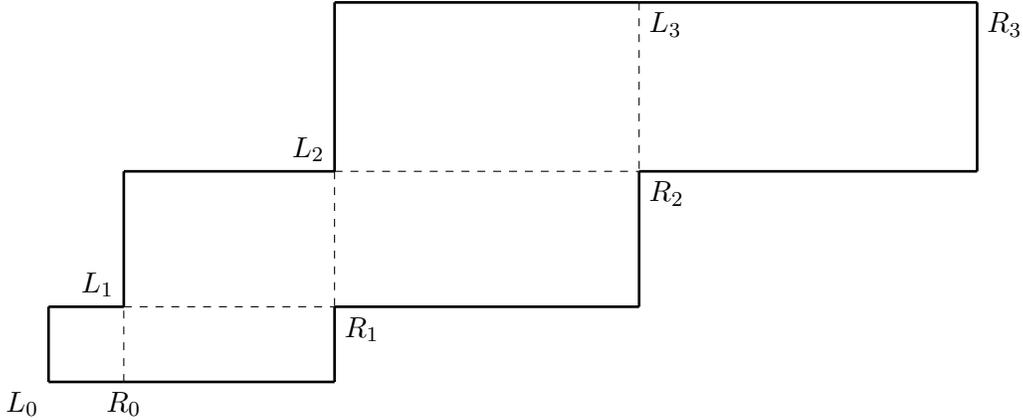

\medskip

\noindent\textbf{Case 2.} We now prove that $\nu_i$ wins on $P_{i+2} \cap \{x = 1\}$ for $0 \leq i \leq \lfloor n/2 \rfloor - 1$. 

Fix some $i$ satisfying $0 \leq i \leq \lfloor n/2 \rfloor - 1$. Recall that as a holonomy vector, $v_i = \cvec{h_i}{v_i}$. We first compute the bound $q$ for Condition $1$ in this case. Suppose $\nu' = \cvec{a}{b}$ is a vector whose image has horizontal component at most $1$ and slope less than the image of $\nu_i$.
In order for the image of $\nu'$ to have horizontal component at most 1, we require $a + by \leq 1$. Because $\nu'$ is a saddle connection on the staircase shape, which has horizontal and vertical distances $\geq 1$, it must be the case that $a \geq 1$ and $b \geq 1$. Furthermore, since the matrices in the Poincar\'e section preserve the ordering of slopes, in order for the image $\nu'$ to have smaller slope that the image of $\nu_i$, it must be true that $b/a < v_i/h_i$ (where the latter quantity is the slope of $\nu_i$). Combining these facts, we find
    \begin{align*} a \leq 1 - by < 1 - \frac{v_i}{h_i}ay < 1 + \frac{v_i(h_{i + 1} - 1)}{h_iv_{i + 1}}a &\implies \\[10pt]
    a \bigg(1 - \frac{v_i(h_{i + 1} - 1)}{h_iv_{i + 1}} \bigg) = a \bigg( \frac{v_i - 1}{v_iv_{i + 1} + v_{i - 1}v_{i + 1}} \bigg) < 1 &\implies \\[10pt]
    a < \frac{v_iv_{i + 1} + v_{i - 1}v_{i + 1}}{v_i - 1} = 1 + h_{i + 1} + \frac{v_{i + 1}}{v_i - 1}.\end{align*}
    Now, note that $\frac{v_{i + 1}}{v_i - 1}$ varies inversely with $i$:
    \[\frac{d}{di} \bigg( \frac{v_{i + 1}}{v_i - 1} \bigg) = - \frac{\pi \sin(\pi/n) (1 + \cos((2 + i)\pi/n))}{n(\sin((1 + i)\pi/n) - \sin(\pi/n))^2} < 0.\]
    Since $\frac{v_{i + 1}}{v_i - 1}$ decreases as $i$ increases, we find
    \[a < 1 + h_{i + 1} + \frac{v_{i + 1}}{v_i - 1} \leq 1 + h_{i + 1} + \frac{v_2}{v_1 - 1} = 1 + h_1 + h_{i + 1}\] Thus, our bound $q$ for Condition 1 will be $1 + h_1 + h_{i+1}$. 
    
    \medskip
    
    \textbf{Subcase 2.1.} We first rule out all vectors that do not pass through an edge of the staircase; that is, we will show that any vector other than $\nu_i$ that does not pass through an edge of the staircase fails to meet at least one of the conditions. Recall the labeling of $\mathcal{S'}$ with an example given in Figure \ref{fig:labeled_vertices1}.
    All $\sigma_j$ and the $\nu_j$ with $j > i$ violate Condition 3 by the ordering of the $f(\lambda_j)$s proven in Lemma \ref{lem:forder}, and for $j < i$, the vector $\nu_j$ has larger slope than $\nu_i$ by Lemma \ref{lem:order}, so Condition 2 is violated.
    
    \textbf{Subsubcase 2.1.1.} Now consider vectors linking $L_j$ and $L_k$ for some $k > j + 1$ (we don't need to consider $k = j+1$ because the vectors linking $L_j$ and $L_{j+1}$ are precisely the set of $\nu_j$).
    If $k \leq i + 1$, we note that $\cvec{a}{b}$ will be strictly steeper than $\nu_k$ and hence violates Condition 2; this is because $\cvec{a}{b}$ would have to pass to the right of $L_{k-1}$ and hence would be steeper than $\nu_{k-1}$ (since $\nu_{k-1}$ joins $L_{k-1}$ and $L_k$).
    Since $\nu_{k-1}$ is at least as steep as $\nu_i$ for $k \leq i+1$, Condition 2 is violated.
    
    \begin{figure}[h!]
    \centering
\begin{tikzpicture}[scale=0.75]
\coordinate[label=below left:$L_0$] (a) at (0,0);
\coordinate (b) at (0,1);
\coordinate[label=below:$R_0$] (c) at (1,0);
\coordinate[label=above left:$L_1$] (d) at (1,1);
\coordinate (e) at (3.801937735804838,0);
\coordinate[label=below right:$R_1$] (f) at (3.801937735804838,1);
\coordinate (g) at (1,2.801937735804838);
\coordinate[label=above left:$L_2$] (h) at (3.801937735804838,2.801937735804838);
\coordinate[label=below right:$R_2$] (i) at (7.850855075327144,2.801937735804838);
\coordinate (j) at (7.850855075327144,1);
\coordinate[label=below right:$L_3$] (k) at (7.850855075327144,5.048917339522305);
\coordinate (l) at (3.801937735804838,5.048917339522305);
\coordinate[label=below right:$R_3$] (m) at (12.344814282762078,5.048917339522305);
\coordinate (n) at (12.344814282762078,2.801937735804838);

\draw [line width=1pt] (a) -- (b) node at (-0.4, 0.5) {};
\draw [line width=1pt] (b)-- (d);
\draw [line width=1pt] (a)-- (c) node at (0.5, -0.4) {};
\draw [line width=1pt] (c)-- (e) node at (2.4, -0.4) {};
\draw [line width=1pt] (e)-- (f);
\draw [line width=1pt] (f)-- (j) node at (5.825, 0.6) {};
\draw [line width=1pt] (j)-- (i);
\draw [line width=1pt] (i)-- (n) node at (10.1, 2.4) {};
\draw [line width=1pt] (n)-- (m);
\draw [line width=1pt] (m)-- (k);
\draw [line width=1pt] (k)-- (l);
\draw [line width=1pt] (l)-- (h) node at (3.4, 3.92) {};
\draw [line width=1pt] (h)-- (g);
\draw [line width=1pt] (g)-- (d) node at (0.6, 1.9) {};
\draw [line width=3pt, color=yellow] (h)--(k) node[midway, above, color=black] {$\nu_2$};
\draw [line width=3pt, color=green] (d)--(h);
\draw [line width=2pt, color=black] (a)--(h);
\draw [line width=2pt, color=black] (a)--(k);
\draw [dashed] (c) -- (d);
\draw[dashed] (d) -- (f);
\draw [dashed] (h) -- (f);
\draw[dashed] (h) -- (i);
\draw [dashed] (k)-- (i);
\end{tikzpicture}
    \caption{If fix $i=2$ in this case, we see that any vector linking $L_j$ and $L_{k}$ (black vectors) with $k \leq i+1$ while staying in the staircase is either steeper than $\nu_1$ (green vector) or $\nu_2$ (yellow vector).}
    \label{fig:labeled_vertices2}
\end{figure}
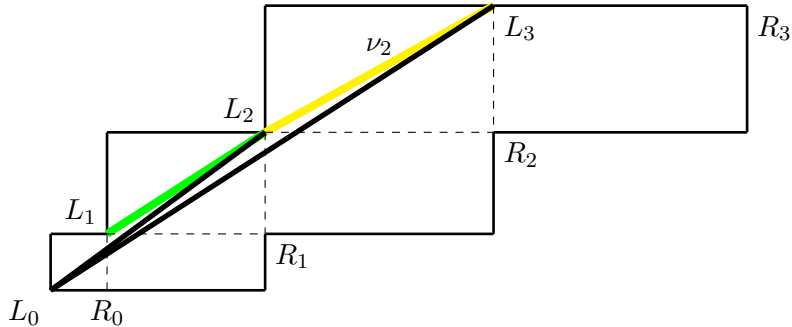

Now suppose $k > i + 1$. It must be true that $j \geq k - 2$, otherwise Condition 1 is violated: \[a \geq h_k + h_{k - 1} + h_{k - 2} \geq h_{i + 1} + h_1 + 1\] If $k > i + 2$, then $j > i$, meaning $\cvec{a}{b}$ has larger $a$ and smaller slope than $\nu_j$, which we showed earlier violates Condition 3, and hence $\cvec{a}{b}$ violates Condition 3 as well. This leaves the case of $k = i+2$, meaning $j = i$ is forced. The holonomy vector here is $\cvec{h_i+h_{i+1}}{v_i+v_{i+1}}$ and can be shown to violate Condition 3:
    \begin{align*} h_i + h_{i + 1} - (v_{i + 1} + v_i) \bigg( \frac{h_{i + 1} - 1}{v_{i + 1}} \bigg) &= 1 + h_i - v_i \bigg( \frac{h_{i + 1} - 1}{v_{i + 1}} \bigg)  \\[5pt]
   &=  1 + v_{i - 1} + v_i - v_i \bigg( \frac{v_i + v_{i + 1} - 1}{v_{i + 1}} \bigg) \\[5pt] &= 1 + \frac{v_i - (v_i^2 - v_{i - 1}v_{i + 1})}{v_{i + 1}} \\[5pt] &=
    1 + \frac{v_i - 1}{v_{i + 1}} \geq 1\end{align*} Thus, all vectors connecting an $L_j$ to another $L_k$ have been eliminated. 
    
    \textbf{Subsubcase 2.1.2.} Now consider the vectors linking some $R_j$ to some $L_k$.
    If $k \leq i + 1$, then Condition 1 is violated, for the same reasons that we mentioned in the $k \leq i + 1$ case for vectors connecting $L_j$ and $L_k$.  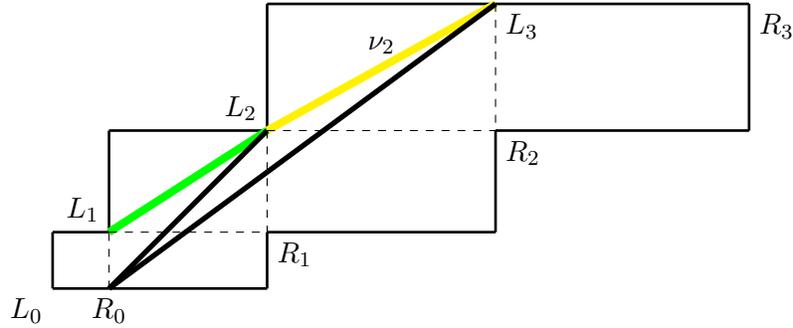
\begin{figure}[h!]
    \centering
\begin{tikzpicture}[scale=0.75]
\coordinate[label=below left:$L_0$] (a) at (0,0);
\coordinate (b) at (0,1);
\coordinate[label=below:$R_0$] (c) at (1,0);
\coordinate[label=above left:$L_1$] (d) at (1,1);
\coordinate (e) at (3.801937735804838,0);
\coordinate[label=below right:$R_1$] (f) at (3.801937735804838,1);
\coordinate (g) at (1,2.801937735804838);
\coordinate[label=above left:$L_2$] (h) at (3.801937735804838,2.801937735804838);
\coordinate[label=below right:$R_2$] (i) at (7.850855075327144,2.801937735804838);
\coordinate (j) at (7.850855075327144,1);
\coordinate[label=below right:$L_3$] (k) at (7.850855075327144,5.048917339522305);
\coordinate (l) at (3.801937735804838,5.048917339522305);
\coordinate[label=below right:$R_3$] (m) at (12.344814282762078,5.048917339522305);
\coordinate (n) at (12.344814282762078,2.801937735804838);

\draw [line width=1pt] (a) -- (b) node at (-0.4, 0.5) {};
\draw [line width=1pt] (b)-- (d);
\draw [line width=1pt] (a)-- (c) node at (0.5, -0.4) {};
\draw [line width=1pt] (c)-- (e) node at (2.4, -0.4) {};
\draw [line width=1pt] (e)-- (f);
\draw [line width=1pt] (f)-- (j) node at (5.825, 0.6) {};
\draw [line width=1pt] (j)-- (i);
\draw [line width=1pt] (i)-- (n) node at (10.1, 2.4) {};
\draw [line width=1pt] (n)-- (m);
\draw [line width=1pt] (m)-- (k);
\draw [line width=1pt] (k)-- (l);
\draw [line width=1pt] (l)-- (h) node at (3.4, 3.92) {};
\draw [line width=1pt] (h)-- (g);
\draw [line width=1pt] (g)-- (d) node at (0.6, 1.9) {};
\draw [line width=3pt, color=yellow] (h)--(k) node[midway, above, color=black] {$\nu_2$};
\draw [line width=3pt, color=green] (d)--(h);
\draw [line width=2pt, color=black] (c)--(h);
\draw [line width=2pt, color=black] (c)--(k);
\draw [dashed] (c) -- (d);
\draw[dashed] (d) -- (f);
\draw [dashed] (h) -- (f);
\draw[dashed] (h) -- (i);
\draw [dashed] (k)-- (i);
\end{tikzpicture}
    \caption{If fix $i=2$ in this case, we see that any vector linking $L_j$ and $R_{k}$ with $k \leq i+1$ while staying in the staircase (black vectors) is either steeper than $\nu_1$ (green vector) or $\nu_2$ (yellow vector).}
    \label{fig:labeled_vertices3}
\end{figure}
    
    Suppose $k > i + 1$. If $j < k - 3$, then Condition 1 is violated:
    \[a \geq h_{k - 1} + h_{k - 2} + h_{k - 3} \geq h_{i + 1} + h_1 + 1.\]
    If $j = k - 2$, then Condition 2 is violated:
    \[b/a = (v_{k - 1} + v_{k - 2})/h_{k - 1} = 1 \geq v_i/h_i\]
    Since it must be true that $j < k - 1$ for $\cvec{a}{b}$ to have positive slope, we are left with $j = k - 3$, $j \geq i - 1$. Any of these vectors violates Condition 2:
    \begin{align*}b/a &= \frac{v_j + v_{j + 1} + v_{j + 2}}{h_{j + 1} + h_{j + 2}} \\[5pt] &> \frac{v_j + v_{j + 1} + v_{j + 2}}{h_j + h_{j + 1} + h_{j + 2}} \\[5pt]
    &= \frac{(v_j + v_{j + 1} + v_{j + 2})(v_j + v_{j + 1})}{(h_j + h_{j + 1} + h_{j + 2})h_{j + 1}} \\[5pt] &= \frac{v_j^2 + 2v_jv_{j + 1} + v_{j + 1}^2 + v_jv_{j + 2} + v_{j + 1}v_{j + 2}}{(h_j + h_{j + 1} + h_{j + 2})h_{j + 1}} \\[5pt] &=
    \frac{v_j^2 - v_{j - 1}v_{j + 1} - v_{j + 1}^2 + v_jv_{j + 2} + v_{j + 1}^2 + 2v_jv_{j + 1} + v_{j + 1}^2 + v_{j - 1}v_{j + 1} + v_{j + 1}v_{j + 2}}{(h_j + h_{j + 1} + h_{j + 2})h_{j + 1}}\\[5pt] &= 
     \frac{1 - 1 + v_{j - 1}v_{j + 1} + 2v_jv_{j + 1} + 2v_{j + 1}^2 + v_{j + 1}v_{j + 2}}{(h_j + h_{j + 1} + h_{j + 2})h_{j + 1}} \\[5pt]&=
    \frac{(h_j + h_{j + 1} + h_{j + 2})v_{j + 1}}{(h_j + h_{j + 1} + h_{j + 2})h_{j + 1}} = v_{j + 1}/h_{j + 1} \geq v_i/h_i. \end{align*}
    \medskip
    Thus, all vectors connecting a vertex of the form $R_j$ to a vertex of the form $L_k$ violate at least one condition.
    
    \textbf{Subsubcase 2.1.3.} Now consider a saddle connection staying in the staircase connecting an $R_j$ or $L_j$ to some $R_k$. We showed previously that all $\sigma_i$'s violate Condition 3. Given any $k$, a vector $\cvec{a}{b}$ terminating at $R_k$ has $a$ at least as large as the horizontal component of $\sigma_k$ and slope that is shallower than the slope of $\sigma_k$ ($\cvec{a}{b}$ passes into $H_{k-1}$ above the starting point of $\sigma_k$ and terminates at the same place as $\sigma_k$), so any such $\cvec{a}{b}$ violates Condition 3. This eliminates all vectors staying in the staircase (i.e., vectors that do not pass through an edge).
    
     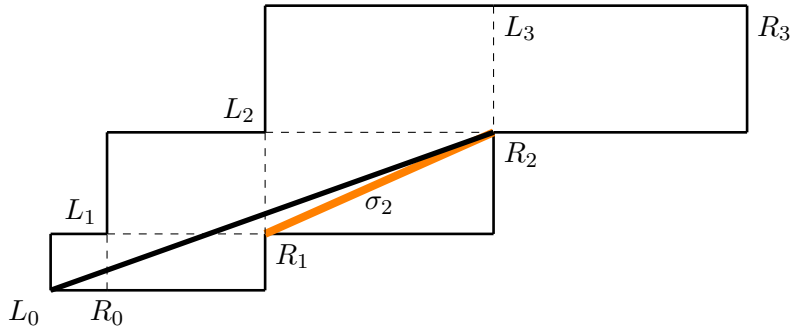
\begin{figure}[h!]
    \centering
\begin{tikzpicture}[scale=0.75]
\coordinate[label=below left:$L_0$] (a) at (0,0);
\coordinate (b) at (0,1);
\coordinate[label=below:$R_0$] (c) at (1,0);
\coordinate[label=above left:$L_1$] (d) at (1,1);
\coordinate (e) at (3.801937735804838,0);
\coordinate[label=below right:$R_1$] (f) at (3.801937735804838,1);
\coordinate (g) at (1,2.801937735804838);
\coordinate[label=above left:$L_2$] (h) at (3.801937735804838,2.801937735804838);
\coordinate[label=below right:$R_2$] (i) at (7.850855075327144,2.801937735804838);
\coordinate (j) at (7.850855075327144,1);
\coordinate[label=below right:$L_3$] (k) at (7.850855075327144,5.048917339522305);
\coordinate (l) at (3.801937735804838,5.048917339522305);
\coordinate[label=below right:$R_3$] (m) at (12.344814282762078,5.048917339522305);
\coordinate (n) at (12.344814282762078,2.801937735804838);

\draw [line width=1pt] (a) -- (b) node at (-0.4, 0.5) {};
\draw [line width=1pt] (b)-- (d);
\draw [line width=1pt] (a)-- (c) node at (0.5, -0.4) {};
\draw [line width=1pt] (c)-- (e) node at (2.4, -0.4) {};
\draw [line width=1pt] (e)-- (f);
\draw [line width=1pt] (f)-- (j) node at (5.825, 0.6) {};
\draw [line width=1pt] (j)-- (i);
\draw [line width=1pt] (i)-- (n) node at (10.1, 2.4) {};
\draw [line width=1pt] (n)-- (m);
\draw [line width=1pt] (m)-- (k);
\draw [line width=1pt] (k)-- (l);
\draw [line width=1pt] (l)-- (h) node at (3.4, 3.92) {};
\draw [line width=1pt] (h)-- (g);
\draw [line width=1pt] (g)-- (d) node at (0.6, 1.9) {};
\draw [line width=3pt, color=orange] (f)--(i) node[midway, below, color=black] {$\sigma_2$};
\draw [line width=2pt, color=black] (a)--(i);
\draw [dashed] (c) -- (d);
\draw[dashed] (d) -- (f);
\draw [dashed] (h) -- (f);
\draw[dashed] (h) -- (i);
\draw [dashed] (k)-- (i);
\end{tikzpicture}
    \caption{For the example of $k = 2$, we see that any vector connecting some $R_j$ or $L_j$ (in this case $L_0$) and $R_2$ will pass above $R_1$ and hence be longer and shallower than $\sigma_2$.}
    \label{fig:labeled_vertices4}
\end{figure}

\medskip
\newpage
\textbf{Subcase 2.2.} We now rule out vectors that pass through the edges of the staircase. 

\textbf{Subsubcase 2.2.1.} If the first edge that a vector $\cvec{a}{b}$ passes through is a vertical edge of length $v_k$, then $\cvec{a}{b}$ is longer and shallower than $\sigma_{k + 1}$, and so $\cvec{a}{b}$ violates Condition 3. 

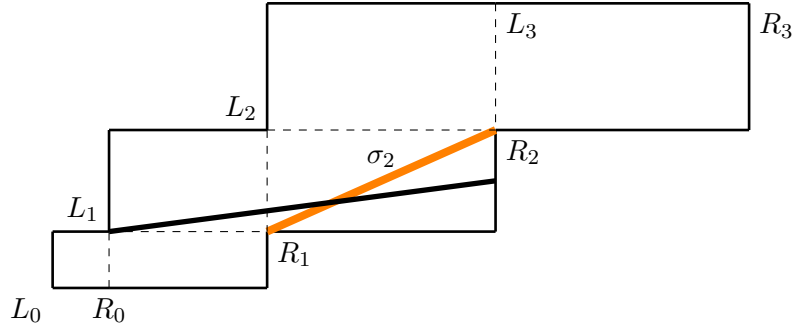
\begin{figure}[h!]
    \centering
\begin{tikzpicture}[scale=0.75]
\coordinate[label=below left:$L_0$] (a) at (0,0);
\coordinate (b) at (0,1);
\coordinate[label=below:$R_0$] (c) at (1,0);
\coordinate[label=above left:$L_1$] (d) at (1,1);
\coordinate (e) at (3.801937735804838,0);
\coordinate[label=below right:$R_1$] (f) at (3.801937735804838,1);
\coordinate (g) at (1,2.801937735804838);
\coordinate[label=above left:$L_2$] (h) at (3.801937735804838,2.801937735804838);
\coordinate[label=below right:$R_2$] (i) at (7.850855075327144,2.801937735804838);
\coordinate (j) at (7.850855075327144,1);
\coordinate[label=below right:$L_3$] (k) at (7.850855075327144,5.048917339522305);
\coordinate (l) at (3.801937735804838,5.048917339522305);
\coordinate[label=below right:$R_3$] (m) at (12.344814282762078,5.048917339522305);
\coordinate (n) at (12.344814282762078,2.801937735804838);

\draw [line width=1pt] (a) -- (b) node at (-0.4, 0.5) {};
\draw [line width=1pt] (b)-- (d);
\draw [line width=1pt] (a)-- (c) node at (0.5, -0.4) {};
\draw [line width=1pt] (c)-- (e) node at (2.4, -0.4) {};
\draw [line width=1pt] (e)-- (f);
\draw [line width=1pt] (f)-- (j) node at (5.825, 0.6) {};
\draw [line width=1pt] (j)-- (i) coordinate[midway] (z) {};
\draw [line width=1pt] (i)-- (n) node at (10.1, 2.4) {};
\draw [line width=1pt] (n)-- (m);
\draw [line width=1pt] (m)-- (k);
\draw [line width=1pt] (k)-- (l);
\draw [line width=1pt] (l)-- (h) node at (3.4, 3.92) {};
\draw [line width=1pt] (h)-- (g);
\draw [line width=1pt] (g)-- (d) node at (0.6, 1.9) {};
\draw [line width=3pt, color=orange] (f)--(i) node[midway, above, color=black] {$\sigma_2$};
\draw [line width=2pt, color=black] (d)--(z);
\draw [dashed] (c) -- (d);
\draw[dashed] (d) -- (f);
\draw [dashed] (h) -- (f);
\draw[dashed] (h) -- (i);
\draw [dashed] (k)-- (i);
\end{tikzpicture}
    \caption{For the example of $k = 1$, we see that any vector passing through $v_1$ is longer and shallower than $\sigma_2$.}
    \label{fig:labeled_vertices5}
\end{figure}

\textbf{Subsubcase 2.2.2.} Thus we are left with vectors that first pass through a horizontal edge $h_k$. Note that in this case, the vector must travel a horizontal distance of at least $h_k$. If $k \leq i$, then $\cvec{a}{b}$ is steeper than $\nu_k$, which we showed is at least as steep as $\nu_i$, so Condition 2 is violated.

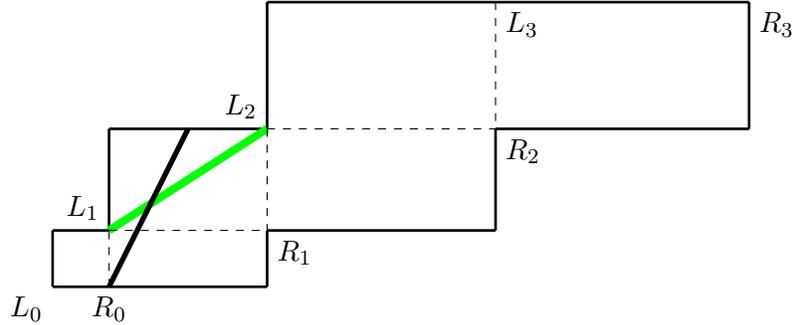
\begin{figure}[h!]
    \centering
\begin{tikzpicture}[scale=0.75]
\coordinate[label=below left:$L_0$] (a) at (0,0);
\coordinate (b) at (0,1);
\coordinate[label=below:$R_0$] (c) at (1,0);
\coordinate[label=above left:$L_1$] (d) at (1,1);
\coordinate (e) at (3.801937735804838,0);
\coordinate[label=below right:$R_1$] (f) at (3.801937735804838,1);
\coordinate (g) at (1,2.801937735804838);
\coordinate[label=above left:$L_2$] (h) at (3.801937735804838,2.801937735804838);
\coordinate[label=below right:$R_2$] (i) at (7.850855075327144,2.801937735804838);
\coordinate (j) at (7.850855075327144,1);
\coordinate[label=below right:$L_3$] (k) at (7.850855075327144,5.048917339522305);
\coordinate (l) at (3.801937735804838,5.048917339522305);
\coordinate[label=below right:$R_3$] (m) at (12.344814282762078,5.048917339522305);
\coordinate (n) at (12.344814282762078,2.801937735804838);

\draw [line width=1pt] (a) -- (b) node at (-0.4, 0.5) {};
\draw [line width=1pt] (b)-- (d);
\draw [line width=1pt] (a)-- (c) node at (0.5, -0.4) {};
\draw [line width=1pt] (c)-- (e) node at (2.4, -0.4) {};
\draw [line width=1pt] (e)-- (f);
\draw [line width=1pt] (f)-- (j) node at (5.825, 0.6) {};
\draw [line width=1pt] (j)-- (i);
\draw [line width=1pt] (i)-- (n) node at (10.1, 2.4) {};
\draw [line width=1pt] (n)-- (m);
\draw [line width=1pt] (m)-- (k);
\draw [line width=1pt] (k)-- (l);
\draw [line width=1pt] (l)-- (h) node at (3.4, 3.92) {};
\draw [line width=1pt] (h)-- (g) coordinate[midway] (z);
\draw [line width=1pt] (g)-- (d) node at (0.6, 1.9) {};
\draw [line width=3pt, color=green] (d)--(h);
\draw [line width=2pt, color=black] (c)--(z);
\draw [dashed] (c) -- (d);
\draw[dashed] (d) -- (f);
\draw [dashed] (h) -- (f);
\draw[dashed] (h) -- (i);
\draw [dashed] (k)-- (i);
\end{tikzpicture}
    \caption{If fix $k=1$ and $i = 2$ in this case, we see that any vector passing through the side of length $h_1$ is steeper than $\nu_1$ (green vector), which in turn is steeper than $\nu_2$.}
    \label{fig:labeled_vertices6}
\end{figure}

Given that $k > i$, the vector $\cvec{a}{b}$ can travel at most two horizontal distances (with a single exception), otherwise Condition 1 is violated for $i > 0$; more precisely, if $\cvec{a}{b}$ travelled three horizontal distances, note that $a$ must be at least $h_k + h_{k+1} + h_l$ for some distance $h_l$. This is because the vector must traverse $h_k$ if it crosses through a horizontal edge with length $h_k$, and moreover, since the vector must have slope less than $1$ to satisfy Condition 2, it must pass through the vertical edge of length $v_{k-1}$ after passing through the horizontal edge of length $h_k$, then traverse horizontal length $h_{k-1}$ as well. In this case, Condition 1 is violated:
\[a \geq h_k + h_{k-1} + h_l \geq h_{i + 1} + h_i + h_l \geq h_{i + 1} + h_1 + 1\]

Thus, we are left with vectors $\cvec{a}{b}$ passing through at most $2$ horizontal distances, with a singular exception. In the special case of $i = 0$, there is one particular vector of this form that does not violate Condition 1. 

This is the vector beginning at $L_0$, passing through the horizontal face just to the left of $L_2$, passing through the vertical face below $R_1$, then ending at $L_1$. This vector is $\cvec{2 + h_1}{2 + v_1}$, and it violates Condition 3: \[a - b \bigg( \frac{h_1 - 1}{v_1} \bigg) = 2 + h_1 - (2 + v_1) = 1 \geq 1.\]

Now we eliminate all other vectors passing through a horizontal side of length $h_k$. Note that $\cvec{a}{b}$ cannot travel only $h_k$ horizontally, because if $\cvec{a}{b}$ first passes through a horizontal side $h_k$, then it must travel at least $v_k + v_{k - 1} = h_k$ vertically, so having horizontal length $h_k$ would mean having a slope of at least $1$, violating Condition 2. Thus, $\cvec{a}{b}$ must then pass through the vertical side of length $v_{k-1}$ and then horizontally travel length $h_{k - 1}$. This is only possible if $\cvec{a}{b}$ begins at $L_k$ or $R_{k - 1}$.

\smallskip

Suppose $\cvec{a}{b}$ begins at $R_{k - 1}$. As it travels the horizontal length of $h_k$, it must climb vertically at least $v_{k - 1} + v_k = h_k$ in order to pass through the horizontal edge of length $h_k$. As we found before, this violates Condition 2.

So we are left with $\cvec{a}{b}$ beginning at $L_k$ and passing through the horizontal edge of length $h_k$ for $k > i$. After passing through $h_k$, the vector must pass through the vertical edge of length $v_{k - 1}$ as stated earlier. Upon passing through $v_{k - 1}$, $\cvec{a}{b}$ can only end at $L_k$ or pass through the horizontal edge of length $h_{k - 1}$ and end at $R_{k - 1}$ without travelling horizontally more than the width of two rectangles or having slope at least $1$.

If the saddle connection ends at $L_k$, it is the vector $(h_{k - 1} + h_k, v_{k - 1} + v_k)$, the same vector as the saddle connection connecting $L_{k - 1}$ and $L_{k + 1}$ without passing through edges, which we have previously eliminated.

If the saddle connection ends at $R_{k-1}$, it is the vector $(h_{k - 1} + h_k, v_{k - 1} + v_{k - 1} + v_k)$, the same vector as the saddle connection connecting $R_{k - 2}$ and $L_{k + 1}$ without passing through edges, which we have previously eliminated.

\textbf{Subcase 2.3.} We have now eliminated all vectors passing through edges of the staircase, with a few exceptions.

The previous proof assumes the staircase extends for an arbitrary number of rectangles in either direction from the rectangle containing $\nu_i$, so in order to complete this proof we must eliminate vectors that reach the ends of the staircase shape.
    
\textbf{Subsubcase 2.3.1.} The vectors that pass through the edge between $L_0$ and $R_0$ will first be eliminated. For $\cvec{a}{b}$ passing through the edge between $L_0$ and $R_0$, consider where it intersects the edge of the staircase or a vertex.
If $\cvec{a}{b}$ ends at some $R_k$ or passes through a vertical edge below some $R_k$, then $\cvec{a}{b}$ is both longer and less steep than $\sigma_k$, which violates Condition 3, so $\cvec{a}{b}$ also violates Condition 3.
If $\cvec{a}{b}$ travels the length of the staircase to the edge at the opposite end of the staircase, then $a \geq 1 + h_1 + h_i$, so Condition 1 is violated.
Hence, $\cvec{a}{b}$ must end at some $L_k$ or pass through a horizontal edge to the left of some $L_k$.
If $k \leq i + 1$ then $\cvec{a}{b}$ is steeper than $\nu_{k - 1}$, which itself is at least as steep as $\nu_i$, so Condition 2 is violated.
If $k > i + 1$, then $a \geq 1 + h_1 + h_{k - 1} \geq 1 + h_1 + h_{i + 1}$, so Condition 1 is violated.

\textbf{Subsubcase 2.3.2.} For even $n$, we will eliminate the vectors passing through the vertical edge below $L_{n/2}$, which forms an edge of the staircase when $n$ is even. Let $k = n/2 - 1$. Notice that such a vector necessarily crosses $h_k$ at least twice. Thus, such a vector cannot travel more than $2$ horizontal distances or else it would violate Condition 1. In fact, such a vector must travel a horizontal distance of exactly $2h_k$. 


We first suppose that $n \geq 6$. The distance travelled by this vector is $2h_k$, so it suffices to show that $h_k \geq 1 + h_1$ in order to violate Condition 1. We compute \[h_k - 1 - h_1 = -2 \cos \left(\frac{\pi }{n}\right)+\cot \left(\frac{\pi }{2 n}\right)-2.\] The derivative of this function is \[\frac{\pi  \left(\csc ^2\left(\frac{\pi }{2 n}\right)-4 \sin \left(\frac{\pi
   }{n}\right)\right)}{2 n^2},\] which is positive for $n \geq 3$. At $n = 6$, this difference evaluates to $0$, meaning that the difference is indeed nonnegative for all $n \geq 6$. In other words, Condition 1 is violated for $n \geq 6$. 
   
   Let $n = 4$. The only possible $\nu_i$ vector in this case is $\nu_0$, since we address the last $\nu_i$ as an entirely separate case from the other $\nu_i$s. Thus Condition 3 is
  \[a - b \bigg (\frac{h_1 - 1}{v_1} \bigg) = a - b \leq 1\]
   Knowing $a = 2h_1$, $b$ can be $v_1$, $1 + v_1$, or $2 + v_1$ without reaching a slope of at least that of $\nu_1$. So,
   \[a - b \geq 2h_1 - (2 + v_1) = v_1 > 1\] Thus Condition 3 is violated. Recall that we do not consider the $n = 2$ case in this paper, so this completes Case 2 for even $n$.
    
\textbf{Subsubcase 2.3.3.} Let $n$ be odd. In this case we must eliminate the vectors passing through the horizontal edge to the right of $L_{\lfloor n/2 \rfloor}$. Let $k = (n - 3)/2$. This vector cannot travel more than $2$ horizontal distances, otherwise we would have $b \geq h_{k + 1} + h_k + h_{k - 1} \geq 1 + h_1 + h_i$, which violates Condition 1. So $\cvec{a}{b}$ must start at $L_k$, $R_{k - 1}$ or $R_k$.
Any $\cvec{a}{b}$ that begins at $L_k$ and passes through the horizontal face before any other faces will have larger $a$ and smaller slope than $\nu_k$, meaning it violates Condition 3.

Otherwise, $\cvec{a}{b}$ must pass through the horizontal face to the left of $L_{k + 1}$ then travel above $R_k$, meaning that over a horizontal distance of $h_k$, it travels vertically at least $v_k + v_{k - 1} = h_k$, so Condition 2 is violated.

Now suppose $\cvec{a}{b}$ begins at $R_{k - 1}$. The only way for $\cvec{a}{b}$ to avoid violating Condition 1 is by ending at $R_{k + 1}$ after passing through the horizontal edge once. For $n > 5$, this violates Condition 1, as $h_1 + 1 \leq h_{k+1}$: \[h_{k+1} - h_1 - 1 = \left(\sin \left(\frac{\pi }{2 n}\right)+\sin \left(\frac{3 \pi }{2 n}\right)-1\right)
   \left(-\csc \left(\frac{\pi }{2 n}\right)\right) > 0.\] 
   
For $n = 5$, this violates Condition 3:
\[f((h_1 + h_2, 1 + 2v_1)) = \frac{1 - h_1 - h_2}{1 + 2v_1} \frac{-3\phi}{1 + 2\phi} < -1 = \frac{1 - h_1}{v_1} = f(\nu_1),\] where $\phi$ is the golden ratio $\phi = \frac{1+\sqrt{5}}{2}$.
For $n = 3$, this saddle connection does not exist.
    
    In the last case, where $\cvec{a}{b}$ begins at $R_k$, it can end at either $R_{k + 1}$, $R_k$, or $L_{k + 1}$ after crossing the right edge precisely once, or else it would violate Condition 1. If $\cvec{a}{b}$ ends at $R_{k + 1}$, then from direct computation (the lower bound of $1$ is the slope when the vector crosses the edge exactly once), this vector has slope at least $1$, which is definitely too steep.
    
    If $\cvec{a}{b}$ ends at $L_{k + 1}$, then it has smaller slope and larger horizontal component than $\nu_k$ (since it would pass through the vertical edge above $L_k$) and thus violates Condition 3.  
    
    The only other possible vector that does not travel horizontally across more than two rectangles is $\nu'$ starting at $R_k$, passing through the horizontal edge to the right of $L_{k + 1}$, passing through the vertical edge below $R_{k + 1}$, passing through the horizontal edge to the left of $L_{k + 1}$, then ending at $R_k$.
    This is the same vector as the one beginning at $R_{k - 1}$, passing through the horizontal edge to the left of $L_{k + 1}$, and ending at $R_{k + 1}$, which we have eliminated. This eliminates all vectors other than $\nu_i$ for Case 2.
    
    Furthermore, $\nu_i$ has positive slope $v_i/h_i$ and is such that $0 < a + by \leq 1$:
    \begin{align*}a + by = h_i + v_i y &> h_i + v_i \cdot \frac{1 - h_{i + 1}}{v_{i + 1}} \\[10pt] &= \frac{h_iv_{i + 1} + v_i - h_{i + 1}v_i}{v_{i + 1}} \\[10pt] &=\frac{v_{i - 1}v_{i + 1} + v_iv_{i + 1} + v_i - v_i^2 - v_iv_{i + 1}}{v_{i + 1}} \\[10pt] &= \frac{v_{i - 1}v_{i + 1} + v_i - v_i^2}{v_{i + 1}} \\[10pt] &= \frac{v_i - 1}{v_{i + 1}} \geq 0,\end{align*} and
   \[ a + by = h_i + v_i y \leq h_i + v_i \cdot \frac{1 - h_i}{v_i} = 1\]
    Thus we have shown that for $\nu_i$ with $0 \leq i < \lfloor n/2 \rfloor - 1$, $\nu_i$ is the winning saddle connection on $P_{i + 2} \cap \{x = 1\}$.

\textbf{Case 3.} We now prove $\nu_{\lfloor n/2 \rfloor - 1}$ has the smallest slope of any saddle connection whose image has horizontal component at most $1$ for all $(1, y) \in \Omega_1$ in the region the region $P_{\lfloor n/2 \rfloor + 1} \cap \{x=1\}$.  Suppose $\cvec{a}{b} \neq \nu_{\lfloor n/2 \rfloor-1}$ wins over $\nu_{\lfloor n/2 \rfloor - 1}$. We start by computing the bound $q$ for Condition 1. 

We have two cases, based on the parity of $n$.

First suppose that $n$ is even. Note that $n \geq 4$, as we exclude the $n = 2$ case from this paper.

The bounds on $y$ demand that  $y > \frac{1 - h_{\lfloor n/2 \rfloor - 1}}{v_{ n/2 - 2}}.$ Then, using the fact that $a \geq 1$ and $a + by \leq 1$, we find:   \begin{align*}
        a \leq 1 - by &< 1 - \frac{v_{\lfloor n/2 \rfloor - 1}}{h_{\lfloor n/2 \rfloor - 1}}ay \\[10pt] &< 1 - \frac{v_{\lfloor n/2 \rfloor - 1}}{h_{\lfloor n/2 \rfloor - 1}}a\left(\frac{1-h_{\lfloor n/2 \rfloor - 1}}{v_{\lfloor n/2 \rfloor - 2}}\right)\\[10pt] &= 1 - \frac{v_{\lfloor n/2 \rfloor - 1}}{v_{\lfloor n/2 \rfloor - 2}}a\left(\frac{1-h_{\lfloor n/2 \rfloor - 1}}{h_{\lfloor n/2 \rfloor - 1}}\right) \end{align*} Rearranging,
        \begin{align*}a\left(1+\frac{v_{\lfloor n/2 \rfloor - 1}}{v_{\lfloor n/2 \rfloor - 2}}\left(\frac{1-h_{\lfloor n/2 \rfloor - 1}}{h_{\lfloor n/2 \rfloor - 1}}\right)\right) < 1 \end{align*} Therefore,
       \begin{align*}a\left(h_{\lfloor n/2 \rfloor - 1} + \frac{v_{\lfloor n/2 \rfloor - 1}}{v_{\lfloor n/2 \rfloor - 2}}(1-h_{\lfloor n/2 \rfloor - 1})\right) &= a\left(\frac{\cos(\pi/(2n)) - \sin(\pi/(2n))}{\cos(\pi/(2n)) + \sin(\pi/(2n))}\right)< h_{\lfloor n/2 \rfloor - 1}
    \end{align*} For $n \geq 4$, $\cos(\pi/(2n)) - \sin(\pi/(2n)) > 0$, so we can safely rearrange:
    \begin{align*}
        a < h_{\lfloor n/2 \rfloor - 1}\left(\frac{\cos(\pi/(2n)) + \sin(\pi/(2n))}{\cos(\pi/(2n)) - \sin(\pi/(2n))}\right) = h_{\lfloor n/2 \rfloor - 1}\left(1+\frac{2}{\cot(\pi/(2n)) - 1}\right)
    \end{align*} Now, we claim that \[\frac{2h_{\lfloor n/2 \rfloor - 1}}{\cot(\pi/(2n)) - 1} \leq 1 + h_1.\] We prove this by showing that \[\frac{2h_{\lfloor n/2 \rfloor - 1}}{\cot(\pi/(2n)) - 1} - (1 + h_1) \leq 0.\] To do so, we first see that when $n = 4$, this expression is $0$. Moreover, taking the derivative with respect to $n$ gives us \[\frac{\pi  \left(2 \sin \left(\frac{\pi }{n}\right)+\cos \left(\frac{2 \pi}{n}\right)\right)}{n^2 \left(\sin \left(\frac{\pi }{n}\right)-1\right)}, \] which is indeed always negative for $n \geq 4$. Hence, our bound becomes \[a < h_{\lfloor n/2 \rfloor - 1}\left(1+\frac{2}{\cot(\pi/(2n)) - 1}\right) \leq 1 + h_1 + h_{\lfloor n/2 \rfloor - 1}.\]
    
    Now suppose that $n$ is odd and greater than $3$ (we will consider the case $n = 3$ later). Then, $y > \frac{1 - h_{\lfloor n/2 \rfloor}}{v_{\lfloor n/2 \rfloor - 1}}$. Combining this with previously mentioned facts, we find:
    \begin{align*}
        a \leq 1 - by &< 1 - \frac{v_{\lfloor n/2 \rfloor - 1}}{h_{\lfloor n/2 \rfloor - 1}}ay \\[10pt] &< 1 - \frac{v_{\lfloor n/2 \rfloor - 1}}{h_{\lfloor n/2 \rfloor - 1}}\left(\frac{1-h_{\lfloor n/2 \rfloor}}{v_{\lfloor n/2 \rfloor - 1}}\right)a\\[10pt] &= 1 - a\left(\frac{1-h_{\lfloor n/2 \rfloor}}{h_{\lfloor n/2 \rfloor - 1}}\right)\end{align*}
        Rearranging,
        \begin{align*}
        a\left(1+\frac{1-h_{\lfloor n/2 \rfloor}}{h_{\lfloor n/2 \rfloor - 1}}\right) = a\left(\frac{h_{\lfloor n/2 \rfloor - 1} -h_{\lfloor n/2 \rfloor}+1}{h_{\lfloor n/2 \rfloor - 1}}\right) &= a \left(\frac{1-2\cos((\lfloor n/2 \rfloor)\pi/n)}{h_{\lfloor n/2 \rfloor - 1}}\right)
        < 1,\end{align*}
        so we conclude
        \begin{align*}
        a\left(\frac{1-2\cos((n-1)\pi/(2n))}{h_{\lfloor n/2 \rfloor - 1}}\right) &< 1.
    \end{align*}
    Observe that \[1-2\cos((n-1)\pi/(2n))> 0.\]
    Hence, we rearrange: \[a < \frac{h_{\lfloor n/2 \rfloor - 1}}{1-2\cos((n-1)\pi/(2n))} = \frac{h_{\lfloor n/2 \rfloor - 1}}{1-2\sin(\pi/(2n))}.\]

    We find that in fact,
    \[(1 + h_1 + h_{\lfloor n/2 \rfloor}) - \frac{h_{\lfloor n/2 \rfloor - 1}}{1-2\sin(\pi/(2n))} = \frac{1}{2 \sin \left(\frac{\pi }{2 n}\right)-1}+2 \cos \left(\frac{\pi }{n}\right)+1\geq 0,\]
    meaning that our inequality becomes
    \[a < \frac{h_{\lfloor n/2 \rfloor - 1}}{1-2\sin(\pi/(2n))} \leq 1 + h_1 + h_{\lfloor n/2 \rfloor}.\] 
    
    Thus for all $n > 3$, we have that $a < 1 + h_1 + h_{\lfloor n/2 \rfloor}$ for odd $n$, and $a < 1 + h_1 + h_{\lfloor n/2 \rfloor - 1}$ for even $n$; this is our bound $q$ for Condition 1 in this case.

By Lemma \ref{lem:order}, we note that the vectors $\nu_j$ for all $j < \lfloor n/2 \rfloor - 1$ are steeper than $\nu_{\lfloor n/2 \rfloor - 1}$, meaning that they all violate Condition 2. Moreover, we know from the same lemma that all $\sigma$ vectors fail Condition 3, because we are comparing to a $\nu$ vector.
    
\textbf{Subcase 3.1.} We begin with saddle connections staying in the staircase. 

\textbf{Subsubcase 3.1.1.} Consider saddle connections staying in the staircase linking any vertex to a vertex $L_j$ for any $j$. Since these saddle connections stay entirely within the staircase, we note that they are necessarily steeper than the vector $\nu_j$ joining $L_{j-1}$ and $L_j$ (they enter the rectangle containing $L_{j-1}$ and $L_j$ at a point farther to the right than $L_{j-1}$). This means they violate Condition 2 since $\nu_j$ is at least as steep as $\nu_{\lfloor n/2 \rfloor - 1}$. 
  

\textbf{Subsubcase 3.1.2.} We proceed to saddle connections ending at $R_k$ for any $m$. Condition 3 is violated, as this vector necessarily has larger slope and lower $a$ value than $\sigma_k$. 


\textbf{Subcase 3.2.} This rules out all saddle connections that do not pass through the edges of the staircase, so we now proceed to saddle connections that do pass through the edges of the staircase. 
    
\textbf{Subsubcase 3.2.1.} Any saddle connection which first passes through a horizontal edge to the left of some $L_k$ is necessarily steeper than $\nu_{k-1}$, violating Condition 2. 

    
\textbf{Subsubcase 3.2.2.} If a saddle connection first passes through the vertical edge below some $R_k$ for $k > 0$, then this vector necessarily has larger slope and lower $a$ value than $\sigma_k$, meaning that it violates Condition 3.

    
\textbf{Subsubcase 3.2.3.} The saddle connections that first pass through the horizontal edge between $L_0$ and $R_0$ will now be eliminated. The only possible starting vertex for such a saddle connection is $L_0$; in this case, the slope of the vector is higher than 1, meaning Condition 2 is violated.
Note that all other saddle connections which pass through the edge between $L_0$ and $R_0$ first pass through a different edge, and these have been eliminated in previous cases.

    
\textbf{Subsubcase 3.2.4.} If $n$ is even, we must eliminate the vectors passing through the vertical edge below $L_{n/2}$. Notice that such a vector necessarily crosses $h_{n/2 - 1}$ at least twice. Thus, such a vector cannot travel more than $2$ horizontal distances or else it would violate Condition 1. In fact, such a vector must travel a horizontal distance of exactly $2h_{n/2- 1}$. 
    
We first suppose that $n \geq 6$. If we show that $h_{\lfloor n/2 \rfloor - 1} \geq 1 + h_1$, then this vector would violate Condition 1. We compute
    \[h_{\lfloor n/2 \rfloor - 1} - 1 - h_1 = -2 \cos \left(\frac{\pi }{n}\right)+\cot \left(\frac{\pi }{2 n}\right)-2.\] The derivative of this function is \[\frac{\pi  \left(\csc ^2\left(\frac{\pi }{2 n}\right)-4 \sin \left(\frac{\pi
   }{n}\right)\right)}{2 n^2},\] which is positive for $n \geq 3$. At $n = 6$, this difference evaluates to $0$, meaning that the difference is indeed nonnegative for all $n \geq 6$. In other words, Condition 1 is violated for $n \geq 6$, leaving us with the case $n = 4$.
   
Now consider the case $n = 4$. In this case, $\cvec{a}{b}$ must begin at $R_0$ or $L_1$ and end at either $R_1$ or $L_2$.


If $\cvec{a}{b}$ begins at $R_0$, it must pass through the vertical face below $L_2$ and either (a) end at $L_2$, or (b) pass through the horizontal face to the left of $L_2$ and end at $R_1$. In case (a), the vector is $(2h_1, 1 + v_1)$, which violates Condition 3:
   \[a - b(h_1 - 1) = 2 + 2v_1 - (1 + v_1)v_1 = 2 + v_1 - v_1^2 = \sqrt{2} > 1\]
   In case (b), the vector is $\cvec{2h_1}{2 + v_1}$, which has slope $\frac{1 + v_1/2}{h_1} > \frac{v_1}{h_1}$, which violates Condition 2.
   If $\cvec{a}{b}$ instead begins at $L_1$, then it must pass through the vertical edge below $L_2$, and either (c) end at $L_2$ or (d) pass through the horizontal edge to the left of $L_2$ and end at $R_1$.
   In case (c), the vector is $\cvec{2h_1}{v_1}$ which has smaller slope and larger $a$ than $\sigma_1 = \cvec{h_1}{1}$, so it violates Condition 3.
   The vector in case (d) is $\cvec{2h_1}{1 + v_1}$, which we already eliminated. This concludes the proof for $n = 4$ and hence for the case where $n$ is even.

\textbf{Subsubcase 3.2.5.} Finally, we will eliminate the saddle connections passing through the horizontal edge to the right of $L_{\lfloor n/2 \rfloor}$, which forms an edge of the staircase when $n$ is odd and greater than 3.


Such a saddle connection can travel no more than $2$ horizontal distances, otherwise it will violate Condition 1 (since one of these horizontal distances must be $h_{\lfloor n/2 \rfloor}$.
So, letting $k = \lfloor n/2 \rfloor - 1$, a candidate saddle connection must start at $L_k$, $R_{k - 1}$ or $R_k$.
    
If $\cvec{a}{b}$ begins at $L_k$, we observe that in order to avoid violating Condition 3, it must end at $R_{k + 1}$. The vector must therefore be $(h_k + h_{k + 1}, 2v_k)$. For $n > 5$, this violates Condition 1, as $h_1 + 1 \leq h_{\lfloor n/2 \rfloor}$:
    \[h_{\lfloor n/2 \rfloor} - h_1 - 1 = \left(\sin \left(\frac{\pi }{2 n}\right)+\sin \left(\frac{3 \pi }{2 n}\right)-1\right)
   \left(-\csc \left(\frac{\pi }{2 n}\right)\right) > 0.\]
   For $n = 5$, this vector violates Condition 3:
   \[a - b \bigg( \frac{h_{k + 1} - 1}{v_k} \bigg) = h_k + h_{k + 1} - 2v_k \bigg( \frac{h_{k + 1} - 1}{v_k} \bigg) = 2 + h_k - h_{k + 1} \geq 1\]
    
    Now suppose $\cvec{a}{b}$ begins at $R_{k - 1}$.
    The only way for $\cvec{a}{b}$ to avoid violating Condition 1 is by ending at $R_{k + 1}$ after passing through the horizontal edge at least once.
    For $n > 5$, this violates Condition 1, as $h_1 + 1 \leq h_{\lfloor n/2 \rfloor}$.
   For $n \leq 5$, consider the vector crossing the edge once, $(h_k + h_{k + 1}, v_{k - 1} + 2v_k)$.
   For $n = 5$ this vector has slope greater than $v_k/h_k$, violating Condition 2.

   Since any $\cvec{a}{b}$ crossing the edge more than once will have even larger slope, this case also violates Condition 2.
    
Finally, suppose that $\cvec{a}{b}$ begins at $R_k$. Then, it must end at $R_{k + 1}$, $L_{k + 1}$, or $R_k$ or else it would violate Condition 1. If it ends at $R_{k + 1}$, then because the lower bound of $1$ is the slope when the vector crosses the edge exactly once, this vector has slope at least $1$, which violates Condition 2.
    
The only possible $\cvec{a}{b}$ ending at $L_{k + 1}$ in this case is the vector passing once through the horizontal edge to the left of $R_{k + 1}$, then through the vertical edge below $R_{k + 1}$, before ending at $L_{k + 1}$. This is the same vector as the saddle connection beginning at $L_{\lfloor n/2 \rfloor - 1}$, then passing through the horizontal edge to the right of $L_{k + 1}$, and finally ending at $R_{k + 1}$, which we have eliminated earlier.
    
The only other possible vector that does not travel horizontally across more than two rectangles starts at $R_k$, passes through the horizontal edge to the right of $L_{k + 1}$, then passes through the vertical edge below $R_{k + 1}$, then passes through the horizontal edge to the left of $L_{k + 1}$, and finally ends at $R_k$. This is the same vector as the one beginning at $R_{k - 1}$ and ending at $R_{k + 1}$ discussed earlier, which we have eliminated. This completes this case for all $n > 3$.

\textbf{Subcase 3.3.} Finally, suppose $n = 3$. We directly check that no other vectors win in the region corresponding to $\nu_{\lfloor n/2 \rfloor - 1}$. In this case, $h_{\lfloor n/2 \rfloor - 1} = h_0 = 1$, $v_{\lfloor n/2 \rfloor - 1} = v_0 = 1$, $y > -1$, and $h_1 = 2$. Suppose that a vector $\cvec{a}{b}$ wins. Then, we must have that $a + by \leq 1$. As $y > -1$ in this region, this gives us $a - b < 1$, or equivalently, $a < b + 1$. We must also have that $a > b$, or else the slope of the vector wouldn't be less than that of $\nu_{\lfloor n/2 \rfloor - 1}$. Thus, $b < a < b + 1$, which is a contradiction since all lengths and heights in the staircase when $n = 3$ are integers.
    
Thus we have shown that $\nu_{\lfloor n/2 \rfloor - 1}$ wins for all $(1, y) \in \Omega_1$ with $\frac{1 - h_{\lceil n/2 \rceil - 1}}{v_{\lceil n/2 \rceil - 2}} < y \leq \frac{1 - h_{\lfloor n/2 \rfloor - 1}}{v_{\lfloor n/2 \rfloor - 1}}$, which is simply $P_{\lfloor n/2 \rfloor + 1} \cap \{x = 1\}$.

%
%
%
%

\textbf{Case 4.} We now prove that the vectors $\sigma_i$ win on $P_{n + 1 - i} \cap \{x = 1\}$. Begin by fixing $i$.

We first compute a bound $q$ for Condition 1 in this case. In order for the image of some $\cvec{a}{b}$ to have horizontal length at most $1$, we must have $a + by \leq 1$. We also have $a \geq 1$ and $b \geq 1$, as these are the shortest horizontal and vertical distances on the $S'$, respectively. Moreover, we must have that $b/a < v_{i-1}/h_i$, as this is the slope of the candidate vector for $P_i$. Finally, we have that $y > \frac{1 - h_{i-1}}{v_{i-2}}$. Combining these facts, we find
    \begin{align*}
        a \leq 1 - by < 1 - \frac{v_{i-1}}{h_i}ay < 1 - \frac{v_{i-1}}{h_i}\left(\frac{1-h_{i-1}}{v_{i-2}}\right)a \implies \\[5pt]
        a\left(1 + \frac{v_{i-1}(1-h_{i-1})}{h_iv_{i-2}}\right) < 1 \implies \\ a\left(h_i + \frac{v_{i-1}(1-h_{i-1})}{v_{i-2}}\right) = a\left(\frac{v_{i-1}-1}{v_{i-2}}\right) < h_i \implies \\[5pt]
        a < h_i\left(\frac{v_{i-2}}{v_{i-1}-1}\right) = 1 + h_{i-1} + \frac{v_{i-2}}{v_{i-1}-1} < 1 + h_{i-1} + \frac{v_i}{v_{i-1}-1}.
    \end{align*}
    In Case 2, we showed
    \[\frac{v_i}{v_{i-1}-1} < h_1,\]
    so we can conclude that
    \[a < 1 + h_1 + h_{i-1}\]
    This gives us the desired bound $q$ for Condition 1.

By Lemma \ref{lem:order}, we note that vectors $\nu_k$ for all $k$ and $\sigma_k$ with $k < i$ are steeper than $\sigma_i$, violating Condition 2.
Moreover, Lemma \ref{lem:forder} implies that all $\sigma_k$ with $k > i$  violate Condition 3. 
    
\textbf{Subcase 4.1.} We begin with saddle connections staying in the staircase.

\textbf{Subsubcase 4.1.1.} First consider any saddle connection staying in the staircase that ends at a vertex $L_k$ for any $k$. Since these saddle connections stay entirely within the staircase, we note that they are necessarily steeper than the vector $\nu_k$ joining $L_{k-1}$ and $L_k$ (they enter the rectangle containing $L_{k-1}$ and $L_k$ at a point farther to the right than $L_{k-1}$). This means they violate Condition 2 as $\nu_k$ violates Condition 2. This can be seen in Figure \ref{fig:end_at_l_k_case_4}.

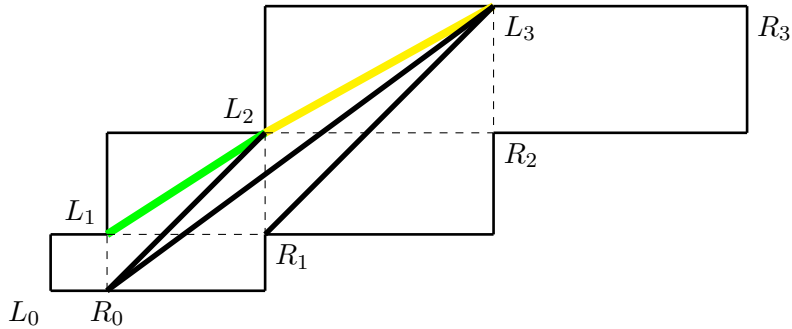
\begin{figure}[h!]
    \centering
\begin{tikzpicture}[scale=0.75]
\coordinate[label=below left:$L_0$] (a) at (0,0);
\coordinate (b) at (0,1);
\coordinate[label=below:$R_0$] (c) at (1,0);
\coordinate[label=above left:$L_1$] (d) at (1,1);
\coordinate (e) at (3.801937735804838,0);
\coordinate[label=below right:$R_1$] (f) at (3.801937735804838,1);
\coordinate (g) at (1,2.801937735804838);
\coordinate[label=above left:$L_2$] (h) at (3.801937735804838,2.801937735804838);
\coordinate[label=below right:$R_2$] (i) at (7.850855075327144,2.801937735804838);
\coordinate (j) at (7.850855075327144,1);
\coordinate[label=below right:$L_3$] (k) at (7.850855075327144,5.048917339522305);
\coordinate (l) at (3.801937735804838,5.048917339522305);
\coordinate[label=below right:$R_3$] (m) at (12.344814282762078,5.048917339522305);
\coordinate (n) at (12.344814282762078,2.801937735804838);

\draw [line width=1pt] (a) -- (b) node at (-0.4, 0.5) {};
\draw [line width=1pt] (b)-- (d);
\draw [line width=1pt] (a)-- (c) node at (0.5, -0.4) {};
\draw [line width=1pt] (c)-- (e) node at (2.4, -0.4) {};
\draw [line width=1pt] (e)-- (f);
\draw [line width=1pt] (f)-- (j) node at (5.825, 0.6) {};
\draw [line width=1pt] (j)-- (i);
\draw [line width=1pt] (i)-- (n) node at (10.1, 2.4) {};
\draw [line width=1pt] (n)-- (m);
\draw [line width=1pt] (m)-- (k);
\draw [line width=1pt] (k)-- (l);
\draw [line width=1pt] (l)-- (h) node at (3.4, 3.92) {};
\draw [line width=1pt] (h)-- (g);
\draw [line width=1pt] (g)-- (d) node at (0.6, 1.9) {};
\draw [line width=3pt, color=yellow] (h)--(k) node[midway, above, color=black] {};
\draw [line width=3pt, color=green] (d)--(h);
\draw [line width=2pt, color=black] (c)--(h);
\draw [line width=2pt, color=black] (c)--(k);
\draw [line width=2pt, color=black] (f)--(k);
\draw [dashed] (c) -- (d);
\draw[dashed] (d) -- (f);
\draw [dashed] (h) -- (f);
\draw[dashed] (h) -- (i);
\draw [dashed] (k)-- (i);
\end{tikzpicture}
    \caption{Vectors ending at $L_k$ for any $k$ (shown in black here) are necessarily steeper than the $\nu$ vector that ends at $L_k$ (shown in green or yellow), which is in turn steeper than every $\sigma$ vector.}
    \label{fig:end_at_l_k_case_4}
\end{figure}

\textbf{Subsubcase 4.1.2.} We now eliminate vectors staying in the staircase linking $R_j$ and $R_k$ (with $k > j + 1$). If $k < i$, then Condition 3 is violated, as this vector is necessarily shallower and longer than the vector $\sigma_k$ which also ends at $R_k$ as seen in Figure \ref{fig:end_at_r_k_shallow_case_4}. On the other hand, if $j \geq i$, then this saddle connection is necessarily steeper than the saddle connection $\sigma_{j + 1}$ starting at $R_j$ and ending at $R_{j+1}$, as seen in Figure \ref{fig:end_at_r_k_steep_case_4}, meaning that these vectors violate Condition 2.

\begin{figure}[h!]
    \centering
\begin{tikzpicture}[scale=0.75]
\coordinate[label=below left:$L_0$] (a) at (0,0);
\coordinate (b) at (0,1);
\coordinate[label=below:$R_0$] (c) at (1,0);
\coordinate[label=above left:$L_1$] (d) at (1,1);
\coordinate (e) at (3.801937735804838,0);
\coordinate[label=below right:$R_1$] (f) at (3.801937735804838,1);
\coordinate (g) at (1,2.801937735804838);
\coordinate[label=above left:$L_2$] (h) at (3.801937735804838,2.801937735804838);
\coordinate[label=below right:$R_2$] (i) at (7.850855075327144,2.801937735804838);
\coordinate (j) at (7.850855075327144,1);
\coordinate[label=below right:$L_3$] (k) at (7.850855075327144,5.048917339522305);
\coordinate (l) at (3.801937735804838,5.048917339522305);
\coordinate[label=below right:$R_3$] (m) at (12.344814282762078,5.048917339522305);
\coordinate (n) at (12.344814282762078,2.801937735804838);

\draw [line width=1pt] (a) -- (b) node at (-0.4, 0.5) {};
\draw [line width=1pt] (b)-- (d);
\draw [line width=1pt] (a)-- (c) node at (0.5, -0.4) {};
\draw [line width=1pt] (c)-- (e) node at (2.4, -0.4) {};
\draw [line width=1pt] (e)-- (f);
\draw [line width=1pt] (f)-- (j) node at (5.825, 0.6) {};
\draw [line width=1pt] (j)-- (i);
\draw [line width=1pt] (i)-- (n) node at (10.1, 2.4) {};
\draw [line width=1pt] (n)-- (m);
\draw [line width=1pt] (m)-- (k);
\draw [line width=1pt] (k)-- (l);
\draw [line width=1pt] (l)-- (h) node at (3.4, 3.92) {};
\draw [line width=1pt] (h)-- (g);
\draw [line width=1pt] (g)-- (d) node at (0.6, 1.9) {};
\draw [line width=3pt, color=orange] (f)--(i) node[midway, below, color=black] {$\sigma_2$};
\draw [line width=2pt, color=black] (c)--(i);
\draw [dashed] (c) -- (d);
\draw[dashed] (d) -- (f);
\draw [dashed] (h) -- (f);
\draw[dashed] (h) -- (i);
\draw [dashed] (k)-- (i);
\end{tikzpicture}
    \caption{Let $n = 7$ and $i = 3$ (so the vector we are trying to show victory for is $\sigma_3$). Vectors ending at $R_k$ for any $k < i$ (shown in black here) are necessarily shallower and longer than the $\sigma$ vector that ends at $R_k$ (in orange), which in turn violates Condition 3.}
    \label{fig:end_at_r_k_shallow_case_4}
\end{figure}
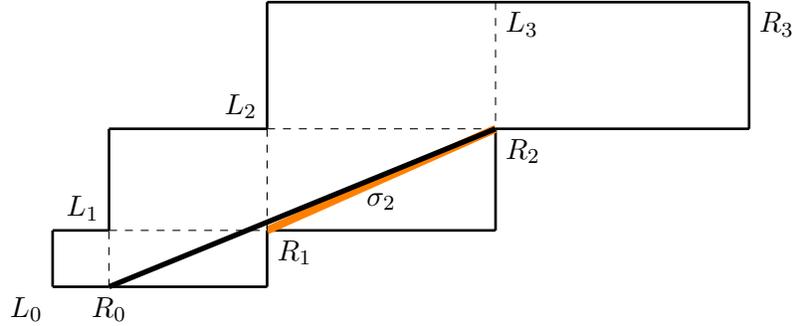

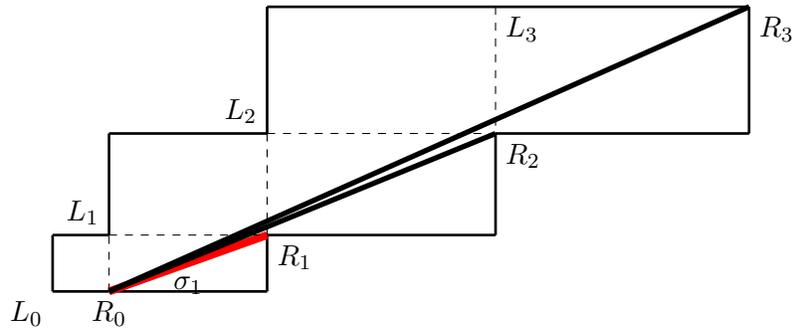
\begin{figure}[h!]
    \centering
\begin{tikzpicture}[scale=0.75]
\coordinate[label=below left:$L_0$] (a) at (0,0);
\coordinate (b) at (0,1);
\coordinate[label=below:$R_0$] (c) at (1,0);
\coordinate[label=above left:$L_1$] (d) at (1,1);
\coordinate (e) at (3.801937735804838,0);
\coordinate[label=below right:$R_1$] (f) at (3.801937735804838,1);
\coordinate (g) at (1,2.801937735804838);
\coordinate[label=above left:$L_2$] (h) at (3.801937735804838,2.801937735804838);
\coordinate[label=below right:$R_2$] (i) at (7.850855075327144,2.801937735804838);
\coordinate (j) at (7.850855075327144,1);
\coordinate[label=below right:$L_3$] (k) at (7.850855075327144,5.048917339522305);
\coordinate (l) at (3.801937735804838,5.048917339522305);
\coordinate[label=below right:$R_3$] (m) at (12.344814282762078,5.048917339522305);
\coordinate (n) at (12.344814282762078,2.801937735804838);

\draw [line width=1pt] (a) -- (b) node at (-0.4, 0.5) {};
\draw [line width=1pt] (b)-- (d);
\draw [line width=1pt] (a)-- (c) node at (0.5, -0.4) {};
\draw [line width=1pt] (c)-- (e) node at (2.4, -0.4) {};
\draw [line width=1pt] (e)-- (f);
\draw [line width=1pt] (f)-- (j) node at (5.825, 0.6) {};
\draw [line width=1pt] (j)-- (i);
\draw [line width=1pt] (i)-- (n) node at (10.1, 2.4) {};
\draw [line width=1pt] (n)-- (m);
\draw [line width=1pt] (m)-- (k);
\draw [line width=1pt] (k)-- (l);
\draw [line width=1pt] (l)-- (h) node at (3.4, 3.92) {};
\draw [line width=1pt] (h)-- (g);
\draw [line width=1pt] (g)-- (d) node at (0.6, 1.9) {};
\draw [line width=3pt, color=red] (c)--(f) node[midway, below, color=black] {$\sigma_1$};
\draw [line width=2pt, color=black] (c)--(i);
\draw [line width=2pt, color=black] (c)--(m);
\draw [dashed] (c) -- (d);
\draw[dashed] (d) -- (f);
\draw [dashed] (h) -- (f);
\draw[dashed] (h) -- (i);
\draw [dashed] (k)-- (i);
\end{tikzpicture}
    \caption{Let $n = 7$ and $i = 2$ (so the vector we are trying to show victory for is $\sigma_2$). Vectors starting at $R_j$ for any $j \leq i$ (shown in black here) are necessarily steeper than the $\sigma$ vector that starts at $R_j$ (in red), which in turn violates Condition 2.}
    \label{fig:end_at_r_k_steep_case_4}
\end{figure}

Finally, we check the case where $j < i$ and $k \geq i$. If $k - j \geq 3$, then this saddle connection must travel through at least $h_k$ and two other horizontal lengths, which violates Condition 1 (as one of these lengths must be at least $1$ and the other length, being distinct from the first length, must be at least $h_1$).
This implies that $k - j = 2$, so the only cases we need to consider are $k = i, j = i - 2$ and $k = i+1, j = i-1$. In the latter case, the vector is steeper than $\sigma_i$ (it begins at the same vertex as $\sigma_i$ but travels higher than $R_i$), violating Condition 2. 

This leaves us with the former case. We show that in this case, $f\cvec{a}{b} > 1$, violating Condition 3. We have that \begin{align*}h_{i-1} + h_i - (v_{i-1} + v_{i-2})\left(\frac{h_{i-1}-1}{v_{i-2}}\right) &=
    v_{i-2} + 2v_{i-1} + v_i - (v_{i-1} + v_{i-2})\left(\frac{v_{i-2} + v_{i-1}-1}{v_{i-2}}\right) \\[10pt] &= v_{i-2} + 2v_{i-1} + v_i - v_{i-2} - v_{i-1} + 1 - \frac{v_{i-2}v_{i-1}+v_{i-1}^2 - v_{i-1}}{v_{i-2}}\\[10pt] &=
    \frac{v_{i-2}v_{i-1}+v_{i-2}v_i+v_{i-2}-v_{i-2}v_{i-1}-v_{i-1}^2+v_{i-1}}{v_{i-2}}\\[10pt] &= \frac{-1+v_{i-1}+v_{i-2}}{v_{i-2}}=1+\frac{v_{i-1}-1}{v_{i-2}} > 1,\end{align*} as we know that $i \geq 2$, meaning $v_{i-1} > 1$.
Thus we have eliminated all vectors connecting some $R_j$ to some $R_k$ without passing through an edge of $S'$.
    
\textbf{Subsubcase 4.1.3.}  We now eliminate the vectors starting at $L_j$ and ending at $R_k$ for $k \geq j+1$. If $k < i$, then Condition 3 is violated as this vector is necessarily shallower and longer than $\sigma_k$, which also ends at $R_k$ (which we noted earlier violates Condition 3), as seen in Figure \ref{fig:start_at_l_k_end_at_r_k_shallow_case_4}.

\begin{figure}[h!]
    \centering
\begin{tikzpicture}[scale=0.75]
\coordinate[label=below left:$L_0$] (a) at (0,0);
\coordinate (b) at (0,1);
\coordinate[label=below:$R_0$] (c) at (1,0);
\coordinate[label=above left:$L_1$] (d) at (1,1);
\coordinate (e) at (3.801937735804838,0);
\coordinate[label=below right:$R_1$] (f) at (3.801937735804838,1);
\coordinate (g) at (1,2.801937735804838);
\coordinate[label=above left:$L_2$] (h) at (3.801937735804838,2.801937735804838);
\coordinate[label=below right:$R_2$] (i) at (7.850855075327144,2.801937735804838);
\coordinate (j) at (7.850855075327144,1);
\coordinate[label=below right:$L_3$] (k) at (7.850855075327144,5.048917339522305);
\coordinate (l) at (3.801937735804838,5.048917339522305);
\coordinate[label=below right:$R_3$] (m) at (12.344814282762078,5.048917339522305);
\coordinate (n) at (12.344814282762078,2.801937735804838);

\draw [line width=1pt] (a) -- (b) node at (-0.4, 0.5) {};
\draw [line width=1pt] (b)-- (d);
\draw [line width=1pt] (a)-- (c) node at (0.5, -0.4) {};
\draw [line width=1pt] (c)-- (e) node at (2.4, -0.4) {};
\draw [line width=1pt] (e)-- (f);
\draw [line width=1pt] (f)-- (j) node at (5.825, 0.6) {};
\draw [line width=1pt] (j)-- (i);
\draw [line width=1pt] (i)-- (n) node at (10.1, 2.4) {};
\draw [line width=1pt] (n)-- (m);
\draw [line width=1pt] (m)-- (k);
\draw [line width=1pt] (k)-- (l);
\draw [line width=1pt] (l)-- (h) node at (3.4, 3.92) {};
\draw [line width=1pt] (h)-- (g);
\draw [line width=1pt] (g)-- (d) node at (0.6, 1.9) {};
\draw [line width=3pt, color=orange] (f)--(i) node[midway, below, color=black] {$\sigma_2$};
\draw [line width=2pt, color=black] (a)--(i);
\draw [line width=2pt, color=black] (d)--(i);
\draw [dashed] (c) -- (d);
\draw[dashed] (d) -- (f);
\draw [dashed] (h) -- (f);
\draw[dashed] (h) -- (i);
\draw [dashed] (k)-- (i);
\end{tikzpicture}
    \caption{Let $n = 7$ and $i = 3$ (so the vector we are trying to show victory for is $\sigma_3$). Vectors starting at some $L_j$ and ending at $R_k$ for any $k < i$ (shown in black here) are necessarily shallower and longer than the $\sigma$ vector that ends at $R_k$ (in orange), which in turn violates Condition 3.}
    \label{fig:start_at_l_k_end_at_r_k_shallow_case_4}
\end{figure}
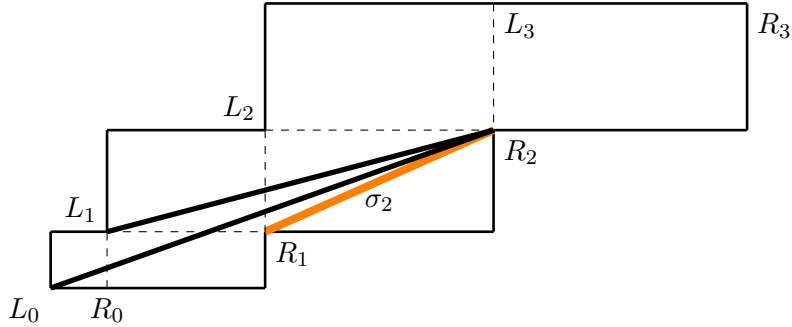

Suppose that $k \geq i$. If $k - j \geq 2$, this saddle connection violates Condition 3 (as it encompasses at least three horizontal distances on the staircase, of which one is at least as long as the length $h_i$, and the other two of which are distinct). Thus, we must have $k - j = 1$, in which case the slope of the saddle connection is precisely the aspect ratio $2 + 2\cos(\pi/n)$ (it will be the diagonal of some rectangle, as seen in Figure \ref{fig:weird_aspect_ratio}), which fails Condition 4. This rules out all vectors that do not pass through the edges of the staircase.

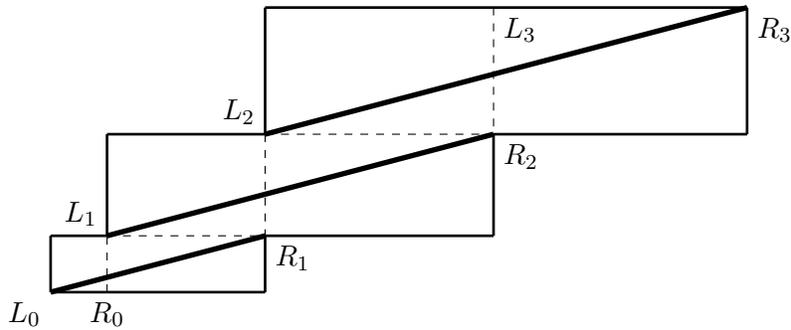
\begin{figure}[h!]
    \centering
\begin{tikzpicture}[scale=0.75]
\coordinate[label=below left:$L_0$] (a) at (0,0);
\coordinate (b) at (0,1);
\coordinate[label=below:$R_0$] (c) at (1,0);
\coordinate[label=above left:$L_1$] (d) at (1,1);
\coordinate (e) at (3.801937735804838,0);
\coordinate[label=below right:$R_1$] (f) at (3.801937735804838,1);
\coordinate (g) at (1,2.801937735804838);
\coordinate[label=above left:$L_2$] (h) at (3.801937735804838,2.801937735804838);
\coordinate[label=below right:$R_2$] (i) at (7.850855075327144,2.801937735804838);
\coordinate (j) at (7.850855075327144,1);
\coordinate[label=below right:$L_3$] (k) at (7.850855075327144,5.048917339522305);
\coordinate (l) at (3.801937735804838,5.048917339522305);
\coordinate[label=below right:$R_3$] (m) at (12.344814282762078,5.048917339522305);
\coordinate (n) at (12.344814282762078,2.801937735804838);

\draw [line width=1pt] (a) -- (b) node at (-0.4, 0.5) {};
\draw [line width=1pt] (b)-- (d);
\draw [line width=1pt] (a)-- (c) node at (0.5, -0.4) {};
\draw [line width=1pt] (c)-- (e) node at (2.4, -0.4) {};
\draw [line width=1pt] (e)-- (f);
\draw [line width=1pt] (f)-- (j) node at (5.825, 0.6) {};
\draw [line width=1pt] (j)-- (i);
\draw [line width=1pt] (i)-- (n) node at (10.1, 2.4) {};
\draw [line width=1pt] (n)-- (m);
\draw [line width=1pt] (m)-- (k);
\draw [line width=1pt] (k)-- (l);
\draw [line width=1pt] (l)-- (h) node at (3.4, 3.92) {};
\draw [line width=1pt] (h)-- (g);
\draw [line width=1pt] (g)-- (d) node at (0.6, 1.9) {};
\draw [line width=2pt, color=black] (a)--(f);
\draw [line width=2pt, color=black] (d)--(i);
\draw[line width = 2pt, color=black] (h)--(m);
\draw [dashed] (c) -- (d);
\draw[dashed] (d) -- (f);
\draw [dashed] (h) -- (f);
\draw[dashed] (h) -- (i);
\draw [dashed] (k)-- (i);
\end{tikzpicture}
    \caption{Let $n = 7$. Vectors starting at $L_j$ and ending at $R_{j+1}$ are precisely the diagonals of the large rectangles, which have aspect ratio equal to $2 + 2\cos(\pi/n)$.}
    \label{fig:weird_aspect_ratio}
\end{figure}

 \textbf{Subcase 4.2.}   We now proceed to vectors that do pass through the edges of the staircase.

 \textbf{Subsubcase 4.2.1.} Any of these vectors that starts by passing through the horizontal edge containing $L_k$ is necessarily steeper than the $\nu$ vector ending at $L_k$, meaning these vectors violate Condition 2, as in Figure \ref{fig:pass_through_horizontal_case_4}.

\begin{figure}[h!]
    \centering
\begin{tikzpicture}[scale=0.75]
\coordinate[label=below left:$L_0$] (a) at (0,0);
\coordinate (b) at (0,1);
\coordinate[label=below:$R_0$] (c) at (1,0);
\coordinate[label=above left:$L_1$] (d) at (1,1);
\coordinate (e) at (3.801937735804838,0);
\coordinate[label=below right:$R_1$] (f) at (3.801937735804838,1);
\coordinate (g) at (1,2.801937735804838);
\coordinate[label=above left:$L_2$] (h) at (3.801937735804838,2.801937735804838);
\coordinate[label=below right:$R_2$] (i) at (7.850855075327144,2.801937735804838);
\coordinate (j) at (7.850855075327144,1);
\coordinate[label=below right:$L_3$] (k) at (7.850855075327144,5.048917339522305);
\coordinate (l) at (3.801937735804838,5.048917339522305);
\coordinate[label=below right:$R_3$] (m) at (12.344814282762078,5.048917339522305);
\coordinate (n) at (12.344814282762078,2.801937735804838);

\draw [line width=1pt] (a) -- (b) node at (-0.4, 0.5) {};
\draw [line width=1pt] (b)-- (d);
\draw [line width=1pt] (a)-- (c) node at (0.5, -0.4) {};
\draw [line width=1pt] (c)-- (e) node at (2.4, -0.4) {};
\draw [line width=1pt] (e)-- (f);
\draw [line width=1pt] (f)-- (j) node at (5.825, 0.6) {};
\draw [line width=1pt] (j)-- (i);
\draw [line width=1pt] (i)-- (n) node at (10.1, 2.4) {};
\draw [line width=1pt] (n)-- (m);
\draw [line width=1pt] (m)-- (k);
\draw [line width=1pt] (k)-- (l);
\draw [line width=1pt] (l)-- (h) node at (3.4, 3.92) {};
\draw [line width=1pt] (h)-- (g) coordinate[midway] (z);
\draw [line width=1pt] (g)-- (d) node at (0.6, 1.9) {};
\draw [line width=3pt, color=green] (d)--(h);
\draw [line width=2pt, color=black] (c)--(z);
\draw [line width=2pt, color=black] (d)--(z);
\draw [dashed] (c) -- (d);
\draw[dashed] (d) -- (f);
\draw [dashed] (h) -- (f);
\draw[dashed] (h) -- (i);
\draw [dashed] (k)-- (i);
\end{tikzpicture}
    \caption{For instance, any vector passing through the side of length $h_1$ is steeper than $\nu_1$ (green vector), which in turn is steeper than every $\sigma$ vector.}
    \label{fig:pass_through_horizontal_case_4}
\end{figure}
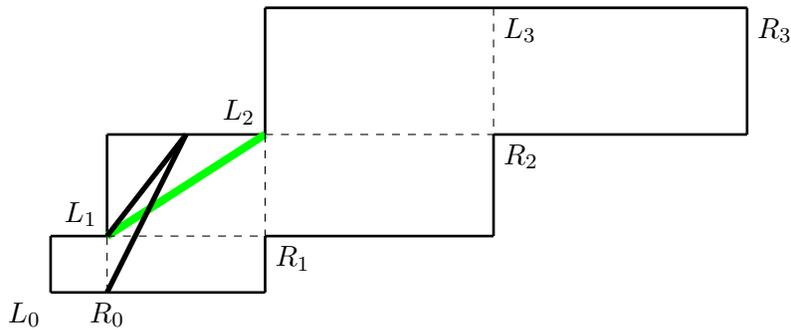
    
 \textbf{Subsubcase 4.2.2.}   Now consider a vector that starts by passing through a right wall. Suppose the vector started at $L_j$ and passed the vertical edge below $R_{j+1}$. Then, this vector would be shallower than the aspect ratio of the rectangle and can thus be eliminated on account of failing Condition 4. If the vector starts at $R_k$ or $L_{k}$ with $k > i$ and travels across more than three horizontal distances, then it violates Condition 1, as one of the horizontal distances it traverses is at least as long as $h_i$. 
    
    Now suppose that the saddle connection first crosses the right edge below $R_k$ for some $k < i$. Then, this saddle connection is longer and shallower than the $\sigma$ vector in that rectangle, meaning that it violates Condition 3, as in Figure \ref{fig:pass_through_vertical_case_4}.

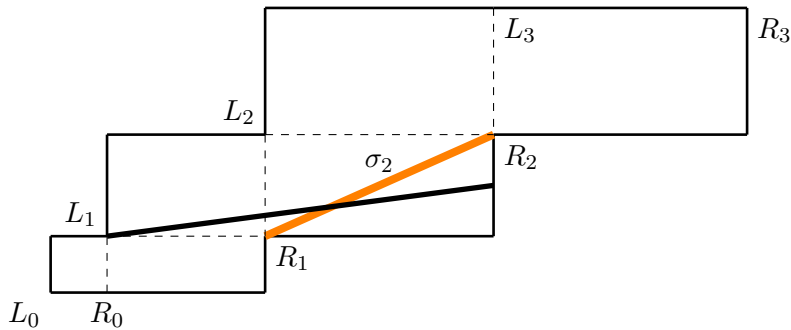
\begin{figure}[h!]
    \centering
\begin{tikzpicture}[scale=0.75]
\coordinate[label=below left:$L_0$] (a) at (0,0);
\coordinate (b) at (0,1);
\coordinate[label=below:$R_0$] (c) at (1,0);
\coordinate[label=above left:$L_1$] (d) at (1,1);
\coordinate (e) at (3.801937735804838,0);
\coordinate[label=below right:$R_1$] (f) at (3.801937735804838,1);
\coordinate (g) at (1,2.801937735804838);
\coordinate[label=above left:$L_2$] (h) at (3.801937735804838,2.801937735804838);
\coordinate[label=below right:$R_2$] (i) at (7.850855075327144,2.801937735804838);
\coordinate (j) at (7.850855075327144,1);
\coordinate[label=below right:$L_3$] (k) at (7.850855075327144,5.048917339522305);
\coordinate (l) at (3.801937735804838,5.048917339522305);
\coordinate[label=below right:$R_3$] (m) at (12.344814282762078,5.048917339522305);
\coordinate (n) at (12.344814282762078,2.801937735804838);

\draw [line width=1pt] (a) -- (b) node at (-0.4, 0.5) {};
\draw [line width=1pt] (b)-- (d);
\draw [line width=1pt] (a)-- (c) node at (0.5, -0.4) {};
\draw [line width=1pt] (c)-- (e) node at (2.4, -0.4) {};
\draw [line width=1pt] (e)-- (f);
\draw [line width=1pt] (f)-- (j) node at (5.825, 0.6) {};
\draw [line width=1pt] (j)-- (i) coordinate[midway] (z) {};
\draw [line width=1pt] (i)-- (n) node at (10.1, 2.4) {};
\draw [line width=1pt] (n)-- (m);
\draw [line width=1pt] (m)-- (k);
\draw [line width=1pt] (k)-- (l);
\draw [line width=1pt] (l)-- (h) node at (3.4, 3.92) {};
\draw [line width=1pt] (h)-- (g);
\draw [line width=1pt] (g)-- (d) node at (0.6, 1.9) {};
\draw [line width=3pt, color=orange] (f)--(i) node[midway, above, color=black] {$\sigma_2$};
\draw [line width=2pt, color=black] (d)--(z);
\draw [dashed] (c) -- (d);
\draw[dashed] (d) -- (f);
\draw [dashed] (h) -- (f);
\draw[dashed] (h) -- (i);
\draw [dashed] (k)-- (i);
\end{tikzpicture}
    \caption{Fix $i = 3$. For instance, we see that any vector passing through $v_1$ is longer and shallower than $\sigma_2$, which fails Condition 3.}
    \label{fig:pass_through_vertical_case_4}
\end{figure}

Proceed to the case of saddle connections that first cross the right edge below $R_k$ for some $i \leq k$.
Then, if the vector originates from some $L_j$, we know that the vector must have originated at $L_{k-2}$ or before, or else it would have slope shallower than the aspect ratio.
However, this vector must cross at least three horizontal distances and therefore violates Condition 1.
Thus, in this case, the vector must originate from some $R_j$.
In particular, to avoid violating Condition 1, it must originate from $R_{k-1}$ or $R_{k-2}$.
However, in either case, the vector must pass through the horizontal edge to the left of $L_k$ after passing through the vertical edge to avoid violating Condition 4.
Thus if the vector begins at $R_{k-2}$, we have
\[a \geq h_{k-1} + h_k + h_{k-1} > 1 + h_1 + h_{i-1}\]
We deduce that Condition 1 is violated,
and we are left with vectors beginning at $R_{k-1}$.
Such a vector may only travel $2$ horizontal distances without violating Condition 1, so it must end at either $L_k$ or $R_{k-1}$ upon passing through the horizontal edge to the left of $L_k$.
If it were to end at $L_k$ it would have slope greater than 1, violating Condition 2.
The only remaining option, the vector ending at $R_{k-1}$, is the same as the vector connecting $R_{k-2}$ and $R_k$ lying within the staircase, which we have previously eliminated.

 \textbf{Subcase 4.3.}  Now let us address the special cases of vectors passing through edges at the ends of the staircase.    
    
\textbf{Subsubcase 4.3.1.}  The vectors that pass through the edge between $L_0$ and $R_0$ will first be eliminated.
Consider the point where such a vector next intersects an edge of $S'$ or terminates at a vertex after passing through the aforementioned horizontal edge.
Such a vector that proceeds to intersect a horizontal edge to the left of some $L_k$ or terminate at some $L_k$ has a larger slope than $\nu_{k-1}$, meaning it violates Condition 2.
On the other, if such a vector proceeds to intersect a vertical edge below some $R_k$ or end at some $R_k$ for $k < i$, then it has larger $a$ value and smaller slope than $\sigma_{i-1}$, meaning it violates Condition 3.
This leaves vectors which, after passing through the horizontal edge between $L_0$ and $R_0$, intersect a vertical edge below some $R_k$ with $k \geq i$ or reach the opposite end of $S'$.
In this case, we have:
\[a \geq 1 + h_1 + h_k > 1 + h_1 + h_{i-1},\]
whence Condition 1 is violated.
    
\textbf{Subsubcase 4.3.2.}  When $n$ is even, we must consider the special case of saddle connections crossing the vertical edge below $L_{n/2}$. Let $k = n/2 - 1$. Notice that such a vector necessarily crosses $h_k > h_{i-1}$ at least twice. Such a vector cannot travel more than $2$ horizontal distances or else it would violate Condition 1. 


Thus, such a vector must either start at $R_{k - 1}$ or $L_{k}$. If the vector starts at $R_{k - 1}$, then it must be steeper than $\sigma_{k - 1}$, which is at least as steep as $\sigma_i$, meaning such vectors fail Condition 2. Now, suppose that the vector starts at $L_k$. Then the vector must end at $L_{k + 1}$ or $R_k$ after traversing $h_k$ twice (steeper vectors have slope greater than 1, violating Condition 2). In the first case, the slope of the vector is at most the aspect ratio, which means it violates Condition 4. In the second case, we have:
\[b/a = \frac{v_k + v_{k-1}}{2h_k} > \frac{v_{k-1}}{h_k}\]
Thus this vector has a higher slope than $\sigma_k$, which has slope at least as steep as $\sigma_i$, meaning Condition 2 is violated.
    
\textbf{Subsubcase 4.3.3.}  We must also consider a special case when $n$ is odd.


In this case, we must eliminate saddle connections crossing the horizontal edge to the right of $L_{(n - 1)/2}$.
Let $k = (n-3)/2$.
First of all, notice that any such vector crosses the horizontal distance $h_{k+1} > h_{i-1}$, meaning that any such vector which travels horizontally across at least 3 rectangles violates Condition 2.
Furthermore, vectors which start at $R_k$ have larger slope than $\sigma_k$ and thus fail Condition 2.
Thus we are left with vectors beginning at $L_k$ and $R_{k-1}$, which must end at $R_{k+1}$ to avoid violating Condition 1.
Both of these vectors have larger slope than $\sigma_{k+1}$, violating Condition 2. 
    


Finally, note that the vectors $\sigma_i$ do indeed have positive slope $v_{i-1}/h_i$ and is such that $0 < a + by \leq 1$:
\begin{align*}a + by = h_i + v_{i-1}y &> h_i + v_{i-1} \left(\frac{1 - h_{i-1}}{v_{i-2}}\right) \\[10pt] &= \frac{v_{i-2}v_{i-1} + v_{i-2}v_i + v_{i-1} - v_{i-2}v_{i-1} - v_{i-1}^2}{v_{i-2}} \\[10pt] &= \frac{v_{i-1} - 1}{v_{i-2}} > 0.\end{align*}
On the other hand, 
\begin{align*}a + by = h_i + v_{i-1}y \leq h_i + v_{i-1}\left(\frac{1 - h_i}{v_{i-1}}\right) = 1\end{align*}
Thus we have shown that for $1 < i \leq \lceil n/2 \rceil - 1$, $\sigma_i$ wins on $P_{n-i+1} \cap \{x = 1\}$.

%
%
%
%

\textbf{Case 5.} We finally show that $\sigma_1$ wins on $P_n \cap \{x=1\}$.
    
We first establish bounds on the horizontal length of a candidate vector $\cvec{a}{b}$ if it were to win over $\sigma_1$. In order to satisfy Condition 3, we must have that $\frac{1 - a}{b} > -h_1,$ or equivalently, $a < 1 + h_1b$. Condition 2 stipulates that $a/b > h_1$, so \[h_1b < a < 1 + h_1b.\] 
    
For any $\cvec{a}{b}$ satisfying the above conditions, the endpoint of the saddle connection must be a distance of between $0$ and $1$ to the right of a line of slope $1/h_1$ passing through the starting vertex of $\cvec{a}{b}$. So let us consider the lines of slope $1/h_1$ passing through each vertex of the staircase. The line of slope $1/h_1$ beginning at $L_i$ travels within the staircase to $R_{i + 2}$, as using trigonometric identities gives us that \[\frac{(h_{i - 1} + h_i + h_{i + 1})}{h_i} = h_1.\]
    
This equation also shows that the line of slope $1/h_1$ beginning at $R_i$ with $i > 0$ passes through the vertical edge below $R_{i + 1}$, then through the horizontal edge to the left of $L_{i + 1}$, then through the vertical edge below $R_i$, before intersecting $L_i$, as this saddle connection has the same holonomy vector (and thus slope) as the one connecting $L_i$ and $R_{i + 2}$.
For $i = 0$, the line of slope $1/h_1$ beginning at $R_i$ simply connects directly to $R_1$.

Now we observe that there are no vertices with horizontal distance between $0$ and $1$ to the right of the lines of slope $1/h_1$ described above. The vertex $R_{i + 1}$ is horizontally at least distance $1$ to the right of the line from $L_i$. We show that 
\[v_ih_1 + 1 \leq h_i + h_{i + 1}. \] 
We first compute that 
\[h_i+h_{i+1} - v_ih_1 - 1 = \csc \left(\frac{\pi }{n}\right) \sin \left(\frac{\pi  (i+1)}{n}\right)-1.\]
Notice that for fixed $n$, when $i = 0$, this expression evaluates to $0$. Moreover, the partial derivative with respect to $i$ is \[\frac{\pi  \csc \left(\frac{\pi }{n}\right) \cos \left(\frac{\pi  (i+1)}{n}\right)}{n},\] which is strictly positive for $i \leq n/2 - 1$ (which is indeed true for our purposes).

Therefore, \[h_i+h_{i+1} - v_ih_1 - 1 \geq 0\] for all relevant $i$, meaning that \[v_ih_1 + 1 \leq h_i + h_{i + 1}. \] 
    
This calculation also shows the vertex $L_{i + 1}$ is horizontally at least distance $1$ to the right of the line from $R_i$.
All other vertices are horizontally at least some distance $h_j \geq 1$ to the right of the lines from each $L_i$ or $R_i$ described above.
    
The only special case we have failed to consider above is when the lines constructed above pass through the upper edge of the staircase shape.
For even $n$, the following calculation shows that the line of slope $1/h_1$ from $L_{n/2 - 2}$ passes through the vertical edge below $L_{n/2}$ once before intersecting $L_{n/2}$, as through trigonometric identities, we can verify that \[h_1(v_{n/2 - 2} + v_{n/2 - 1}) = h_{n/2 - 2} + 2h_{n/2 - 1}.\]
    
The above calculation also shows that the line of slope $1/h_1$ through $L_{n/2 - 1}$ passes through the vertical face below $L_{n/2}$, then passes through the horizontal face to the left of $L_{n/2}$, then passes through the vertical face below $R_{n/2 - 1}$, before ending at $L_{n/2 - 1}$, because this saddle connection has the same holonomy vector as the one described immediately before.
Again, notice that there are no vertices with horizontal distance between $0$ and $1$ to the right of the lines described above.
By the following calculation, the vertex $L_{n/2}$ is horizontally at least distance $1$ to the right of the line from $L_{n/2 - 1}$.
We show that \[h_1v_{n/2 - 1} + 1 \leq 2h_{n/2 - 1}\] by proving the equivalent inequality \[ 2h_{n/2 - 1} - h_1v_{n/2 - 1} - 1 \geq 0.\] Using the definitions of $h_i$ and $v_i$, we have \[2h_{n/2 - 1} - h_1v_{n/2 - 1} - 1 = \frac{1}{2} \left(\tan \left(\frac{\pi }{2 n}\right)+\cot \left(\frac{\pi }{2n}\right)-2\right).\]

Note that we must have $n > 2$. Evaluating the above expression at $n = 2$ gives us \[2h_{2/2 - 1} - h_1v_{2/2 - 1} - 1 = 0.\] Moreover, the derivative of this expression with respect to $n$ is \[\frac{\pi  \left(\csc ^2\left(\frac{\pi }{2 n}\right)-\sec ^2\left(\frac{\pi }{2n}\right)\right)}{4 n^2},\]
and since $\csc\left(\frac{\pi }{2 n}\right) \geq \sec\left(\frac{\pi }{2 n}\right)$ whenever $n \geq 2$, the derivative is positive for all $n \geq 2$, meaning that we can conclude \[ 2h_{n/2 - 1} - h_1v_{n/2 - 1} - 1 \geq 0,\] and hence the desired inequality follows.
All other vertices are horizontally at least some distance $h_j \geq 1$ to the right of the lines from each $L_i$ or $R_i$ described above, or have been addressed in the earlier cases ignoring the edges at the upper end of $S'$.
    
For odd $n$, the following calculation shows that the line of slope $1/h_1$ from $L_{\lfloor n/2 \rfloor - 1}$ passes through the horizontal face to the right of $L_{\lfloor n/2 \rfloor}$, then intersects $L_{\lfloor n/2 \rfloor}$.
Using trigonometric identities, we can verify that 
\[2h_1v_{\lfloor n/2 \rfloor - 1} = 2h_{\lfloor n/2 \rfloor - 1} + h_{\lfloor n/2 \rfloor}.\]

As before, all other vertices are horizontally at least some distance $h_j \geq 1$ to the right of the line from each $L_i$ or $R_i$, or have been addressed in the earlier cases in which the edges at the upper end of $S'$ were ignored.
    
Thus, for every vertex in the staircase, there is no $\cvec{a}{b}$ beginning at that vertex and ending at another vertex such that $h_1b < a < 1 + h_1b$.
We conclude that $\sigma_1$ wins in the region of $\Omega_1$ with $x = 1$ and
$-h_1 < y \leq 1 - h_1$.

\textbf{Finishing Touches.} Our case work gives us that the $\lambda_i$ vectors for $1 \leq i \leq n$ are exactly the set of saddle connection vectors which win at points $(1, y) \in \Omega_1$.
By Lemma \ref{check_on_x=1}, the $\lambda_i$ vectors are therefore the complete set of vectors which win at any point in $\Omega_1$.

Now, in order to completely describe the return time on $\Omega_1$, we must simply determine which $\lambda_i$ wins over the other $\lambda_j$ vectors at each point in $\Omega_1$.

By Lemma \ref{lem:order}, we know that the $\lambda_i$ vectors are in order of decreasing slope, so the vector which wins at $(x, y) \in \Omega_1$ is $\lambda_i$ with maximal $i$ such that $ax + by \leq 1$.
Well, $\lambda_n$ has the maximal $i$ value among the all $\lambda_i$ vectors, so $\lambda_n = \cvec{h_1}{1}$ wins at every $(x, y) \in \Omega_1$ with $h_1x + y \leq 1$; this is exactly the region $P_n$.
Thus the return time at $(x, y)$ will be the slope of $M_{x, y}\lambda_n$, which is
\[\frac{v_0}{x(h_1x + v_0y)}.\]
For $\lfloor n/2 \rfloor + 2 \leq i < n$, $\lambda_i = \sigma_{n-i+1}$ wins at every point $(x, y) \in \Omega_1$ with $h_{n-i+1}x + yv_{n-i} \leq 1$ where $\lambda_j$ for $j>i$ does not satisfy the same condition, or equivalently, $h_{n-j+1}x + yv_{n-j} > 1$ for $i < j \leq n$.
For $j < k < n$ and for all $(x, y) \in \Omega_1$, if $h_{n-j+1}x + yv_{n-j} > 1$ and $h_1x + y > 1$ then $h_{n-k+1}x + yv_{n-k} > 1$ as well, meaning that the above conditions can be simplified to $h_{n-i}x + yv_{n-i-1} > 1$ and $h_1x + y > 1$, which along with the aforementioned condition $h_{n-i+1}x + yv_{n-i} \leq 1$ exactly defines $P_i$.
Thus the return time at $(x, y)$ will be the slope of $M_{x, y}\lambda_i$, which is
\[\frac{v_{n-i}}{x(h_{n+1-i}x + v_{n-i}y)}.\]
For $i = \lfloor n/2 \rfloor + 1$, $\lambda_i = \nu_{\lfloor n/2 \rfloor - 1}$ wins at every point $(x, y) \in \Omega_1$ with $h_{\lfloor n/2 \rfloor - 1}x + yv_{\lfloor n/2 \rfloor - 1} \leq 1$ where $\lambda_j$ for $j>i$ does not satisfy the same condition, or equivalently, $h_jx + yv_{j - 1} > 1$ for all $1 \leq j < \lceil n/2 \rceil$.
As noted earlier, this condition can be simplified to $h_{\lceil n/2 \rceil - 1}x + yv_{\lceil n/2 \rceil - 2} > 1$ and $h_1x + y > 1$, which along with the aforementioned condition $h_{\lfloor n/2 \rfloor - 1}x + yv_{\lfloor n/2 \rfloor - 1} \leq 1$ exactly defines $P_i$.
Thus the return time for $(x, y) \in P_i$ will be the slope of $M_{x, y}\lambda_i$, which is
\[\frac{v_{\lfloor n/2 \rfloor - 1}}{x(h_{\lfloor n/2 \rfloor - 1}x + v_{\lfloor n/2 \rfloor - 1}y)}.\]
For $1 < i \leq \lfloor n/2 \rfloor$, $\lambda_i = \nu_{i - 2}$ wins at every point $(x, y) \in \Omega_1$ with $h_{i - 2}x + yv_{i - 2} \leq 1$ where $\lambda_j$ for $j>i$ does not satisfy the same condition, or equivalently, $h_jx + yv_j > 1$ for all $j > i$, $h_jx + yv_{j - 1} > 1$ for all $1 \leq j < \lceil n/2 \rceil$.
For $j < k$ and for all $(x, y) \in \Omega_1$, if $h_jx + yv_j > 1$ and $h_1x + y > 1$ then $h_kx + yv_k > 1$ as well, and if $h_jx + yv_j > 1$ for any $j$ then $h_{k + 1}x + yv_k > 1$ for all $k$, meaning that the above conditions can be simplified to $h_{i - 2}x + yv_{i - 2} \leq 1$, $h_{i - 1}x + v_{i - 1}y > 1$, and $h_1x + y > 1$.
These conditions together exactly describe the region $P_i$.
Thus the return time for $(x, y) \in P_i$ will be the slope of $M_{x, y}\lambda_i$, which is
\[\frac{v_{i - 2}}{x(h_{i - 2}x + v_{i - 2}y)}.\]
Now the only $\lambda$ vector that remains is $\lambda_1 = \cvec{0}{1}$.
Since this vector has higher slope than all other $\lambda$ vectors, it can only win within the region $\Omega_1 \setminus\bigcup_{i > 1} P_i$, which is exactly $P_1$.
$\lambda_1$ wins throughout $P_1$ because $0 \cdot x + 1 \cdot y = y \leq 1$ at every point in $P_1$.
Thus the return time for $(x, y) \in P_1$ will be the slope of $M_{x, y}\lambda_1$, which is
\[\frac{1}{xy}.\]

\end{proof}

\end{appendix}

\bibliography{gapsreferences}

\begin{thebibliography}{10}

\bibitem{Ath13}
{\sc Athreya, J.~S.}
\newblock Gap distributions and homogeneous dynamics.
\newblock In {\em Geometry, topology, and dynamics in negative curvature},
  vol.~425 of {\em London Math. Soc. Lecture Note Ser.} Cambridge Univ. Press,
  Cambridge, 2016, pp.~1--31.

\bibitem{AthCha}
{\sc Athreya, J.~S., and Chaika, J.}
\newblock The distribution of gaps for saddle connection directions.
\newblock {\em Geom. Funct. Anal. 22}, 6 (2012), 1491--1516.

\bibitem{ACL}
{\sc Athreya, J.~S., Chaika, J., and Leli{\`e}vre, S.}
\newblock The gap distribution of slopes on the golden {L}.
\newblock In {\em Recent trends in ergodic theory and dynamical systems},
  vol.~631 of {\em Contemp. Math.} Amer. Math. Soc., Providence, RI, 2015,
  pp.~47--62.

\bibitem{AC13}
{\sc Athreya, J.~S., and Cheung, Y.}
\newblock A {P}oincar\'e section for the horocycle flow on the space of
  lattices.
\newblock {\em Int. Math. Res. Not. IMRN}, 10 (2014), 2643--2690.

\bibitem{DS84}
{\sc Dani, S.~G., and Smillie, J.}
\newblock Uniform distribution of horocycle orbits for {F}uchsian groups.
\newblock {\em Duke Mathematical Journal 51}, 1 (1984), 184--194.

\bibitem{EskinMasur}
{\sc Eskin, A., and Masur, H.}
\newblock Asymptotic formulas on flat surfaces.
\newblock {\em Ergodic Theory and Dynamical Systems 21}, 2 (2001), 443--478.

\bibitem{Heersink}
{\sc Heersink, B.}
\newblock Poincar\'{e} sections for the horocycle flow in covers of {${\rm
  SL}(2,\Bbb R)/{\rm SL}(2,\Bbb Z)$} and applications to {F}arey fraction
  statistics.
\newblock {\em Monatsh. Math. 179}, 3 (2016), 389--420.

\bibitem{HSinvariants}
{\sc Hubert, P., and Schmidt, T.~A.}
\newblock Invariants of translation surfaces.
\newblock {\em Ann. Inst. Fourier (Grenoble) 51}, 2 (2001), 461--495.

\bibitem{HS}
{\sc Hubert, P., and Schmidt, T.~A.}
\newblock An introduction to {V}eech surfaces.
\newblock In {\em Handbook of dynamical systems. {V}ol. 1{B}}. Elsevier B. V.,
  Amsterdam, 2006, pp.~501--526.

\bibitem{KumanduriWang}
{\sc Kumanduri, L., Sanchez, A., and Wang, J.}
\newblock Slope gap distributions of {V}eech surfaces.
\newblock {P}reprint,
  \href{http://arxiv.org/abs/2102.10069}{arXiv:math/2102.10069}.

\bibitem{MarkStrom}
{\sc Marklof, J., and Str{\"o}mbergsson, A.}
\newblock The distribution of free path lengths in the periodic lorentz gas and
  related lattice point problems.
\newblock {\em Annals of Mathematics\/} (2010), 1949--2033.

\bibitem{Masur}
{\sc Masur, H.}
\newblock Ergodic theory of translation surfaces.
\newblock In {\em Handbook of dynamical systems. {V}ol. 1{B}}. Elsevier B. V.,
  Amsterdam, 2006, pp.~527--547.

\bibitem{sanchez}
{\sc Sanchez, A.}
\newblock Gaps of saddle connection directions for some branched covers of
  tori.
\newblock to appear, ETDS, 2021.

\bibitem{S81}
{\sc Sarnak, P.}
\newblock Asymptotic behavior of periodic orbits of the horocycle flow and
  eisenstein series.
\newblock {\em Communications on Pure and Applied Mathematics 34\/} (1981),
  719--739.

\bibitem{SU}
{\sc Smillie, J., and Ulcigrai, C.}
\newblock Beyond {S}turmian sequences: coding linear trajectories in the
  regular octagon.
\newblock {\em Proc. Lond. Math. Soc. (3) 102}, 2 (2011), 291--340.

\bibitem{taha2}
{\sc Taha, D.}
\newblock The {B}oca-{C}obeli-{Z}aharescu map analogue for the {H}ecke triangle
  groups {$G_q$}.
\newblock {P}reprint,
  \href{http://arxiv.org/abs/1810.10668}{arXiv:math/1810.10668}.

\bibitem{taha}
{\sc Taha, D.}
\newblock On cross sections to the geodesic and horocycle flows on quotients of
  $\operatorname{SL}(2, \mathbb{R})$ by hecke triangle groups {$G_q$}.
\newblock {P}reprint,
  \href{http://arxiv.org/abs/1906.07250}{arXiv:math/1906.07250}.

\bibitem{UW}
{\sc Uyanik, C., and Work, G.}
\newblock The distribution of gaps for saddle connections on the octagon.
\newblock {\em Int. Math. Res. Not. IMRN}, 18 (2016), 5569--5602.

\bibitem{V89}
{\sc Veech, W.~A.}
\newblock Teichm{\"u}ller curves in moduli space, {E}isentein series and an
  application to triangular billiards.
\newblock {\em Inventiones mathematicae 97\/} (1989), 553--583.

\bibitem{Vee95}
{\sc Veech, W.~A.}
\newblock Geometric realizations of hyperelliptic curves.
\newblock {\em Algorithms, fractals, and dynamics\/} (1995), 217--226.

\bibitem{Vee98}
{\sc Veech, W.~A.}
\newblock Siegel measures.
\newblock {\em Annals of Mathematics 148\/} (1998), 865--944.

\bibitem{Vorobets}
{\sc Vorobets, Y.~B.}
\newblock Plane structures and billiards in rational polygons: the {V}eech
  alternative.
\newblock {\em Uspekhi Mat. Nauk 51}, 5(311) (1996), 3--42.

\bibitem{Work}
{\sc Work, G.}
\newblock A transversal for horocycle flow on {$\mathcal{H}(\alpha)$}.
\newblock {\em Geom. Dedicata 205\/} (2020), 21--49.

\bibitem{Zorich}
{\sc Zorich, A.}
\newblock Flat surfaces.
\newblock {\em Frontiers in number theory, physics, and geometry I\/} (2006),
  439--585.

\end{thebibliography}
\bibliographystyle{acm}

\end{document}